\documentclass[11pt]{amsart}
\usepackage{amsmath}
\usepackage{amsfonts}
\usepackage{amssymb}
\usepackage{latexsym}
\usepackage{mathrsfs}
\usepackage{enumerate}
\usepackage{color}
\usepackage[all]{xy}
\usepackage{hyperref}
\usepackage[OMLmathsfit]{isomath}
\usepackage[width=5.8in, height=8.5in, bottom=1.3in, centering]{geometry}

\newtheorem{theorem}{Theorem}[section]
\newtheorem{lemma}[theorem]{Lemma}
\newtheorem{proposition}[theorem]{Proposition}

\newtheorem{corollary}[theorem]{Corollary}

\theoremstyle{definition}
\newtheorem{definition}[theorem]{Definition}
\newtheorem{remark}[theorem]{Remark}
\newtheorem{example}[theorem]{Example}

\numberwithin{equation}{section}

\newcommand{\C}{\mathbb{C}}

\newcommand{\N}{\mathbb{N}}

\newcommand{\R}{\mathbb{R}}

\newcommand{\T}{\mathbb{T}}
\newcommand{\Z}{\mathbb{Z}}

\newcommand{\co}{\mskip0.5mu\colon\thinspace}

\newcommand{\inertia}[1]{{\Lambda}{#1}}
\newcommand{\inertianull}[1]{{\Lambda}_0{#1}}
\newcommand{\cartan}{\mathsf{T}}
\newcommand{\sphere}{\mathsf{S}}

\newcommand{\calC}{\mathcal{C}}
\newcommand{\calO}{\mathcal{O}}
\newcommand{\calI}{\mathcal{I}}
\newcommand{\calS}{\mathcal{S}}
\newcommand{\sfA}{\mathsf{A}}
\newcommand{\sfB}{\mathsf{B}}
\newcommand{\sfG}{\mathsf{G}}
\newcommand{\sfH}{\mathsf{H}}
\newcommand{\mfrak}{\mathfrak{m}}

\newcommand{\id}{\operatorname{id}}
\newcommand{\Sat}{\operatorname{Sat}}

\newcommand{\sttimes}{{\hspace{0.1em}{_s}\hspace{-.1em}\times_{\hspace{-.05em}t}\hspace{0.1em}}}
\newcommand{\fgtimes}[2]{{\hspace{0.1em}{_{#1}}\hspace{-.1em}\times_{\hspace{-.05em}{#2}}\hspace{0.1em}}}

\newcommand{\arr}[1]{\mathsfit{#1}}
\newcommand{\obj}[1]{\mathsfit{#1}}

\newcommand{\zartan}{T}
\newcommand{\stratan}{T^\textup{st}}
\newcommand{\centralizer}{\mathsf{Z}}
\newcommand{\normalizer}{\mathsf{N}}
\newcommand{\im}{\operatorname{im}}

\begin{document}

\title[Differentiable stratified groupoids]{Differentiable stratified groupoids and a de Rham theorem for inertia spaces}

\author{Carla Farsi}
\address{Carla Farsi, {\rm Department of Mathematics, University of Colorado at Boulder, Campus
Box 395, Boulder, CO 80309-0395, USA}}
\email{farsi@euclid.colorado.edu}

\author{Markus J. Pflaum}
\address{Markus J. Pflaum, {\rm Department of Mathematics, University of Colorado at Boulder, Campus
Box 395, Boulder, CO 80309-0395, USA and \newline
\indent Max-Planck-Institut f\"ur Mathematik in den Naturwissenschaften, Inselstr.~22, 04103 Leipzig, Germany}}
\email{markus.pflaum@colorado.edu }

\author{Christopher Seaton}
\address{Christopher Seaton, {\rm Department of Mathematics and Computer Science,
Rhodes College, 2000 N. Parkway, Memphis, TN 38112, USA}}
\email{seatonc@rhodes.edu}

\begin{abstract}
We introduce the notions of differentiable groupoids and differentiable stratified groupoids,
generalizations of Lie groupoids in which the  spaces of objects and arrows have the
structures of differentiable spaces, respectively differentiable stratified spaces,
compatible with the groupoid structure. After studying basic properties of these groupoids
including Morita equivalence, we prove a  de Rham theorem for proper
locally contractible differentiable stratified groupoids.
We then  focus on the study of the inertia groupoid associated to a proper Lie groupoid.
We show that the loop and the inertia space of a proper Lie groupoid
can be endowed with a natural Whitney (b)-regular stratification, which we call the
orbit Cartan type stratification. Endowed with this stratification,
the inertia groupoid of a proper Lie groupoid becomes a locally contractible differentiable
stratified groupoid.
\end{abstract}

\maketitle

\tableofcontents


\section{Introduction}
\label{sec:Intro}

The theory of Lie group actions on smooth manifolds is fundamental for several areas in mathematics and
has a long tradition. But examples of natural Lie group actions do not only comprise actions on smooth manifolds but
also on singular spaces such as orbifolds \cite{LermanTolman,Schmah,Vergne} or manifolds with boundary or corners
\cite{MargalefOuterelo, MelrosePDO}. Lie group actions on singular spaces arise also in singular symplectic
reduction \cite{SjamaarLerman} and the transverse cotangent bundle of a $G$-manifold \cite{DeConcProcesiVergne,ParadanVergneIndex}.
In these cases, the corresponding translation groupoid, while not a
Lie groupoid, has significantly more structure than merely that of a topological groupoid. In particular,
the spaces of objects and arrows inherit differentiable stratified space structures compatible with the
groupoid structure maps. One of the main goals of this paper is to define and study the category of such groupoids
and their singular structures.

In this paper, we introduce
\emph{differentiable stratified groupoids}, a class of groupoids within the category of differentiable
stratified spaces that share many of the properties of Lie groupoids \cite{CrainicMestreOrbispace}.
The definition is designed so that, under mild hypotheses, the restriction of a
differentiable stratified groupoid to a stratum of the orbit space is a Lie groupoid.
The definitions of a differentiable groupoid  and of a
differentiable stratified groupoid we propose are deliberately broad so as to be easily verified and
to include many interesting cases. Examples of such groupoids arise as restrictions or quotients
of Lie groupoid actions or in other contexts from a related Lie groupoid. Often more specific properties are inherited
from the context. To take advantage of these we define the subclasses of sliceable differentiable
(stratified) groupoids in Definitions \ref{def:LocTransDiffGroupoid} and \ref{def:LocTransStrat}
and introduce a local contractibility condition in Definition \ref{def:local-contractibility}.


One major motivating example is that of a compact Lie group acting differentiably on a
differentiable stratified space in such a way that the strata are permuted by the action.
We will show that the corresponding translation
groupoid is a differentiable stratified groupoid.
We study the properties of these groupoids and their object spaces including the
appropriate notion of Morita equivalence. Moreover, we prove a de Rham theorem which
relates the singular cohomology of the orbit space of a sliceable differentiable stratified groupoid
fulfilling a local contractibility condition to the cohomology of basic differential forms on the object
space; see Theorem \ref{thrm:deRhamDiffStrat}. To prove this theorem, we first localize
to groupoid charts and then prove a Poincar\'e lemma for basic forms in the context of sliceable
differentiable stratified groupoids.

A particularly important example which we consider is that of the \emph{inertia groupoid}
of a proper Lie groupoid $\sfG$. The inertia  groupoid is presented as the translation groupoid of
$\sfG$ acting on the so-called \emph{loop space} $\inertianull{\sfG}\subset \sfG$ which consist of all
arrows having the same source and target.
When $\sfG$ is an orbifold groupoid, the loop space $\inertianull{\sfG}$ is a smooth manifold so that
the inertia groupoid presents an orbifold, the \emph{inertia orbifold}. The inertia orbifold plays
an import role in the geometry and index theory of orbifolds; see \cite{AdemLeidaRuan,PflPosTanAITO}. However,
when $\sfG$ is not an orbifold groupoid, the loop space $\inertianull{\sfG}$ becomes singular,
and very little is known about its singularity structure in general. One of the goals of this paper is to deepen the
understanding of inertia groupoids by exhibiting  explicit stratifications for them.
Motivation for such a study comes from the abundance and diversity where
differentiable stratified groupoids appear, often in a somewhat hidden way. 
In addition to the mentioned orbifolds, loop spaces and their associated inertia groupoids
appear naturally in geometry e.g.\ as commuting varieties when the adjoint action of a
Lie group $G$ on itself is considered \cite{AdemCohenCommutElts,PanyushevYakimova,RichardsonCommutVar}.
The loop space corresponds then to the set of pairs of commuting elements in $G$. Similarly,
spaces of (conjugacy classes of) commuting $m$-tuples of elements of $G$ as studied e.g.\ in
\cite{AdemGomezEqKThry,GomezPettet,TorresSjerveCommutNTuple} provide related examples of
differentiable stratified groupoids.
Important for the understanding of moduli spaces of flat connections \cite{Goldman1984} and for representation
theory are sets of conjugacy classes of homomorphisms $\pi\to G$, where $\pi$ is a finitely generated discrete
group and $G$ a compact connected Lie group. Such spaces can also be naturally interpreted as differentiable stratified groupoids. 

When $\sfG$  is the translation groupoid $G \ltimes M$ associated to the smooth action of
a compact Lie group $G$ on a manifold $M$, the orbit space of the $G$-action on the corresponding loop space
 - here this orbit space is called the \emph{inertia space} - was first considered in \cite{BrylinskiCycHomEquivar}.
By virtue of its structure as a semialgebraic $G$-set, the loop space is a Whitney (b)-regular
stratified space, see Example \ref{ex-semialgebraic}, but the explicit structure of this stratification is not well understood.
In the previous paper \cite{FarsiPflaumSeaton}, we introduced an explicit Whitney (b)-regular stratification  of the loop space
of $G\ltimes M$ and its associated orbit space. Applying this stratification to local charts of the form $G\ltimes M$ for a proper
Lie groupoid $\sfG$ does in general not yield a well-defined global
stratification of the loop space of $\sfG$. Specifically, the local stratifications from that paper
may not coincide on intersections of charts.
Here, we will give a significant modification of the stratification in \cite{FarsiPflaumSeaton} that can be patched together
on charts to yield a well-defined stratification of the loop and inertia spaces of an arbitrary proper Lie groupoid. We call
this stratification the \emph{orbit Cartan type stratification}. For the translation groupoid of a compact group action
the orbit Cartan type stratification presented here is in general coarser than the one considered
in our previous paper  \cite{FarsiPflaumSeaton}.
Because the construction is local, we carry it out on a
single groupoid chart, which, up to Morita equivalence, is given by the translation groupoid
associated to a finite-dimensional $G$-representation $V$ where $G$ is a proper Lie groupoid, see Theorem
\ref{thrm:InertiaStratGMnfld}. For this case, the main ideas of our construction rely on providing
a stratification of  slices, see Section \ref{subsec:InertiaGpoidStrat} for details. To achieve this,
we first decompose the Cartan subgroups of the stabilizers into equivalence classes determined by
their fixed point sets in $V$, which motivates the name of the stratification. We then use these
decompositions together with the fixed point sets of the stabilizers to stratify the slice and take
saturations to stratify the loop and orbit spaces. In Section \ref{subsec:GonMStratProof} we
demonstrate that our orbit Cartan type decomposition indeed satisfies all the properties of
a stratification and turns the loop and the inertia space into differentiable stratified spaces; see Theorem \ref{thrm:InertiaStrat}.
Moreover, in Section \ref{subsec:Whitney}, we show that the orbit Cartan type stratification
is Whitney (b)-regular. That this stratification is given explicitly allows us finally to verify in Theorem \ref{thrm:InertiaDeRham} that
the local contractibility condition of Definition \ref{def:local-contractibility} is fulfilled for
the inertia space, hence the de Rham theorem \ref{thrm:deRhamDiffStrat} holds in this case.

It appears that the stratification and the de Rham theorem for the inertia space could provide a
decisive tool to verify Brylinski's claim \cite[Prop., p.~25]{BrylinskiCycHomEquivar}, which is not
fully proven in his paper, that for a given transformation groupoid $G\ltimes M$ the complex of
basic differential forms on the loop space is quasi-isomorphic to the Hochschild chain complex of
the convolution algebra on $G\ltimes M$. Our results from \cite{FarsiPflaumSeaton} have been used
in a proof of Brylinski's conjecture for the case of an  $S^1$-action, see
\cite{PflPosTanGrauGroth,PflPosTanHochschild}.
Work on the general case is in progress.

The outline of this paper is as follows.
In Section \ref{sec:StratDiffGpoids} we introduce the fundamental concepts of differentiable (stratified)
groupoids we will rely on throughout this paper, demonstrate several basic properties, and
consider morphisms of differentiable (stratified) groupoids and Morita equivalence.
In Section \ref{sec:StratDiffGpoidExamples} we present several examples of differentiable stratified
groupoids and continue in Section \ref{sec:algebroid} with the definition of the algebroid of a
differentiable stratified groupoid assumed to fulfill some additional technical hypotheses.
The algebroid is then used in Section \ref{sec:Invariants} to present the de Rham theorem for sliceable
differentiable stratified groupoids. In Section \ref{sec:InertiaGpoid}  we establish 
that the inertia space of a proper Lie groupoid admits a locally contractible Whitney (b)-regular
stratification, the \emph{orbit Cartan type stratification}. The result is stated in
Theorem  \ref{thrm:InertiaStrat} and relies on  Theorem \ref{thrm:InertiaStratGMnfld} which provides the
result for the local case, that is, for linear actions of compact Lie groups. In
Proposition \ref{prop:LocalStratWhitneyB} we  demonstrate Whitney (b)-regularity for the
orbit Cartan type stratification. The appendix contains a recollection of the the basics of differentiable
spaces and describes the category of differentiable stratified spaces used throughout this paper.

\subsection*{Acknowledgments}
C.F.~would like to thank the Harish-Chandra Research Institute (Allahabad, India) and the
DST Center for Interdisciplinary Mathematical Sciences of Banares Hindu University (Varanasi, India) for hospitality
during work on this manuscript.
Moreover, C.F.~greatfully acknowledges support by a Simons Foundation Collaboration Grant under award number 4523991.
C.S.~was supported by a Rhodes College Faculty Development Grant, the E.C.~Ellett Professorship in Mathematics, and
the Meyers fund. In addition, C.S.~would like to thank the University of Colorado at Boulder for hospitality during
work on this manuscript.
M.J.P.~thanks the Max-Planck Institute for Mathematics in Bonn and the Max-Planck Institute for the Mathematics of the Sciences in Leipzig for hospitality and support.
M.J.P.~also acknowledges NSF support under contract DMS 1105670 and support by a
Simons Foundation Collaboration Grant under award number 359389.


\section{Fundamentals}
\label{sec:StratDiffGpoids}

In this section, we give the definitions and basic properties of the groupoids under consideration.
See Appendices \ref{ap:DiffSpaces} and \ref{ap:StratSpaces} for a review of the notions
of differentiable and differentiable stratified spaces used in this paper.

\subsection{Topological groupoids}
Recall that by a groupoid $\sfG$  one understands a small category
with object set $\sfG_0$ and arrow set $\sfG_1$ such that all arrows are invertible. We use
$s, t\co\sfG_1\to\sfG_0$ to denote the source and target maps, respectively, write $u\co \sfG_0\to\sfG_1$
for the unit map, $i\co\sfG_1\to\sfG_1$ for the inverse map, and
finally denote by $m\co\sfG_1\sttimes\sfG_1\to\sfG_1$ the multiplication or composition map.
The maps $s$, $t$, $u$, $i$, and $m$ are collectively referred to as the \emph{structure maps} of the groupoid.
If $\sfG$ and $\sfH$ are groupoids, a \emph{morphism} $f\co\sfG\to\sfH$ is a functor,
i.e.~a pair of functions $f_0\co\sfG_0\to\sfH_0$ and $f_1\co\sfG_1\to\sfH_1$ that commute
with each of the structure maps.

The \emph{orbit} through a point $\obj{x}\in \sfG_0$ is defined as the set  of all $\obj{y} \in \sfG_0$
for which there exists a $\arr{g}\in \sfG_1$   such that $s(\arr{g})=\obj{x}$ and
$t(\arr{g})=\obj{y}$. It is denoted by $\sfG \obj{x}$.
Obviously, the object space is partitioned into orbits. We denote the set of orbits of a groupoid $\sfG$
by $|\sfG|$, and the canonical projection from the object to the orbit space by $\pi\co \sfG_0 \to |\sfG|$.

A groupoid $\sfG$ is called a \emph{topological groupoid} if $\sfG_0$ and $\sfG_1$ are
topological spaces and each of the structure maps are continuous. This implies in particularly that
the unit map is a homeomorphism onto its image. If the topological groupoid $\sfG$ is
\emph{Hausdorff}, which means that $\sfG_1$ is Hausdorff, the image $u(\sfG_0)$ is closed in $\sfG_1$,
cf.~\cite[Chap.~1 Sec.~2]{RenGACA}. The orbit space of a topological
groupoid $\sfG$ is always assumed to carry the quotient topology with respect to the canonical projection
$\pi\co \sfG_0 \to |\sfG |$. If $\sfG$ and $\sfH$ are topological groupoids and  $f\co\sfG\to\sfH$ a
morphism of groupoids, then $f$ is a \emph{morphism of topological groupoids} if in addition $f_0$ and $f_1$ are
continuous functions.

If the source map of a topological groupoid $\sfG$ is an open
map one says that $\sfG$ is an \emph{open topological groupoid}. Note that as $i$ is a homeomorphism
and $t = s\circ i$, the target map of an open topological groupoid is open as well. A topological groupoid  $\sfG$
for which the source map $s$ is a local homeomorphism is called an \emph{\'etale groupoid}.
$\sfG$ is called a \emph{quasi-proper groupoid} if $s\times t \co \sfG_1\times\sfG_1 \to \sfG_0$ is a proper
map, and a \emph{proper groupoid} if it is Hausdorff and quasi-proper. Finally, we say that a topological
groupoid $\sfG$ is \emph{compact} or \emph{locally compact} if the topological space $\sfG_1$
has the respective property.  Note that then $\sfG_0$ is  compact or locally compact, respectively, as well by
the proof of \cite[Prop.~2.5]{TuNonHausGpoid}.

\begin{remark}
 In this paper we follow the definitions of quasi-compact, compact, and locally compact spaces by
 Bourbaki \cite{BouGTC1-4}.
 That means, a topological space $X$ is called \emph{quasi-compact} if every open cover of $X$
 admits a finite subcover, \emph{compact} if $X$ is both Hausdorff and quasi-compact, and finally
 \emph{locally compact} if it is Hausdorff and every point $x\in X$ has a compact neighborhood.
 If each point of a topological space $X$ possesses a Hausdorff (respectively quasi-compact) neighborhood,
 we say that $X$ is \emph{locally Hausdorff} (respectively \emph{locally quasi-compact}).
 A continuous map
 $f\co X \to Y $ between not necessarily Hausdorff spaces $X$ and $Y$ is called \emph{proper}
 if $f\times \id_Z \co X \times Z \to Y \times Z$ is a closed map for every topological space $Z$,
 cf.~\cite[I.10.1.~Def.~1]{BouGTC1-4} or \cite[Sec.~1.3]{TuNonHausGpoid}.
 By \cite[I.10.2.~Thm.~1]{BouGTC1-4}, properness of $f$ is equivalent to the property that $f$ is
 closed and $f^{-1}(y)$ is quasi-compact for each $y\in Y$.
\end{remark}

\begin{proposition}
\label{prop:BasicPropTopGroupoids}
  Let $\sfG$ be a topological groupoid.
  Then the following holds true:
  \begin{enumerate}[{\rm (1)}]
  \item  If $\sfG$ is proper and the object space $\sfG_0$ is locally compact, then $\sfG$ is a
         locally compact groupoid.
  \item  \label{it:OpMap} The quotient map $\pi\co \sfG_0 \to |\sfG|$ is an open map of topological spaces
         if $\sfG$ is an open groupoid.
  \item  \label{it:HausdorffOrbitSp} The orbit space of $\sfG$ is Hausdorff if $\sfG$ is open,
         $\sfG_0$ is locally compact, and
         $(s,t)(\sfG_1)$ is closed in $\sfG_0 \times \sfG_0$. In particular, this is the case if
         $\sfG$ is a proper open groupoid.
  \item  \label{it:LocCpctOrbitSp} The orbit space of $\sfG$ is locally compact if
         $\sfG$ is a locally compact open groupoid
         and $(s,t)(\sfG_1)$ is locally closed in $\sfG_0 \times \sfG_0$.
  \end{enumerate}
\end{proposition}
\begin{proof}
  Let $\arr{g} \in \sfG_1$, $K$ be a compact neighborhood of $s(\arr{g})$, and $L$ be a compact neighborhood of $t(\arr{g})$.
  Then $(s,t)^{-1} (K\times L)$ is a compact neighborhood of $\arr{g}$. Since $\sfG_1$ is Hausdorff, the
  first claim is proved. Point (\ref{it:OpMap}) is an immediate consequence of \cite[Prop.~2.11]{TuNonHausGpoid}.
  The claims (\ref{it:HausdorffOrbitSp}) and (\ref{it:LocCpctOrbitSp}) follow from
  \cite[Prop.~2.12]{TuNonHausGpoid}.
\end{proof}

\subsection{Differentiable groupoids}

We now consider topological groupoids endowed with a differentiable structure as defined in
Appendix \ref{app:DiffStratSpaces}.

\begin{definition}
\label{def:DiffGpoid}
Let $\sfG$ be an open topological groupoid. We say that $\sfG$ is a \emph{differentiable groupoid}
if $\sfG_0$ and $\sfG_1$ are differentiable spaces and the structure maps
$s$, $t$, $i$, $u$, and $m$ are morphisms of differentiable spaces.  We say
that $\sfG$ is a \emph{reduced differentiable groupoid} if $\sfG_0$ and $\sfG_1$ are reduced differentiable spaces.
A morphism of topological groupoids  $f\co\sfG\to\sfH$ between differentiable groupoids $\sfG$ and $\sfH$ is called
a \emph{morphism of differentiable groupoids} if the functions $f_0$ and $f_1$ are both morphisms of
differentiable spaces.
\end{definition}

The requirement that the structure maps $s\co\sfG_1\to\sfG_0$ and $t\co\sfG_1\to\sfG_0$
are morphisms of the differentiable spaces $\sfG_0$ and $\sfG_1$ implies  by
\cite[Thm.~7.6]{NGonzalezSanchoBook} that the fibered product
$\sfG_1\sttimes\sfG_1$ inherits the structure of a differentiable space. It is with respect to this structure
that we require that $m\co\sfG_1\sttimes\sfG_1\to\sfG_1$ is a morphism
of differentiable spaces. As the inverse map $i\co\sfG_1\to\sfG_1$ is clearly invertible
with $i^{-1}=i$, it follows that $i$ is an isomorphism of differentiable spaces. Similarly,
$u\co\sfG_0\to\sfG_1$ is an embedding of differentiable spaces.

\begin{remark}
\label{rem:defDiffGroupoidTop}
If $\sfG$ is a differentiable groupoid, the underlying topological spaces
$\sfG_0$ and $\sfG_1$ need not be Hausdorff. However,
because both carry the structure of a differentiable space, $\sfG_0$ and $\sfG_1$ are locally Hausdorff and
locally quasi-compact, see \cite[Chap.~4]{NGonzalezSanchoBook}.
\end{remark}

\begin{example}
\label{ex:DiffGroupoids}
\begin{enumerate}[(a)]
\item
\label{ex:DiffTranslationGpoid}
Let $G$ be a compact Lie group and let $X$ be a Hausdorff differentiable
space with a differentiable $G$-action,
cf.~\cite[Sec.~11.2]{NGonzalezSanchoBook}.
Then the transformation groupoid $G\ltimes X$ is a differentiable groupoid with space of objects $X$ and space of arrows
$G \times X$, the latter being a differentiable space by \cite[Thm.~7.6]{NGonzalezSanchoBook}.
The source map $s\co G \times X \to X$ is the projection, hence smooth and open, and
the target map $t\co G \times X \to X$ is given by
$(g, x) \mapsto gx$, which is differentiable by the definition of a differentiable action.
The unit map $u\co x \mapsto (e, x)$ is easily seen to be a closed embedding since $\{ e \}$ is closed
in $G$.  The domain of the product map $m$ is a differentiable subspace of $G\times X \times G \times X$
by \cite[Thm.~7.6]{NGonzalezSanchoBook} and the proof of \cite[Rem.~7.5]{NGonzalezSanchoBook}.
Moreover, the map $G\times G \times X \to (G\times X) \sttimes (G \times X)$,
$(g,h,x) \mapsto \big((g, hx), (h, x)\big)$
is an isomorphism of differentiable spaces onto the domain of $m$.  Since the pullback of $m$
by this isomorphism is the smooth map $G \times G \times X \to G \times X$, $(g, h, x) \mapsto (gh, x)$,
multiplication in $G\ltimes X$  is smooth. Similarly one shows that the inverse map
$i \co G \times X \to G \times X$, $(g, x) \mapsto (g^{-1}, gx)$ is smooth.

By definition, the quotient space $X/G$ of the $G$-space $X$ coincides with the orbit space
of the groupoid $|G\ltimes X|$. Hence, if $X$ is Hausdorff, the orbit space becomes
a Hausdorff space by Prop.~\ref{prop:BasicPropTopGroupoids} and
even is a differentiable space by \cite[Thm.~11.17]{NGonzalezSanchoBook} when equipped with the sheaf
of $G$-invariant smooth functions on $X$ as the structure sheaf.
\item
  Let $X$ be a  Hausdorff  differentiable space. Then the pair groupoid $X\times X \rightrightarrows X$
  with source and target maps being projection onto the first and second factor, respectively, is a differentiable groupoid.
  It has only one orbit, hence the orbit space is naturally a differentiable space as well.
\item
In general, the orbit space $|\sfG |$ of a differentiable
groupoid $\sfG$ need not admit a differentiable structure with respect to which the quotient
map $\sfG_0 \to |\sfG |$ is a smooth map. As a well known example, consider the
action of $\Z$ on the circle $S^1$ by an irrational rotation. Then the translation groupoid
$\Z\ltimes S^1$ is a (non-quasi-proper) differentiable groupoid whose orbit space $|\Z\ltimes S^1|$
is not locally Hausdorff. Hence $|\Z\ltimes S^1|$ does not carry the structure of a differentiable space.
\end{enumerate}
\end{example}

Let $\sfG$ be a reduced differentiable groupoid and $Y$ a differentiable subspace of $\sfG_0$.
A \emph{bisection of $\sfG$ over $Y$} then is a continuous map $\sigma\co Y\to \sfG_1$ such that
$s\circ\sigma = \id_Y$, and such that $t\circ\sigma$ is an isomorphism of $Y$ onto the differentiable
subspace $t\circ\sigma(Y)$ of $\sfG_0$, cf.~\cite[Def.~1.4.8]{MackenzieGenThry}.
Note that we will consider bisections over sets $Y$ which may not be open in $\sfG_0$.

In the case that $\sfG$ is proper, the question of whether $|\sfG |$ admits a
differentiable structure remains open in general.
However, from Example \ref{ex:DiffGroupoids} (\ref{ex:DiffTranslationGpoid}), it is clear that the
quotient map of the action groupoid of a differentiable group action on a Hausdorff differentiable space
is a morphism of differentiable spaces. Many of the examples we consider will be of this form locally, which
will allow us to similarly conclude that $|\sfG |$ admits a compatible differentiable space structure;
see Proposition \ref{prop:LocTransDiffOrbitSpace}.
Because of their importance for our further considerations and their interesting topologial structure
we characterize such groupoids in Definition \ref{def:LocTransDiffGroupoid}, below.
To be able to formulate it, we need the following observation.

\begin{proposition}
Let $\sfG$ be a proper reduced differentiable groupoid. Then for every $\obj{x}\in \sfG_0$
the following holds true.
\begin{enumerate}[{\rm (LT1)}]
\setcounter{enumi}{-1}
\item \label{it:IsotropyLie}
    The isotropy group $\sfG_\obj{x}$ becomes a compact Lie group with
    the induced topological and differentiable structures.
\end{enumerate}
\end{proposition}
\begin{proof}
  The isotropy group $\sfG_\obj{x}$ is compact by properness of the groupoid. Hence it has finitely many connected components.
  By \cite[Cor.~4.2]{SpallekCTG},  $\sfG_\obj{x}$ becomes a Lie group with the induced reduced  differentiable structure.
\end{proof}

\begin{definition}
\label{def:LocTransDiffGroupoid}
We say that a proper reduced differentiable groupoid $\sfG$ is \emph{sliceable} if
the following conditions are satisfied for every $\obj{x} \in \sfG_0$.
\begin{enumerate}[{\rm (LT1)}]
\item \label{it:ExistenceSlice}
    Locally, the groupoid is isomorphic to the product of a translation groupoid with a pair groupoid.
    More precisely, there is an open Hausdorff neighborhood $U_\obj{x}\subset\sfG_0$ of $\obj{x}$ and a
    relatively closed  connected reduced differentiable subspace $Y_\obj{x} \subset U_\obj{x}$ containing $\obj{x}$
    together with a smooth $\sfG_\obj{x}$-action on $Y_\obj{x}$  such that $\obj{x}$ is a fixed point of the
    $\sfG_\obj{x}$-action and such that the restriction $\sfG_{|U_\obj{x}}$ is isomorphic as a differentiable groupoid to
    the product of the translation groupoid $G_\obj{x}\ltimes Y_\obj{x}$ and the pair groupoid
    $O_\obj{x} \times O_\obj{x}$, where $O_\obj{x}$ is an open neighborhood of $\obj{x}$ in its orbit.
    Such a neighborhood $U_\obj{x}$ is called a \emph{trivializing neighborhood} of $\obj{x}$, and
    $Y_\obj{x}$ a \emph{$\sfG$-slice} or a \emph{groupoid-slice} at $\obj{x}$.
\item \label{it:CompatibilitySlices}
    For each $\obj{z} \in Y_\obj{x}$, one may choose
    a trivializing neighborhood $U_\obj{z}$ of $\obj{z}$
    with corresponding $\sfG$-slice $Y_\obj{z}$
    such that $U_\obj{z} \subset U_\obj{x}$ and $Y_\obj{z} \subset Y_\obj{x}$.
\item \label{it:ExistenceBisections}
  The trivializing neighborhoods and $G$-slices of two objects
  $\obj{x}$ and $\obj{y}$ in the same orbit can be chosen to be compatible with
  translation by a smooth bisection in the following sense.
  For each arrow $\arr{g} \in \sfG$ with $s(\arr{g}) =\obj{x}$
  and $\obj{y} = t(\arr{g})$ there exist trivializing neigborhoods
  $U_\obj{x}$ of $\obj{x}$ and $U_\obj{y}$ of $\obj{y}$ and a smooth
  bisection  $\sigma$ over $U_\obj{x}$ such that
  $\sigma(\obj{x}) = \arr{g}$ and such that the morphism of differentiable
  groupoids $f \co \sfG_{|U_\obj{x}} \to \sfG_{|U_\obj{y}} $ with components
  $f_0 := t\circ\sigma\co U_\obj{x} \to U_\obj{y}$ and
  $f_1 \co \sfG_{|U_\obj{x},1} \to \sfG_{|U_\obj{y},1}$, $\arr{h} \mapsto \sigma (t(\arr{h}))  \arr{h}(\sigma (s(\arr{h})))^{-1} $
     is an isomorphism.
\end{enumerate}
%
\end{definition}

\begin{example}
  Every proper Lie groupoid is sliceable by e.g.~\cite[Thm.~2.3]{Zung06} or \cite[Cor.~3.11]{PflPosTanGOSPLG}.
  In Section \ref{sec:InertiaGpoid} we show that the inertia groupoid of a proper Lie groupoid is
  always sliceable even when it is singular; see Theorem \ref{thrm:InertiaDeRham}.
\end{example}

\begin{remark}
Clearly, a non-quasi-proper reduced differentiable groupoid may have
non-compact isotropy groups and hence fail to be sliceable;
the groupoid $\Z\ltimes\{pt\}$ is such an example. An interesting question which
we could not answer in this paper is whether a proper reduced differentiable groupoid is necessarily
sliceable. In other words, this question essentially asks whether there is a slice theorem for proper
reduced differentiable groupoids. We are not aware of a proper reduced differentiable groupoid which is not sliceable.

As noted above, a related open question is whether the orbit space of a proper (reduced)
differentiable groupoid inherits a (reduced) differentiable structure in such a way that the orbit map is a
morphism of differentiable spaces. This would follow in the reduced case from a slice theorem as above,
as it holds for sliceable differentiable groupoids by
Proposition \ref{prop:LocTransDiffOrbitSpace} below.
\end{remark}

\begin{remark}
\label{remark-equivariant-slice-chart}
By the following observation one can assume, possibly after shrinking, that a $\sfG$-slice $Y_\obj{x}$ of
$\obj{x}$ possesses a $\sfG_\obj{x}$-equivariant singular chart
$\iota \co Y_\obj{x} \hookrightarrow T_\obj{x}Y_\obj{x} \cong \R^{\operatorname{rk}\obj{x}}$ with $\iota
(\obj{x})=0$; see Appendix \ref{def:DiffSpace} for the definition of a singular chart
and Appendix \ref{SectionTangentSpace} for the definition of the Zariski tangent space $T_\obj{x}Y_\obj{x}$.
Note that $T_\obj{x}Y_\obj{x}$ carrries a natural  $\sfG_\obj{x}$-action inherited from the one on $Y$.
\end{remark}
\begin{proposition}
\label{prop:equivariant-embedding_zariski-space}
  Let $G$ be a compact Lie group, $Y$ a reduced differentiable space
  carrying a differentiable $G$-action, and $x\in Y$ a fixed point.
  Then the Zariski tangent space $T_xY$ inherits a natural  $G$-action
  from the $G$-action on $Y$.
  Moreover, there exists an open $G$-invariant neighborhood $W$ of $x$ in $Y$ and a $G$-equivariant singular chart
  $\iota: W \hookrightarrow T_xY$ mapping $x$ to the origin.
\end{proposition}
\begin{proof}
  By \cite[Prop.~1.3.10]{PflaumBook}  there exists a singular chart
  $\kappa : W \hookrightarrow \R^{\operatorname{rk} x} \cong T_xY$, where
  $\operatorname{rk} x = \operatorname{dim} T_xY$.
  Since $x$ is assumed to be a fixed point,
  one can achieve after possibly shrinking $W$ that $W$ is $G$-invariant.
  Now the proof is literally identical to the proof of
  \cite[Lem.~5.2.6]{NogWinNTSCVDA} when replacing ``holomorphic'' with
  ``smooth''.
\end{proof}

\begin{proposition}
\label{prop:LocTransDiffOrbitSpace}
Let $\sfG$ be a sliceable differentiable groupoid.  Then the orbit space
$|\sfG|$ inherits the structure of a differentiable space with respect to which
the quotient map $\pi\co \sfG_0 \to |\sfG|$ is a smooth map.  Specifically, the structure
sheaf of $|\sfG|$ is given by the differentiable functions on $\sfG_0$ that are constant
on orbits.
\end{proposition}
\begin{proof}
We define the structure sheaf of $|\sfG|$ to be the sheaf of continuous functions
on $|\sfG|$ which pull back under the projection $\pi$ to $\sfG$-invariant smooth functions
on $\sfG_0$.
For each orbit $\sfG \obj{x}\in |\sfG|$, we may use condition (LT\ref{it:ExistenceSlice})
in Definition \ref{def:LocTransDiffGroupoid} to identify a neighborhood of
$\sfG \obj{x}$ with $|G_\obj{x}\ltimes Y_\obj{x}|$, a Hausdorff differentiable space as explained in
Example \ref{ex:DiffGroupoids} (\ref{ex:DiffTranslationGpoid}).
That these local identifications are well defined and isomorphisms of Hausdorff differentiable
spaces are consequences of (LT\ref{it:CompatibilitySlices}) and (LT\ref{it:ExistenceBisections}).
\end{proof}

\subsection{Differentiable stratified groupoids}
\label{Sec:DiffStratGrpoid}

Recall that by a \emph{stratified submersion} (respectively \emph{stratified immersion}), one understands
a morphism $f\co X\to Y$ of reduced differentiable stratified spaces such that the restriction of $f$
to a connected component of a stratum of the maximal decomposition of $X$ is a submersion
(respectively immersion), cf.~\cite[1.2.10]{PflaumBook}. A \emph{stratified surjective submersion}
is a stratified submersion that maps connected components of strata onto connected components
of strata, and a \emph{stratified embedding} is a stratified immersion that is injective
on connected components of strata.

We will now state the definition of a differentiable stratified groupoid, the primary object
of investigation in this paper. Let us first give a brief explanation of this definition.
A differentiable stratified groupoid is a reduced differentiable groupoid such that
the spaces of objects and arrows are differentiable stratified spaces, required by condition
(DSG\ref{it:DSGStratSpaces}), and the structure maps are stratified mappings, see condition
(DSG\ref{it:DSGStrucMaps}). In order for the requirement that multiplication
$m\co \sfG_1 \sttimes \sfG_1 \to \sfG_1$
is a stratified mapping to make sense, the fiber product $\sfG_1 \sttimes \sfG_1$ has to be a
stratified space. By condition (DSG\ref{it:DSGSubmersion}) we impose requirements on the structure maps
to ensure this;
see Remark \ref{rem:defDiffStratGpoid}(\ref{ite:DomainProdDiffStratGrpoid}).
Condition (DSG\ref{it:DSGStratSourceFib}) is a kind of equivariance
condition for the stratification. It entails that every set germ of a source
fiber is contained in a stratum of the arrow space, so that the stratification
of the arrow space is not unnecessarily fine. This will be used to show that
$\sfG$-orbits are locally contained in strata; see
Lemma \ref{lem:OrbitInStrata} below. Likewise, condition (DSG\ref{it:DSGStratSection})
requires that smooth bisections defined on open subsets of strata act on $\sfG_1$ in
a way compatible with its stratification.

\begin{definition}
\label{def:DiffStratGpoid}
A \emph{differentiable stratified groupoid} is a reduced differentiable group\-oid
$\sfG$ such that the following properties hold true.
\begin{enumerate}[{\rm (DSG1)}]
\item
        \label{it:DSGStratSpaces}
        $\sfG_0$ and $\sfG_1$ are differentiable stratified spaces with respective
        stratifications $\mathcal{S}^0$ and $\mathcal{S}^1$.
\item
        \label{it:DSGStrucMaps}
        The structure maps $s$, $t$, $i$, $u$, and $m$ are stratified mappings.
\item
        \label{it:DSGSubmersion}
        The maps $s$ and $t$ are stratified surjective submersions, and $u$ is a stratified embedding.
\item
        \label{it:DSGStratSourceFib}
        For every $\obj{x}\in \sfG_0$ and arrow $\arr{g} \in s^{-1} (\obj{x})$ the set germ $ [s^{-1}(\mathcal{S}_\obj{x}^0)]_\arr{g}$
        is a subgerm of $\mathcal{S}_\arr{g}^1$.
\item
        \label{it:DSGStratSection}
        Let $\obj{x}\in\sfG_0$, $\arr{g}\in\sfG_1$ with $t(\arr{g}) = \obj{x}$, and $U$ be an open
        connected neighborhood of $\obj{x}$ within the stratum of $\sfG_0$ containing $\obj{x}$.
        Assume that $\sigma\co U\to \sfG_1$ is
        a smooth bisection of $\sfG$.  Then the map
        $L_\sigma\co t^{-1}(U)\to\sfG_1$ defined by $\arr{h}\mapsto \sigma(t(\arr{h}))\arr{h}$
        satisfies $L_\sigma(\mathcal{S}_\arr{g}^1) = \mathcal{S}_{L_\sigma(\arr{g})}^1$.
\end{enumerate}

If $\sfG$ and $\sfH$ are differentiable stratified groupoids, a \emph{morphism of differentiable stratified groupoids}
is a morphism of differentiable groupoids $f\co\sfG\to\sfH$ such that $f_0$ and $f_1$ are in addition
stratified mappings.

A differentiable stratified groupoid $\sfG$ is called a \emph{Lie groupoid}
if in addition the following axiom holds true.
\begin{enumerate}[{\rm (DSG1)}]
\setcounter{enumi}{5}
\item  \label{it:LieGroupoid}
        The stratifications $\mathcal{S}^0$ and  $\mathcal{S}^1$ are induced
        by $\sfG_0$ and  $\sfG_1$, respectively, which in other words means
        that both $\sfG_0$ and $\sfG_1$ are smooth manifolds and their
        stratifications have only one stratum.
\end{enumerate}

A \emph{structurally (b)-regular} differentiable stratified groupoid is a differentiable stratified groupoid $\sfG$ such that
the stratifications of $\sfG_0$ and $\sfG_1$, as well as the induced
stratification of $\sfG_1 \sttimes\sfG_1$, are Whitney (b)-regular. Similarly,
$\sfG$ is \emph{structurally locally trivial} if
the stratifications of $\sfG_0$, $\sfG_1$, and the induced stratification of
$\sfG_1 \sttimes\sfG_1$ are topologically locally trivial
(cf.~\cite[Sec.~1.4]{PflaumBook}, Def.~\ref{def:TopLocTriv}).
\end{definition}

\begin{remark}
\label{rem:defDiffStratGpoid}
\begin{enumerate}[(a)]
\item
\label{ite:AddProps}
In the definition, $[s^{-1}(\mathcal{S}_\obj{x}^0)]_\arr{g}$ of course means the set germ of
$s^{-1}(S)$ at $\arr{g}$, where $S$ is a set defining the germ $\mathcal{S}_\obj{x}^0$ at
$\obj{x}$. Note that a smooth bisection $\sigma\co S\supset U \to \sfG_1$ as in (DSG\ref{it:DSGStratSection}) has image in the
stratum of $\sfG_1$ through $\arr{g} = \sigma (\obj{x})$ by (DSG\ref{it:DSGStrucMaps}) and (DSG\ref{it:DSGStratSourceFib}).
The existence of such bisections $\sigma$ is demonstrated by
Lemma \ref{lem:BisectionsExist} and Corollary \ref{cor:ComponentsSameDim}
below. 
\item
Though $\sfG_0$ and $\sfG_1$ are locally Hausdorff spaces, we do not require that they are Hausdorff.
See Appendix \ref{app:DiffStratSpaces} for stratifications
of locally Hausdorff spaces.
\item
One readily checks that Definition \ref{def:DiffStratGpoid} reduces to the standard definition of a
Lie groupoid if (DSG\ref{it:LieGroupoid}) is fulfilled. Observe that in the case of a Lie groupoid,
conditions (DSG\ref{it:DSGStratSourceFib}) and (DSG\ref{it:DSGStratSection}) become trivial.
\item
\label{ite:DomainProdDiffStratGrpoid}
We will always let $\sfG_1 \sttimes \sfG_1$ carry the stratification induced by the
stratification of $\sfG_1$, the existence of which is guaranteed by
Lemma \ref{lem:InducedStrat}.
\end{enumerate}
\end{remark}

Our definition of a structurally (b)-regular differentiable stratified groupoid is stronger than the one of a
stratified Lie groupoid given in \cite[Def.~4.16]{FernandesOrtegaRatiu} in that we require conditions
(DSG\ref{it:DSGStratSourceFib}) and (DSG\ref{it:DSGStratSection}).
The following examples illustrate the kinds of behavior we preclude by
requiring these conditions.

\begin{example}
\label{ex:DSGCounterexamples}
\begin{enumerate}[(a)]
\item
\label{ite:CounterG0TooFine}
Consider the translation groupoid $\sfG = \sphere^1 \ltimes \sphere^1$ where the action is
by left-translation.
We decompose $\sfG_0 = \sphere^1$ into pieces $\{ 1\}$ and $\sphere^1\smallsetminus\{1\}$,
and $\sfG_1 = \sphere^1 \times \sphere^1$ into pieces $\{(1,1)\}$,
$\{ (a,1)\mid a\neq 1\}$, $\{(a, a^{-1})\mid a\neq 1\}$, and
$\{ (a,b)\mid  b\neq 1,\; a\neq b^{-1}\}$. Then it is immediate to check that
(DSG\ref{it:DSGStratSpaces}), (DSG\ref{it:DSGStrucMaps}), and
(DSG\ref{it:DSGSubmersion}) are satisfied.
However, condition (DSG\ref{it:DSGStratSourceFib}) fails, as
the germ of $s^{-1}(1) = \{ (a,1)\mid a\in \sphere^1 \}$ is not contained in
the stratum $\{(1,1)\}$. 
Note that $\sfG_0$ consists of a single connected orbit that is given an
``artificial" stratification that is too fine.
Of course, $\sfG$ is a Lie groupoid when given the trivial stratifications of $\sfG_0$
and $\sfG_1$.

\item
\label{ite:CounterG1TooFine}
Let $G$ be a compact Lie group of positive dimension and let $\sfG = G \ltimes \{p\}$
be the translation groupoid associated to the trivial action of $G$ on a one-point space.
Define a stratification of $\sfG_1 = G \times\{p\}$ by the decomposition into
$\{ (1,p)\}$ and $\{(g,p)\mid  g\neq 1\}$.
Obviously, the source and target maps are stratified submersions, the unit map
is a stratified embedding, and the inverse map is a stratified mapping. Though
the multiplication map $\sfG_1 \sttimes \sfG_1 = \sfG_1\times\sfG_1 \to \sfG_1$
given by $((g,p),(h,p))\mapsto(gh,p)$ is not a stratified mapping with respect
to the induced stratification, the space $\sfG_1 \sttimes \sfG_1$ does admit a
stratification with respect to which the multiplication map $m$ is stratified.
Specifically, we may decompose $\sfG_1 \sttimes \sfG_1$ into the pieces
\[
    \{((g^{-1},p),(g,p))\mid g \in G \}
    \quad\mbox{and}\quad
    \{((g,p),(h,p))\mid g\neq h^{-1}\},
\]
and then $m$ becomes stratified. In this case, while $\sfG_0$ consists of a
single stratum, the stratification of $\sfG_1$ is ``too fine" so
that (DSG\ref{it:DSGStratSourceFib}) again fails;
the preimage $s^{-1}(p) = \sfG_1$ is not contained in the stratum $\{ (1,p)\}$
through $(1,p)$.

\item
\label{ite:CounterDSG5}
Let $\sfG$ be the pair groupoid on the topological disjoint union $\R\sqcup\{p\}$.
Give $\sfG_0$ and  $\sfG_1$ the stratifications
by connected components and natural differentiable structures. One checks that
(DSG\ref{it:DSGStratSpaces}), (DSG\ref{it:DSGStrucMaps}), (DSG\ref{it:DSGSubmersion}), and
(DSG\ref{it:DSGStratSourceFib}) are satisfied. However, let
$\sigma \co \{ p \} \to \sfG_1$ be the
smooth bisection with the single value $\sigma(p) = (0,p)$.
For any arrow of the form $(p,y)$ one has
$\sigma(t(p,y))(p,y) = (0,y)$. Hence the mapping $\arr{h} \mapsto \sigma(t(\arr{h}))\arr{h}$
maps the stratum $\{ (p,y)\mid y\in\R\}$ of $\sfG_1$ into the stratum $\R^2$
as the $y$-axis. Thus (DSG\ref{it:DSGStratSection}) fails.
\end{enumerate}
\end{example}
\noindent
We next collect some useful consequences of Definition \ref{def:DiffStratGpoid},
indicating how the set germs defining the stratifications of $\sfG_0$ and $\sfG_1$ are related via
the source and target maps of $\sfG$ and hence determine one another.
We hereby assume for the remainder of this section that  $\sfG$ denotes a
differentiable stratified groupoid.
\begin{lemma}
\label{lem:DSGTrivConsequences}
Let $\arr{g}\in\sfG_1$ with $s(\arr{g})=\obj{x}$ and $t(\arr{g}) = \obj{y}$. Then one has
\begin{align}
  \label{Ite1}  \mathcal{S}_\arr{g}^1 & = [s^{-1}(\mathcal{S}_\obj{x}^0)]_\arr{g},  \text{ i.e.~$\mathcal{S}^1$ is the
    pullback of $\mathcal{S}^0$ via $s$, \cite[(2.3)]{MatherStratMap}},\\
  \label{Ite2}  \mathcal{S}_\obj{x}^0 & =[s(\mathcal{S}_\arr{g}^1)]_\obj{x},  \\
  \label{Ite3}  \mathcal{S}_\arr{g}^1 & =[t^{-1}(\mathcal{S}_\obj{y}^0)]_\arr{g},\\
  \label{Ite4}  \mathcal{S}_\obj{y}^0 &= [t(\mathcal{S}_\arr{g}^1)]_\obj{y}, \text{ and}, \\
  \label{Ite5}  \mathcal{S}_\obj{y}^0 &= [t(s^{-1}(\mathcal{S}_\obj{x}^0))]_\obj{y} .
\end{align}
In particular, the stratifications $\mathcal{S}^0$ of $\sfG_0$ and $\mathcal{S}^1$ of $\sfG_1$ determine one another.
\end{lemma}
\begin{proof}
Let $R_\arr{g}$ be the connected component of the stratum of $\sfG_1$ containing $\arr{g}$,
and let $S_\obj{x}$ and $S_\obj{y}$ be the connected components of the strata of $\sfG_0$
containing $\obj{x}$ and $y$, respectively. Then $s_{|R_\arr{g}}$ and $t_{|R_\arr{g}}$ are, respectively, surjective submersions onto $S_\obj{x}$ and $S_\obj{y}$.
By (DSG\ref{it:DSGStratSourceFib}), there exists a relatively open neighborhood
$U_\obj{x}$ of $\obj{x}$ in $S_\obj{x}$ and an open neighborhood $V_\arr{g}$ of $\arr{g}$ in $\sfG_1$ such
that $V_\arr{g}\cap s^{-1}(U_\obj{x}) \subset R_\arr{g}$. As $s_{|R_\arr{g}}$ is a smooth map onto $S_\obj{x}$,
$s^{-1}(U_\obj{x})$ is a relatively open neighborhood of $\arr{g}$ in $R_\arr{g}$,
proving \eqref{Ite1}. Since $s_{|R_\arr{g}}$ is a submersion, hence an open map,
$s(V_\arr{g}\cap s^{-1}(U_\obj{x}))$ is an open neighborhood of $\obj{x}$ in $S_\obj{x}$, which gives
\eqref{Ite2}. Then \eqref{Ite3} and \eqref{Ite4} follow from the fact that
$t = s\circ i$ and that $i$ is a stratified mapping with $i^2 = \id_{\sfG_1}$.
Finally, \eqref{Ite5} is a consequence of  \eqref{Ite1} and  \eqref{Ite4}.
\end{proof}

\begin{lemma}
\label{lem:OrbitInStrata}
Let $\sfG$ be a proper differentiable stratified groupoid and let
$\obj{x} \in \sfG_0$ be a point. Then the connected component of the orbit $\sfG\obj{x}$ containing $\obj{x}$ is contained in the stratum of $\sfG_0$ containing $\obj{x}$.
\end{lemma}
\begin{proof}
Let $S_\obj{x}$ be the connected component of the stratum of $\sfG_0$ containing $\obj{x}$.
Suppose for contradiction that the germ of $\sfG \obj{x}$ at $\obj{x}$ is not a subgerm of
$\mathcal{S}_\obj{x}^1$. Then one may construct a sequence $(  \obj{x}_n )_{n\in \N}$ of
elements of $\sfG \obj{x}$ such that $\lim_{n\to\infty} \obj{x}_n =  \obj{x}$, yet each $\obj{x}_n$ is
not contained in $S_\obj{x}$. Since the set
$K = \{ (\obj{x}, \obj{x}_n)\}_{n\in \N}\cup\{ (\obj{x},\obj{x}) \}$ is compact and
$\sfG$ is proper, 
the preimage $C:= (s,t)^{-1}(K)$ is a quasi-compact subset of $\sfG_1$.

For each $n$, choose an arrow $\arr{g}_n\in\sfG_1$ with source $\obj{x}$ and target $\obj{x}_n$. As
each $\arr{g}_n$ is in the quasi-compact set $C$, there is a subsequence
$(\arr{g}_{n_k})_{k\in \N}$ with limit $\arr{g} \in\sfG_1$. However, (DSG\ref{it:DSGStratSourceFib}) entails the existence of an open neighborhood $U_\obj{x}$ of $\obj{x}$ in $S_\obj{x}$ and a neighborhood $V_\arr{g}$ of $\arr{g}$ in $\sfG_1$ such that
$V_\arr{g} \cap s^{-1}(U_\obj{x})$ is contained in the connected component $R_\arr{g}$ of the stratum of $\sfG_1$ containing $\arr{g}$. Infinitely many of the $\arr{g}_{n_k}$ must be contained in $V_\arr{g}$ and each $\arr{g}_{n_k}$
is an element of $s^{-1}(U_\obj{x})$ by construction, so infinitely many of the $\arr{g}_{n_k}$ are contained
in $R_\arr{g}$. But since $t$ is a stratified mapping, infinitely many of the $\obj{x}_{n_k}$
are contained in $S_\obj{x}$, which is a contradiction. It follows that the germ of $\sfG \obj{x}$ at $\obj{x}$ is a subgerm of
$\mathcal{S}_\obj{x}$, hence that the connected component of $\sfG \obj{x}$ containing $\obj{x}$ is a
subset of $S_\obj{x}$.
\end{proof}

The following property is important in realizing the consequences of
(DSG\ref{it:DSGStratSection}) and is proven as in the case of a Lie groupoid;
cf.~\cite[Prop.~1.4.9]{MackenzieGenThry}.

\begin{lemma}
\label{lem:BisectionsExist}
Let $\obj{x}, \obj{y}$ be two points of the differentiable stratified groupoid $\sfG$ and $\arr{g}$ an arrow
with $s(\arr{g}) =  \obj{x}$ and $t(\arr{g}) =  \obj{y}$. Denote by $S_\obj{x}$ and $S_\obj{y}$ the connected
components of the strata of $\sfG_0$ containing $\obj{x}$ and $ \obj{y}$, respectively.
If $\dim S_\obj{x} \leq \dim S_\obj{y}$, then there exists a relatively open neighborhood
$U_\obj{x}$ of $\obj{x}$ in $S_\obj{x}$ and a smooth bisection $\sigma$ of $\sfG$ on $U_\obj{x}$ such that
$\sigma$ is a stratified mapping and $\sigma(\obj{x}) = \arr{g}$.
\end{lemma}
Of course, since $\sigma$ is only defined on a subset of the stratum through $\obj{x}$,
$\sigma$ being a stratified mapping means that its image is contained in the stratum
containing $\arr{g}$. The
hypothesis that $\dim S_\obj{x} \leq \dim S_\obj{y}$ will be seen to be unnecessary below.
\begin{proof}
Let $R_\arr{g}$ be the connected component of the stratum of $\sfG$ containing $\arr{g}$. Then
$s_{|R_\arr{g}}$ and $t_{|R_\arr{g}}$ are surjective submersions onto $S_\obj{x}$ and $S_\obj{y}$, respectively.
Moreover, there are relatively open neighborhoods of $\arr{g}$ in $s^{-1}(\obj{x})$ and $t^{-1}(\obj{y})$
contained in $R_\arr{g}$ by (DSG\ref{it:DSGStratSourceFib}) and
Lemma \ref{lem:DSGTrivConsequences} \eqref{Ite3}. Then there are subspaces
$E \subset F$ of the tangent space $T_\arr{g} R_\arr{g}$ such that
$T_\arr{g} R_\arr{g} = T_\arr{g} (s^{-1}(\obj{x})) \times E = T_\arr{g}(t^{-1}(\obj{y})) \times F$.
Choose a smooth local section $\sigma\co U_\obj{x} \to R_\arr{g}$ such that
$\sigma(\obj{x}) = \arr{g}$ and such that the image of $T_\obj{x}\sigma$ is $E$.
Then $T_\obj{x} (t\circ\sigma)$ is injective so that we can shrink $U_\obj{x}$ in a way
that $t\circ\sigma$ is a diffeomorphism onto its image.
\end{proof}

As an important consequence, we may now conclude that the strata of $\sfG_0$
that meet the orbit $\sfG \obj{x}$ of a point $\obj{x} \in \sfG_0$ must all have the same
dimension. The proof follows \cite[Corollaries 1.4.11 \& 1.4.12]{MackenzieGenThry}.

\begin{corollary}
\label{cor:ComponentsSameDim}
Let $\sfG$ be a differentiable stratified groupoid, and
let $S$ be a connected component of a stratum of $\sfG_0$. Then the following
holds true.
\begin{enumerate}[{\rm (1)}]
\item \label{IteCor1}
      Each stratum of $\sfG_1$ contained in $s^{-1}(S)$ has the same dimension.
\item \label{IteCor2}
      The rank of $t$ on $s^{-1}(S)$ is constant.
\item \label{IteCor3}
      Each connected component $S^\prime$ of a stratum of $\sfG_0$ such that
      $s^{-1}(S)\cap t^{-1}(S^\prime) \neq \emptyset$ has the same dimension as $S$.
\end{enumerate}
\end{corollary}
\begin{proof}
Let $S_1$ and $S_2$ be (not necessarily distinct) connected components of strata of
$\sfG_0$ such that $s^{-1}(S_1)\cap t^{-1}(S_2) \neq \emptyset$. We assume
$\dim S_1 \leq \dim S_2$. Otherwise, we may switch roles and apply the inverse
map, so this hypothesis introduces no loss of generality.
As $s$ and $t$ are stratified surjective submersions, $s^{-1}(S_1)$
and $t^{-1}(S_2)$ are both unions of connected components of strata of $\sfG_1$. Hence
their intersection is a union of connected components of strata. Let $\arr{g} \in \sfG_1$ with
$s(\arr{g}) = \obj{x} \in S_1$ and $t(\arr{g}) =  \obj{y} \in S_2$. Then there exists by
Lemma \ref{lem:BisectionsExist} a smooth bisection $\sigma$ of $\sfG$ on a relatively open
neighborhood $U_\obj{x}$ of $\obj{x}$ in $S_1$ such that $\sigma(\obj{x}) = \arr{g}$. Let $R_{u(\obj{x})}$ and $R_\arr{g}$
denote the connected components of the strata of $\sfG_1$ containing $u(\obj{x})$ and $\arr{g}$,
respectively. By (DSG\ref{it:DSGStratSection}), there is a relatively open
neighborhood $V_{u(\obj{x})}$ of $u(\obj{x})$ in $R_{u(\obj{x})}$ such that $L_\sigma(V_{u(\obj{x})})$ is a
relatively open neighborhood of $\arr{g}$ in $R_\arr{g}$. The requirement that $t\circ\sigma$ is
injective implies that $L_\sigma$ is injective, hence that $R_{u(\obj{x})}$ and $R_\arr{g}$ have the
same dimension. Since $S_2$ and $\arr{g}$ were arbitrary, \eqref{IteCor1} follows.
Moreover, from the definition of $L_\sigma$, we have
$t_{|L_\sigma(V_{u(\obj{x})})} \circ L_\sigma = (t\circ\sigma)\circ t_{|V_{u(\obj{x})}}$.
Then, as $L_\sigma(V_{u(\obj{x})})$ is an open neighborhood of $\arr{g}$ in $R_\arr{g}$ and $t\circ\sigma$ is
a diffeomorphism onto its image, the ranks of $t_{|R_{u(\obj{x})}}$ at $u(\obj{x})$ and $t_{|R_\arr{g}}$ at $\arr{g}$
coincide, yielding \eqref{IteCor2}.
Since $t$ is a stratified surjective submersion, \eqref{IteCor3} is immediate.
\end{proof}

\begin{remark}
\label{rem:DimsEqual}
In particular, the hypothesis $\dim S_\obj{x} \leq \dim S_\obj{y}$ in Lemma \ref{lem:BisectionsExist}
is superfluous by Corollary \ref{cor:ComponentsSameDim}. Indeed, given the other hypotheses, we always have
$\dim S_\obj{x} = \dim S_\obj{y}$.
\end{remark}

\begin{example}
Note that if $S_1$ and $S_2$ are connected components of strata of $\sfG_0$,
even connected components of the same stratum, it need not be the case that the strata of
$s^{-1}(S_1)$ and $s^{-1}(S_2)$ have the same dimension. As an example, let $G$ and $H$
be compact Lie groups, and let $\sfG$ be the disjoint union of $G\ltimes \{p\}$ and
$H \ltimes \{ q\}$. Then $\sfG_0$ is the discrete set $\{p,q\}$. The maximal stratification
of $\sfG_0$ contains a single stratum with two one-point connected components, yet the space of
arrows of $\sfG_{|\{p\}}$ and $\sfG_{|\{q\}}$ have dimensions $\dim G$ and $\dim H$, respectively,
which clearly need not coincide.
\end{example}

\begin{proposition}
\label{prop:RestrictLie}
Assume that $\sfG_0$ and $\sfG_1$ are topologically locally trivial, and let
$S \subset \sfG_0$ be a connected component of a stratum of $\sfG_0$.
Let $\mathcal{P}$ be a collection of connected components of strata of
$\sfG_0$ such that $S \in\mathcal{P}$, and such that for each
$S^\prime \in\mathcal{P}$ the relation
$s^{-1}(S)\cap t^{-1}(S^\prime)\neq\emptyset$ is satisfied. Letting
$P = \bigcup_{S^\prime \in\mathcal{P}} S^\prime$, the restriction $\sfG_{|P}$ then
is a Lie groupoid.
\end{proposition}
\begin{proof}
By Corollary \ref{cor:ComponentsSameDim}, each element of $\mathcal{P}$ has the same dimension
as $S$, hence each connected component of a stratum of $\sfG_1$ contained in
$s^{-1}(P)$ has the same dimension as well. The hypothesis that $\sfG_0$ is a
topologically locally trivial implies that if $S^\prime < S^{\prime\prime}$ cannot occur
for $S^\prime, S^{\prime\prime} \in\mathcal{P}$ so that $P$ is a manifold. The same holds for
$\sfG_1$ so that $s^{-1}(P)\cap t^{-1}(P)$ is a union of strata of the same dimension
and therefore a manifold. Hence $P$ and $s^{-1}(P)\cap t^{-1}(P)$ are manifolds,
which can both be given their trivial stratifications. The claim follows.
\end{proof}

Of course, the hypothesis that $\sfG_0$ and $\sfG_1$ are topologically locally trivial is
only required so that the strata of $\sfG_0$ and $\sfG_1$ of the same dimension are not
contained in one another's closures. Any other hypothesis that ensures this is as well
sufficient.

We also note the following, which will be useful in the sequel.

\begin{corollary}
\label{cor:OrbitMnfld}
Assume that  $\sfG$ is proper and let $\obj{x}\in\sfG_0$. Then each
connected component of the orbit $\sfG \obj{x}$ is a smooth submanifold of
$\sfG_0$. If $\sfG_0$ and $\sfG_1$ are in addition topologically locally
trivial, then $\sfG \obj{x}$ is a smooth submanifold of $\sfG_0$.
\end{corollary}
\begin{proof}
Let $S_\obj{x}$ be the connected component of the stratum containing $\obj{x}$. Then the
connected component of $\sfG \obj{x}$ containing $\obj{x}$ is contained in $S_\obj{x}$ by Lemma \ref{lem:OrbitInStrata}.
The restricted groupoid $\sfG_{|S_\obj{x}}$ is a Lie groupoid by Proposition \ref{prop:RestrictLie}. Since $\sfG \obj{x}\cap S_\obj{x} = (\sfG_{|S_\obj{x}}) \obj{x}$
this implies that the connected components of $\sfG \obj{x}$ contained in
$S_\obj{x}$ are smooth submanifolds of $S_\obj{x}$, hence of $\sfG_0$.
If $\sfG_0$ and $\sfG_1$ are topologically locally trivial, then the restriction of
$\sfG$ to the saturation of $S_\obj{x}$ is as well a Lie groupoid, and the same argument
applies.
\end{proof}

Finally, we include the following example to demonstrate that the hypothesis that
$\sfG_0$ and $\sfG_1$ are topologically locally trivial (or a similar requirement)
is necessary in Proposition \ref{prop:RestrictLie} and Corollary \ref{cor:OrbitMnfld},
cf.~\cite[1.1.12]{PflaumBook}.

\begin{example}
\label{ex:TopSin}
Let $\sfG_0 = S_1 \cup S_2 \subset \R^2$ where $S_1 = \{0\}\times (-1,1)$ and
$S_2 = \{ (x,y)\in \R^2\mid x > 0, y = \sin 1/x \}$, and let $\sfG$ be the pair
groupoid on $\sfG_0$. Then $S_1$ and $S_2$, both of dimension $1$, are the pieces
of a decomposition of $\sfG_0$ with $S_1 \subset \overline{S_2}$. The decomposition
of $\sfG_1$ is into pieces of the form $S_i \times S_j$, each of dimension $2$.
Then $\sfG$ is a differentiable stratified groupoid. However, the single orbit
$\sfG_0$ is not locally connected and hence not a manifold, and $\sfG$ is not
a Lie groupoid. Clearly, however, the restriction of $\sfG$ to either stratum
of $\sfG_0$ is a Lie groupoid.
\end{example}



\subsection{Morita equivalence}
\label{subsec:MoritaEquiv}

Let $\sfG$ be a topological groupoid, $Y$ a topological space, and
$f \co Y \to \sfG_0$ a continuous function. Following \cite{TuNonHausGpoid},
we denote by $\sfG[Y]$ the groupoid with object space
$\sfG[Y]_0 := Y$, arrow space
\[
    \sfG[Y]_1 :=
    (Y\times Y)\fgtimes{(f,f)}{(t,s)} \sfG_1
    =
    \{ (\obj{y}, \obj{z}, \arr{g}) \in Y \times Y \times \sfG_1 \mid
    t(\arr{g}) = f(\obj{y}), \: s(\arr{g}) = f(\obj{z}) \},
\]
and structure maps given as follows.
The source of $(\obj{y},\obj{z}, \arr{g}) \in  \sfG[Y]_1$ is $\obj{z}$,
its target is $\obj{y}$. Multiplication maps
$ (\obj{y}, \obj{w}, \arr{g}), (\obj{w}, \obj{z}, \arr{h}) \in  \sfG[Y]_1$ with $s(\arr{g})=t(\arr{h})$ to
$$
  (\obj{y}, \obj{w}, \arr{g}) \cdot (\obj{w}, \obj{z}, \arr{h}) := (\obj{y}, \obj{z}, \arr{g}\arr{h}).
$$
The unit map is $\sfG[Y]_0 \to \sfG[Y]_1$, $y \mapsto (\obj{y}, \obj{y}, u(f(\obj{x})) )$, and the inverse map is
$\sfG[Y]_1 \to \sfG[Y]_1$,  $(\obj{y}, \obj{z}, \arr{g}) \mapsto (\obj{z}, \obj{y}, \arr{g}^{-1})$. Then $\sfG[Y]$ is a subgroupoid
of $Y \times Y \times \sfG$ where $Y \times Y$ denotes the pair groupoid and
we identify $\sfG[Y]_0 = Y$ with
$Y \fgtimes{f}{\operatorname{id}} \sfG_0$.
By \cite[Prop.~2.7]{TuNonHausGpoid},  $\sfG[Y]$ is a locally closed
subgroupoid of $Y \times Y \times \sfG$ if $\sfG_0$ is locally Hausdorff.
Moreover, $\sfG[Y]$ is locally compact if $\sfG$ and $T$ are locally compact.
Finally, by \cite[Prop.~2.22]{TuNonHausGpoid}, $\sfG[Y]$ is proper
if $\sfG$ is proper.

Now suppose that $\sfG$ is a differentiable groupoid, $Y$  a differentiable space,
and $f$  a differentiable map. Then it is straightforward to see that $\sfG[Y]$ is a
differentiable subgroupoid of $Y \times Y \times \sfG$. In particular,
both $\sfG[Y]_0$ and $\sfG[Y]_1$ are defined as fibered products.
Similarly, if $\sfG$ and $\sfH$
are differentiable stratified groupoids, $Y$ is  a differentiable stratified space,
and $f$ a differentiable stratified surjective submersion, then
$\sfG[Y]$ is a differentiable stratified groupoid as well where $\sfG[Y]_0$
and $\sfG[Y]_1$ are given the induced stratifications.

\begin{definition}
\label{def:MoritaEquiv}
Two open topological groupoids $\sfG$ and $\sfH$ are called
\emph{Morita equivalent as topological groupoids} if there exists a topological space $Y$ together
with open surjective continuous functions $f \co Y \to \sfG_0$ and
$g \co Y \to \sfH_0$ such that $\sfG[Y]$ and $\sfH[Y]$ are isomorphic as
topological groupoids. If $\sfG$ and $\sfH$ are differentiable groupoids,
$Y$ is a differentiable space, and $f$ and $g$ are
differentiable maps, then  $\sfG$ and $\sfH$ are called \emph{Morita equivalent as differentiable groupoids}
if $\sfG[Y]$ and $\sfH[Y]$ are isomorphic as differentiable groupoids.
Similarly, if $\sfG$ and $\sfH$ are differentiable stratified groupoids, $Y$ is a
differentiable stratified space, and $f$ and $g$ are differentiable stratified
surjective submersions, then  $\sfG$ and $\sfH$ are called \emph{Morita equivalent as differentiable
stratified groupoids} if $\sfG[Y]$ and $\sfH[Y]$ are isomorphic as differentiable
stratified groupoids.
\end{definition}

One verifies immediately  that Morita equivalence is transitive.
Specifically, if the maps $\sfG_0\overset{f_1}{\longleftarrow}Y\overset{g_1}{\longrightarrow}\sfH_0$
realize a Morita equivalence between (topological, differentiable, or differentiable
stratified) groupoids $\sfG$ and $\sfH$, and
$\sfH_0\overset{f_2}{\longleftarrow}Y^\prime\overset{g_2}{\longrightarrow}\mathsf{K}_0$ is
a Morita equivalence between $\sfH$ and $\mathsf{K}$, then
$\sfG_0 \xleftarrow{f_1\circ pr_1} Y\fgtimes{g_1}{f_2}Y^\prime \xrightarrow{g_2\circ pr_2} \mathsf{K}_0$
induces a Morita equivalence between $\sfG$ and $\mathsf{K}$.

\begin{proposition}
\label{prop:MoritaHomeo}
If $\sfG$ and $\sfH$ are Morita equivalent open topological groupoids,
then the orbit spaces $|\sfG|$ and $|\sfH|$ are homeomorphic.  Moreover,
if $\sfG$ and $\sfH$ are Morita equivalent differentiable groupoids
such that $|\sfG|$ and $|\sfH|$ admit the structure of differentiable spaces
and the quotient maps are differentiable, then
$|\sfG|$ and $|\sfH|$ are isomorphic as differentiable spaces.
\end{proposition}
\begin{proof}
Because the orbit spaces $|\sfG[Y]|$ and $|\sfH[Y]|$ are clearly homeomorphic,
it is sufficient to show that $|\sfG|$ is homeomorphic to $|\sfG[Y]|$.
To see this, define the map $\sfG[Y]_0 \to \sfG_0$ by $(\obj{y}, \obj{x}) \mapsto \obj{x}$.
Given an arrow $\arr{g}\in\sfG_1$ from $\obj{x}$ to $\obj{x}^\prime$, there exists, by
the surjectivity of $f$, a $\obj{y}^\prime$ such that $f(\obj{y}^\prime) = \obj{x}^\prime$
and hence an arrow $(\obj{y}, \obj{y}^\prime, \arr{g})$ from $(\obj{y}, \obj{x})$ to $(\obj{y}^\prime, \obj{x}^\prime)$.
Conversely, if $(\obj{y}, \obj{y}^\prime, \arr{g})$ is an arrow from $(\obj{y}, \obj{x})$ to $(\obj{y}^\prime, \obj{x}^\prime)$,
then $\arr{g}$ is by definition an arrow from $\obj{x}$ to $\obj{x}^\prime$.  Hence
$(\obj{y},\obj{x}) \mapsto \obj{x}$ maps orbits to orbits. In the differentiable case, this map is
obviously differentiable, so that if the quotient map $\sfG_0 \to |\sfG|$ is
differentiable, then its composition with $(\obj{y},\obj{x}) \mapsto \obj{x}$ is also differentiable.
\end{proof}

For Lie groupoids, the notion of Morita equivalence is often defined in terms
of morphisms called \emph{weak equivalences}. We introduce a similar notion
as follows.

\begin{definition}
\label{def:WeakEquiv}
A morphism $f\co\sfG\to\sfH$ of differentiable stratified groupoids is called
a \emph{weak equivalence} if it is essentially surjective and fully faithful, i.e.
if the following two conditions are satisfied.
\begin{enumerate}
\item[{\rm (ES)}] (Essential Surjectivity)
    The map
    $t\circ\operatorname{pr}_1\co \sfH_1\fgtimes{s}{f_0}\sfG_0 \to \sfH_0$
    is a stratified surjective submersion.
\item[{\rm (FF)}] (Full Faithfulness) The arrow space $\sfG_1$ is a fibered product via the diagram
\[
    \xymatrix{
        \sfG_1
            \ar[d]_{(s,t)}      \ar[r]^{f_1}              &\sfH_1 \ar[d]^{(s,t)}
                    \\
        \sfG_0\times\sfG_0
            \ar[r]^{(f_0,f_0)}
        &
        \sfH_0\times\sfH_0.
    }
\]
\end{enumerate}
\end{definition}

\begin{remark}
There is an analogous notion of a weak equivalence between differentiable groupoids,
where axiom (ES) is replaced with the requirement that
$t\circ\operatorname{pr}_1$ is a surjective map that admits smooth local sections.
\end{remark}

Note that if the object and arrow spaces of $\sfG$ and $\sfH$ are topologically
locally trivial, then the restrictions of $\sfG$ and $\sfH$ to
connected components of strata are Lie groupoids by Proposition \ref{prop:RestrictLie}.
One immediately checks that the restriction of a weak equivalence to strata yields a
weak equivalence of Lie groupoids. We have the following.

\begin{proposition}
Let $\sfG$ and $\sfH$ be differentiable stratified groupoids. The following are
equivalent.
\begin{enumerate}[{\rm (1)}]
\item \label{IteMorEquiv}
      $\sfG$ and $\sfH$ are Morita equivalent as differentiable stratified groupoids.
\item \label{IteExWeakEquiv}
        There is a differentiable stratified groupoid $\mathsf{K}$ together with
        weak equivalences $h\co\mathsf{K}\to\sfG$ and $k\co\mathsf{K}\to\sfH$.
\end{enumerate}
\end{proposition}
\begin{proof}
Assume that \eqref{IteMorEquiv} holds true. Then there exists a differentiable
stratified space $Y$ together with open surjective
differentiable stratified submersions $f\co Y\to \sfG_0$ and $g\co Y\to \sfH_0$
such that $\sfG[Y]$ and $\sfH[Y]$ are isomorphic. Define the
groupoid $\mathsf{K}$ by setting $\mathsf{K}_0 = Y$ and
$\mathsf{K}_1 = \sfG_1 \fgtimes{(f\circ s_{\sfG},f\circ t_{\sfG})}
{(g\circ s_{\sfH},g\circ t_{\sfH})} \sfH_1$, and put
$s_{\mathsf{K}} := f\circ s_{\sfG} = g\circ s_{\sfH}$ and
$t_{\mathsf{K}} := f\circ t_{\sfG} = g\circ t_{\sfH}$.
The unit, inverse, and multiplication maps are defined component-wise. Then it is straightforward to
see that $\mathsf{K}$ is a differentiable stratified groupoid, where $\mathsf{K}_1$
is given the induced stratification. We define a morphism
$h \co \mathsf{K} \to \sfG$ by setting $h_0 = f$ and $h_1 = \operatorname{pr}_1$.
Then $h_0 = f$ is in fact surjective by hypothesis so that
$t_{\sfG}\circ\operatorname{pr}_1\co \sfG_1\fgtimes{s}{h_0}\mathsf{K}_0 \to \sfG_0$
is a composition of stratified surjective submersions, and $\mathsf{K}_1$
is a fibered product by construction, so $h$ is a weak equivalence. The weak
equivalence $k\co\mathsf{K}\to\sfH$ is defined identically.

Now assume that we have a weak equivalence $h\co\mathsf{K}\to\sfG$. Let
$Y = \sfG_1 \fgtimes{s_{\sfG}}{h_0} \mathsf{K}_0$ with the induced stratification,
and let $f = \operatorname{pr}_2 \co Y \to \mathsf{K}_0$ and
$g = s_{\sfG}\circ\operatorname{pr}_1 \co Y \to \sfG_0$. Then $f$ is obviously
an open surjective differentiable stratified submersion, while the same for $g$
follows from condition (ES) in Definition \ref{def:WeakEquiv} (along with applying
the inverse map). By definition, $\mathsf{K}[Y]$ and $\sfG[Y]$ both have $Y$ as object
space. The arrow spaces are given by
\[
    \mathsf{K}[Y]_1=
    (\sfG_1 \fgtimes{s_{\sfG}}{h_0} \mathsf{K}_0 \times
        \sfG_1 \fgtimes{s_{\sfG}}{h_0} \mathsf{K}_0)
        \fgtimes{(\operatorname{pr}_2,\operatorname{pr}_2)}
        {(t_{\mathsf{K}},s_{\mathsf{K}})}\mathsf{K}_1,
\]
and
\[
    \sfG[Y]_1 =
    (\sfG_1 \fgtimes{s_{\sfG}}{h_0} \mathsf{K}_0 \times
        \sfG_1 \fgtimes{s_{\sfG}}{h_0} \mathsf{K}_0)
        \fgtimes{(s_{\sfG}\circ\operatorname{pr}_1,s_{\sfG}\circ\operatorname{pr}_1)}
        {(t_{\mathsf{G}},s_{\mathsf{G}})}\mathsf{G}_1.
\]
We define an isomorphism $\mathsf{K}[Y] \to \sfG[Y]$ as the identity on objects
and by applying $h_1$ to the last factor on arrows. This is obviously a differentiable
stratified mapping. The fact that it is an isomorphism follows from condition (FF)
in Definition \ref{def:WeakEquiv}.
\end{proof}

We will need the following, whose proof is that of
\cite[Prop.~3.7]{PflPosTanGOSPLG} with minor modifications.

\begin{lemma}
\label{lem:SatWeakEquiv}
Let $\sfG$ be a proper differentiable stratified groupoid. Suppose $Y$ is a locally closed
differentiable stratified subspace of $\sfG_0$ such that $\sfG_{|Y}$ is a differentiable
stratified subgroupoid. Then the inclusion map $\iota\co\sfG_{|Y} \to \sfG_{|\Sat(Y)}$
is a weak equivalence, where
\[
  \Sat(Y) := \{ \obj{y} \in \sfG_0 \mid \obj{y} = t(\arr{g})
  \text{ for some $\arr{g}\in\sfG_1$ with $s(\arr{g}) \in Y$}\}
\]
denotes the saturation of $Y$.
\end{lemma}
\begin{proof}
First note that since $\sfG$ is proper and $Y$ locally closed in $\sfG_0$
the saturation $\Sat(Y) \subset \sfG_0$ is also locally closed.
Next observe  that $(\sfG_{|\Sat(Y)})_1 \fgtimes{s}{\iota_0} Y = \{ \arr{g}\in\sfG_1\mid s(\arr{g})\in Y\}$,
so that $t$ clearly restricts to this set as a map that is surjective onto $Y$. Moreover,
$t$ restricts to a stratified surjective submersion $(\sfG_{|Y})_1\to Y$ by hypothesis so that
(ES) is satisfied. The condition (FF) clearly follows from the definition of $\Sat(Y)$.
\end{proof}

With this, we make the following.

\begin{definition}
\label{def:LocTransStrat}
We say that a differentiable stratified groupoid $\sfG$ is
\emph{sliceable} if
for every $\obj{x}\in \sfG_0$ the conditions
(LT\ref{it:ExistenceSlice}) to (LT\ref{it:ExistenceBisections})
from Definition \ref{def:LocTransDiffGroupoid} are satisfied
and if in addition the following holds true:
\begin{enumerate}[{\rm (LT1)}]
\setcounter{enumi}{3}
\item
\label{it:StratifiedSlice}
The slice $Y_\obj{x}$ can be chosen in such a way that
$Y_\obj{x}$ is a differentiable stratified subspace of $U_\obj{x}$,
and such that $Y_\obj{x}$ has the form $Z_\obj{x} \times R_\obj{x}$,
where $Z_\obj{x}$ is a $\sfG_\obj{x}$-invariant stratified subspace of $Y_\obj{x}$
and $R_\obj{x}$ is the stratum through $\obj{x}$. Moreover, the isomorphism
between $\sfG_{|U_\obj{x}}$ and
$(O_\obj{x} \times O_\obj{x}) \times (\sfG_\obj{x} \ltimes Y_\obj{x})$
then becomes an isomorphism  of differentiable stratified groupoids,
where $O_\obj{x}$ carries the trivial stratification.
\end{enumerate}
\end{definition}

We will see that important examples of differentiable stratified groupoids are
sliceable differentiable stratified groupoids. Under additional hypotheses,
we have the following.

\begin{proposition}
\label{prop:LocTransStrat}
Let $\sfG$ be a sliceable differentiable stratified groupoid, and suppose
that for each $x \in \sfG_0$, the set $Y_\obj{x}$ as in
{\rm (LT\ref{it:ExistenceSlice})} of
Definition {\rm\ref{def:LocTransDiffGroupoid}} can be chosen so that on each stratum
of $Y_\obj{x}$ the $\sfG_\obj{x}$-orbit type is constant.
Then the assignment
\[
 |\sfG| \ni \sfG \obj{x} \mapsto \mathcal{S}_{\sfG x} = \pi(\mathcal{S}^0_\obj{x})
\]
defines a stratification of $|\sfG|$ with respect to which the orbit map $\pi\co\sfG_0\to|\sfG|$
is a stratified surjective submersion.
\end{proposition}\begin{proof}
Note that $|\sfG|$ is a differentiable space by Proposition \ref{prop:LocTransDiffOrbitSpace}.
Choose $U_\obj{x}$, $Y_\obj{x}$, etc. as in Definitions \ref{def:LocTransDiffGroupoid} and
\ref{def:LocTransStrat}, and note that we may shrink $U_\obj{x}$ if necessary to assume that
the stratification of $Y_\obj{x}$ consists of a finite number of strata. By Lemma
\ref{lem:SatWeakEquiv} and Proposition \ref{prop:MoritaHomeo}, the inclusion of
$Y_\obj{x}$ into its saturation in $\sfG_0$ induces an isomorphism of the differentiable spaces
$Y_\obj{x}/\sfG_\obj{x}$ and $|\sfG_{\Sat(Y_\obj{x})}|$. As $\Sat(Y_\obj{x})$ contains
$U_\obj{x}$ by definition, and as the orbit map is open by Proposition
\ref{prop:BasicPropTopGroupoids}(\ref{it:OpMap}), it follows that $Y_\obj{x}/\sfG_\obj{x}$ is
isomorphic as a differentiable space to an open neighborhood of $\sfG \obj{x}$ in $|\sfG|$.
Moreover, by (LT\ref{it:StratifiedSlice}), the embedding of $Y_\obj{x}$ into its saturation
in $\sfG_0$ preserves strata, so it is sufficient to show that the stratification of $Y_\obj{x}$
induces a stratification of the orbit space $Y_\obj{x}/\sfG_\obj{x}$.

Now, each stratum $S$ of $Y_\obj{x}$ is a smooth manifold with $\sfG_\obj{x}$-action so that
$S/\sfG_\obj{x}$ is stratified by $\sfG_\obj{x}$-orbit types. Since each $S$ has a single orbit type by
hypothesis, the stratification of $S/\sfG_\obj{x}$ is trivial, i.e. $S/\sfG_\obj{x}$ is a smooth submanifold
of $Y_\obj{x}/\sfG_\obj{x}$. As $Y_\obj{x}$ is assumed to have finitely many strata, the resulting
stratification of $Y_\obj{x}/\sfG_\obj{x}$ is clearly finite. As $\sfG_\obj{x}$ is compact, the orbit map
$Y_\obj{x} \to Y_\obj{x}/\sfG_\obj{x}$ is closed so that $\overline{(S/\sfG_\obj{x})} = \overline{S}/\sfG_\obj{x}$.
Therefore, if $S/\sfG_\obj{x} \cap \overline{S^\prime/\sfG_\obj{x}}\neq\emptyset$ for strata $S$ and $S^\prime$
of $Y_\obj{x}$, then $S\cap \overline{S^\prime}\neq\emptyset$, implying $S\subset\overline{S^\prime}$
and hence $S/\sfG_\obj{x} \subset \overline{S^\prime/\sfG_\obj{x}}$. The pieces $S/\sfG_\obj{x}$ therefore satisfy
the condition of frontier and  define a stratification of $Y_\obj{x}/\sfG_\obj{x}$.
\end{proof}

\begin{remark}
Note that the hypotheses of Proposition \ref{prop:LocTransStrat} need not be satisfied even in the
case of a proper Lie groupoid. For instance, when $G$ is a compact Lie group, a smooth $G$-manifold
equipped with the trivial stratification satisfies these hypotheses if and only if $G$ acts with
a single orbit type.
\end{remark}


\section{Examples of differentiable stratified groupoids}
\label{sec:StratDiffGpoidExamples}

Many examples of differentiable stratified groupoids arise naturally from
differentiable actions of Lie groupoids on stratified differentiable spaces.
Recall that if $\sfG$ is a topological groupoid and $X$ is a topological space,
an \emph{action} of $\sfG$ on $X$ is given by a continuous \emph{anchor map}
$\alpha\co X \to \sfG_0$ together with a continuous map
$\cdot\co\sfG_1 \fgtimes{s}{\alpha}  X \to X$ such that for all
$\obj{x} \in X$ and $\arr{g},\arr{h}\in\sfG_1$ with  $s(\arr{g}) = \alpha(\obj{x})$ and
$t(\arr{g}) = s(\arr{h})$ the relations
\[
  \alpha(\arr{g}\cdot \obj{x}) = t(\arr{g}), \quad
  \arr{h}\cdot(\arr{g}\cdot x) = (\arr{h}\arr{g})\cdot \obj{x}, \quad \text{and} \quad
  u(\alpha(\obj{x}))\cdot \obj{x} = \obj{x}
\]
hold true. As in the case of group actions, the \emph{$\sfG$-orbit of $\obj{x} \in X$} is defined to be
$\{ \arr{g}\cdot \obj{x} \mid s(\arr{g})= \obj{x} \}$. The \emph{translation groupoid} $\sfG\ltimes X$
associated to the action has object space $(\sfG\ltimes X)_0 = X$ and arrow space
$(\sfG\ltimes X)_1 = \sfG_1\fgtimes{s}{\alpha} X$. The structure maps are given by
\begin{equation}
\label{eq:StrucMapsTranslation}
\begin{split}
    & s_{\sfG\ltimes X} (\arr{g}, \obj{x}) = \obj{x},
    \quad
    t_{\sfG\ltimes X} (\arr{g}, \obj{x}) = \arr{g}\cdot \obj{x},
    \quad
    u_{\sfG\ltimes X} (\obj{x}) = (u\circ\alpha(\obj{x}), \obj{x}),
    \\
    & i_{\sfG\ltimes X} (\arr{g}, \obj{x}) = (\arr{g}^{-1}, \arr{g}\cdot \obj{x}),
    \quad
    m_{\sfG\ltimes X} \big( (\arr{h},\arr{g}\cdot x) ,(\arr{g}, \obj{x})\big)
    = (\arr{h}\arr{g}, \obj{x}).
\end{split}
\end{equation}
If $\sfG$ is a differentiable groupoid and $X$  a differentiable space, we say that
the action is \emph{differentiable} if $\alpha$ and $\cdot$ are morphisms of differentiable
space. In this case, the arrow space $(\sfG\ltimes X)_1 = \sfG_1\fgtimes{s}{\alpha} X$ inherits a
differentiable structure by \cite[Thm.~7.6]{NGonzalezSanchoBook}. Then, as each of
the structure maps of $\sfG\ltimes X$ is defined in terms of the structure maps of $\sfG$
and the differentiable maps $\alpha$ and $\cdot$, $\sfG\ltimes X$ is
a differentiable groupoid. For actions of differentiable stratified groupoids, one even has the following.

\begin{proposition}
\label{prop:StratTranslationGpoid}
Let $X$ be a differentiable stratified space that is topologically locally trivial. Further let
$\sfG$ be a differentiable stratified groupoid with $\sfG_0$ and $\sfG_1$
topologically locally trivial that acts differentiably on $X$ in such a way that the $\sfG$-orbit
of each $x \in X$ is a subset of the stratum containing $x$. Then the translation groupoid
$\sfG\ltimes X$ is a differentiable stratified groupoid.
\end{proposition}
\begin{proof}
Let $\mathcal{Z}$ denote the maximal decomposition of $X$, and endow
$(\sfG\ltimes X)_1 = \sfG_1\fgtimes{s}{\alpha} X$ with the stratification induced by
those of $\sfG_1$ and $X$. Then by Lemma \ref{lem:InducedStrat}, $(\sfG\ltimes X)_1$ is
a differentiable stratified space. By definition of the induced stratification, the
germ of the stratum of $(\sfG\ltimes X)_1$ at the point $(\arr{g}, x)$ is given by the
fibered product over $s$ and $\alpha$ of neighborhoods of strata in $\sfG_1$ and $X$.
Since $i$, $u$, $m$, and $\alpha$ are stratified mappings, the structure maps
$s_{\sfG\ltimes X}$, $t_{\sfG\ltimes X}$, $u_{\sfG\ltimes X}$, $i_{\sfG\ltimes X}$, and $m_{\sfG\ltimes X}$ defined
by Equation \eqref{eq:StrucMapsTranslation} are as well stratified
mappings, so that (DSG\ref{it:DSGStratSpaces}) and (DSG\ref{it:DSGStrucMaps}) are satisfied.
To see that $s_{\sfG\ltimes X}$ is a stratified surjective submersion, let $(\arr{g}, x) \in (\sfG\ltimes X)_1$.
Note that as $s$ is a stratified surjective submersion, the restriction of $s$ to the connected
component $R_\arr{g}$ of the stratum of $\sfG_1$ containing $\arr{g}$ is a surjective submersion
onto the connected component $S_{\alpha(x)}$ of the stratum of $\sfG_0$ containing $\alpha(x)$.
As $\alpha$ is a stratified mapping, it maps the connected component $P_x$ of the stratum of $X$
containing $x$ into $S_{\alpha(x)}$. Hence the restriction of $s_{\sfG\ltimes X}$  to
$R_\arr{g} \fgtimes{s}{\alpha} P_x$ is a surjective submersion onto $P_x$.
It follows that $s_{\sfG\ltimes X}$ is a stratified surjective submersion. The
proof for $t_{\sfG\ltimes X}$ is identical. In a similar fashion one shows that $u_{\sfG\ltimes X}$ is a stratified
embedding, because $\alpha$ is a stratified map and $u$ a stratified embedding, so (DSG\ref{it:DSGSubmersion})
is satisfied. Property (DSG\ref{it:DSGStratSourceFib})  follows from the definition of the induced
stratification, as the connected component of the stratum of $(\sfG\ltimes X)_1$ containing the arrow
$(\arr{g}, x)$ equals the fibered product $R_\arr{g} \fgtimes{s}{\alpha} P_x$   with  $R_\arr{g}$
and $P_x$ as above. Therefore, the germ at $(\arr{g}, x)$ of the set of points
$(\arr{h},y)\in (\sfG\ltimes X)_1$ such that $y \in P_x$ is contained in the germ
$[R_\arr{g} \fgtimes{s}{\alpha} P_x]_{(\arr{g}, x)}$ of the stratum through $(\arr{g}, x)$.

To verify (DSG\ref{it:DSGStratSection}), consider again $P_x \subset X$, the connected component of
the stratum through $x$. Since by assumption on $X$ the $\sfG$-orbit of each point
in $P_x$  is a subset of the stratum containing $P_x$, the saturation $P:= \Sat P_x$ has to be a union
of connected components of strata which do not intersect one another's closures by
topological local triviality, see Proposition \ref{prop:RestrictLie}. Assume that $\sigma \co  P \to  (\sfG\ltimes X )_1$
is a  bisection as in (DSG\ref{it:DSGStratSection}). Since $\alpha$ is a stratified mapping, $\alpha(P_x)$
is contained in a connected component $S$ of a stratum of $\sfG_0$. The saturation
$\Sat S$ of $S$ by $\sfG$ now is given by $t(s^{-1}(S))$ and hence is a union of connected
components of strata of $\sfG_0$ such that for each such connected component $S^\prime$,
we have $s^{-1}(S)\cap t^{-1}(S^\prime)\neq\emptyset$ by construction. Then by Proposition
\ref{prop:RestrictLie}, the restriction $\sfG_{|\Sat S}$ is a Lie groupoid. Moreover,
it follows from (DSG\ref{it:DSGStratSourceFib}) that
$\sigma(P ) \subset (\sfG_{| \Sat S}\ltimes P)_1$. Therefore, the map $L_\sigma$ described
in (DSG\ref{it:DSGStratSection}) is a left translation of the Lie groupoid
$\sfG_{|\Sat S}\ltimes P$, see \cite[page 22]{MackenzieGenThry}, implying in particular that
it is a diffeomorphism of the stratum $s_{\sfG\ltimes X}^{-1}(P)$ onto itself. Hence
(DSG\ref{it:DSGStratSection}) holds, completing the proof.
\end{proof}

A particularly important special case
appears when a  Lie group $G$ acts differentiably on a differentiable stratified space $X$ in
such a way that the action restricts to a smooth action on each stratum.
The resulting translation groupoid $G\ltimes X$ then is a differentiable stratified groupoid
by Proposition \ref{prop:StratTranslationGpoid}.
Many significant and naturally occuring examples of differentiable stratified groupoids are constructed
that way. In the following, we will provide  a few.

\begin{example}[Singular symplectic reduction]
\label{sympl-sing-red}
Suppose $(M, \omega)$ is a connected symplectic manifold equipped with a Hamiltonian $G$-action
with moment map $J\co M\to\mathfrak{g}^\ast$ such that $0$ is a singular value for $J$.
By \cite[Thm.~2.1]{SjamaarLerman}, the zero level set $J^{-1}(0)$ and the
\emph{symplectic quotient} $J^{-1}(0)/G$ are both stratified by orbit types.
Moreover, the corresponding stratifications are Whitney (b)-regular.
In addition, $J^{-1}(0)$ inherits from the ambient manifold the structure of
a differentiable stratified space, and the $G$-action preserves the
stratification; see \cite[Thm.~8.3.1]{OrtegaRatiu}. Therefore, the
singular symplectic quotient can be realized as the orbit space of the
differentiable stratified groupoid $G\ltimes J^{-1}(0)$.
\end{example}

\begin{example}[Lie Groupoid actions on manifolds with corners]
\emph{Manifolds with corners} and  \emph{manifolds with boundary} are in a natural way locally trivial differentiable stratified spaces, and even more
$\mathcal{C}^\infty$-cone spaces, see \cite[1.1.19 \& 3.10.3]{PflaumBook} and \cite{VanLeSombergVanzura}.
Compact Lie group actions on manifolds with corners have been considered,
e.g.~in \cite{MargalefOuterelo, MelrosePDO}. By Proposition \ref{prop:StratTranslationGpoid},
the corresponding translation groupoids are differential stratified groupoids under
mild hypotheses.
\label{ex-act-madf-corn}
\end{example}

\begin{example}[Semialgebraic actions]
\label{ex-semialgebraic}
Recall that a \emph{semialgebraic set} is a locally closed subset of $\R^n$ locally given by the solution of a
finite collection of polynomial equations and inequalities, see \cite{ShiotaBook}.
Semialgebraic sets are differentiable stratified spaces in a natural way, since they admit a minimal
Whitney (b)-regular, and hence topologically locally trivial, stratification into semialgebraic manifolds. If a compact
Lie group acts on a semialgebraic set and preserves this stratification, the resulting
translation groupoid is again a stratified differentiable groupoid by
Proposition \ref{prop:StratTranslationGpoid}. Lie group actions on semialgebraic sets
have been considered in \cite{ChoiParkSuh,ParkOrbitTypes,ParkSuhLinearEmbed}.
\label{ex-sem-actn}
\end{example}

\begin{example}[Transverse cotangent bundle]
\label{ex-transv-cot-space}
Let $G$ be a compact Lie group and $M$ a $G$-manifold. The \emph{transverse cotangent bundle}
$T_G^\ast M$ is the subspace of the cotangent bundle $T^\ast M$ consisting of elements that
are conormal to the $G$-orbits in $M$. The transverse cotangent bundle appears in the study
of transversally elliptic $G$-invariant pseudodifferential operators on $M$,
see \cite{DeConcProcesiVergneCohom,ParadanKGTGM,ParadanVergneIndex}.
The action of $G$ on $M$ induces an action of $G$ on
$T_G^* M$. It is not difficult to show that the stratification of $M$ by orbit
types induces a stratification of $T_G^\ast M$ that is compatible with the smooth
structure $T_G^\ast M$ inherits as a subset of $T^\ast M$. Hence the corresponding
translation groupoid is a differentiable stratified groupoid.

If $\sfG$ is a proper Lie groupoid, the \emph{transverse cotangent bundle} $T_\sfG^\ast \sfG_0 \subset T^\ast \sfG_0$
can be defined in a similar fashion as the subspace of all $\xi \in T^\ast \sfG_0$ such that
$\langle \xi , v \rangle =0$ for all $v \in T_\obj{x} \calO_\obj{x}$; here, $\obj{x} \in \sfG_0$ with
$\xi \in T_{\obj{x}}^\ast \sfG_0$ and $\calO_\obj{x}$ is the orbit through $\obj{x}$. By the slice theorem
for groupoids as stated in \cite[Cor.~3.11]{PflPosTanGOSPLG} (also recalled in Section \ref{subsec:InertiaGpoidStrat}
below), it follows that in a neighborhood of a point $\obj{x} \in \sfG_0$ the transverse cotangent bundle $T_\sfG^\ast \sfG_0 $ is isomorphic to  the
exterior tensor product of the transverse cotangent bundle $T^\ast_{\sfG_\obj{x}} Y_{\obj{x}}$ of a slice $Y_\obj{x}$
through $\obj{x}$ with the cotangent bundle $T^\ast O$ of an open connected neighborhood $O$ of $\obj{x}$
in the orbit through $\obj{x}$.
\end{example}

\begin{example}[Singular riemannian foliations]
\label{ex-sing-Riem-foliat}
Recall from \cite[Section 6.1]{MolinoBook} that
a singular riemannian foliation is a pair $(M, \mathcal{F})$ where $M$ is a smooth, connected manifold and $\mathcal{F}$
is a partition of $M$ into connected, immersed submanifolds called \emph{leaves} such that the module of smooth vector
fields on $M$ that are tangent to the leaves is transitive on each leaf
and there is a riemannian metric on $M$ with respect to which every
geodesic that is normal to a leaf is normal to every leaf it intersects.
A singular riemannian foliation is an example of a singular Stefan--Sussmann foliation \cite{StefanFoliatSing,SussmannOrbitsVFldsDistrib}. By
\cite[Section 6.2]{MolinoBook}, $M$ is stratified by unions of leaves of the same dimension; see
also \cite[Sections 1.2--3]{RoyoPrietoSAWSingFoil} and \cite{BoualemMolino}.

Suppose $(M, \mathcal{F})$ is \emph{almost regular}, meaning that the union of leaves of
maximal dimension $k$ is an open, dense subset of $M$. Suppose further that the foliation
$\mathcal{F}$ can be defined by a Lie algebroid of dimension $k$ in the sense of \cite[page 472]{DebordHolonomSingFoliat},
i.e.~the leaves of $\mathcal{F}$ are the maximal connected integral manifolds of the distribution on $M$ induced by a Lie algebroid over $M$.
In \cite{DebordHolonomSingFoliat},
the holonomy groupoid $\sfG$ of $(M,\mathcal{F})$ is constructed as a Lie groupoid with
object space $\sfG_0 = M$ and such that the orbits of $\sfG$ correspond to the leaves $\mathcal{F}$;
see also \cite{PradinesSingFoliat}. Giving $M$ the stratification by leaves of the same dimension
described above and $\sfG_1$ the stratification given by the pullback of this stratification via $s$,
it is easy to see that $\sfG$ has, along with its structure as a Lie groupoid, an alternate structure
as a differential stratified groupoid which is finer than the original one.
These two structures coincide if and only if the foliation
$\mathcal{F}$ is regular \cite[page 190 and Section 1.1]{MolinoBook}, i.e.~all leaves have the same dimension.

Note that the holonomy groupoid of a general singular Stefan--Sussmann foliation was constructed as a
topological groupoid in \cite{AndroulidakisSkandalisFoliat} and coincides with the holonomy groupoid
of \cite{DebordHolonomSingFoliat} when the latter is defined. Other hypotheses under which
this holonomy groupoid naturally has the structure of a differentiable stratified groupoid
are not yet clear.

\end{example}

\section{The  algebroid of a differentiable stratified groupoid}
\label{sec:algebroid}
Given a reduced structurally (b)-regular differentiable stratified groupoid $\sfG$ with $\calS^i$,
$i=0,1$ the decomposition of $\sfG_i$ into its strata we obtain the so-called
stratified tangent bundles
\[
  \stratan \sfG_0 := \bigcup_{S \in \calS^0} TS
  \text{ and }
  \stratan\sfG_1 :=\bigcup_{R \in \calS^1} TR =\bigcup_{S \in \calS^0}TS^1,
  \text{ where $S^1 := s^{-1} (S)$ for $S\in \calS^0$}.
\]
Since, by assumption, the stratifications of $\sfG_0$ and $\sfG_1$ are Whitney (b)-regular,
hence (a)-regular, the stratified tangent bundles
$\stratan\sfG_0 $ and $\stratan\sfG_1 $ inherit the
structures of differentiable stratified spaces by \cite[Thm.~2.1.2]{PflaumBook}.
Moreover, we have tangent maps
$Ts \co \stratan \sfG_1 \to \stratan \sfG_0 $ and
$Tt \co \stratan \sfG_1 \to \stratan \sfG_0 $.
Now we can define the algebroid of $\sfG$.
\begin{definition}
\label{def:Algebroid}
  Given a structurally (b)-regular differentiable stratified groupoid $\sfG$,
  the \emph{differentiable stratified algebroid} of $\sfG$ is defined as the space
  \begin{equation}
  \label{eq:AlgebroidDef}
    \sfA := \bigcup_{S \in \calS^0} \sfA_S ,
  \end{equation}
  where $\sfA_S := u^*_{|S}\ker T_{|S^1}s$ denotes the Lie algebroid
  of the Lie groupoid $\sfG_{|S}$. In other words, $\sfA$ can be identified
  with $u^*\ker Ts$, the restriction of $\ker Ts$ to $\sfG_0$. We define the
  \emph{anchor map} of the algebroid $\sfA$ to be the map
  \[
    \varrho\co  \sfA \to \stratan\sfG_0 , \quad v \mapsto Tt (v).
  \]
\end{definition}

\begin{proposition}
Let $\sfG$ be a proper reduced structurally (b)-regular differentiable stratified groupoid.
Then the differentiable stratified algebroid $\sfA$ of $\sfG$ is a reduced
differentiable stratified space, where the
differentiable structure is that inherited from $\stratan\sfG_1$ and the
stratification is that induced by the decomposition in
Equation \eqref{eq:AlgebroidDef}.
The anchor map $\varrho\co  \sfA \to \stratan\sfG_0$ is a differentiable
stratified map. If $\sfG$ is in addition a sliceable differentiable
stratified groupoid, then there is a bracket
\[
  [ - , - ] \co \Gamma^\infty (\sfG_0, \sfA) \times \Gamma^\infty (\sfG_0, \sfA)
  \to \Gamma^\infty (\sfG_0, \sfA)
\]
which satisfies $\varrho \circ [ - , - ] =  [ \varrho (-)  , \varrho (-) ]$
and which coincides over each stratum $S \in \calS^0$ with the
bracket of the Lie algebroid $ \sfA_S$.
\end{proposition}
By $\Gamma^\infty (\sfG_0, \sfA)$, we mean the sections of the stratified vector bundle
$\sfA$ that are as well morphisms of reduced differentiable spaces.
\begin{proof}
Since  $\sfG_1$ is locally compact, $u(\sfG_0)$ is locally closed in $\sfG_1$ by
\cite[Prop.~2.5]{TuNonHausGpoid}. For each stratum $S$ of $\sfG_0$
the source map $s$ restricts to $S_1$ as a submersion which implies
that $\ker T_{|S_1}s$ is a subbundle, hence closed subset, of $TS_1$.
Hence, if $U$ is an open subset of $\sfG_1$ intersecting finitely many strata,
$\sfA \cap\stratan U$ is a finite union of locally closed sets, therefore
locally closed itself. Then $\sfA$ inherits the structure of a reduced
differentiable space from that of $\stratan \sfG_1$, see
\cite[Ex.~3.21]{NGonzalezSanchoBook}.

Since each $\sfA_S$ is a subbundle of the restriction of $TS_1$
to the smooth submanifold $S_1 \cap u(\sfG_0)$ of $\sfG_1$, one concludes that each $\sfA_S$ is a smooth
submanifold of $\sfA$.
Note that the projection $\pi\co\stratan\sfG_1\to \sfG_1$ is clearly open
as its restriction to an element of $\calS^1$ is a bundle map.
The fact that $\calS^1$ is locally finite  implies then
that $(\sfA_S)_{S\in \calS^0}$ is a locally finite decomposition of $\sfA$.

To verify the condition of frontier, suppose
$\sfA_S \cap \overline{\sfA_{S^\prime}}\neq\emptyset$ for $S, S^\prime\in\calS^0$,
and let $S_1 = s^{-1}(S)$ and $S_1^\prime = s^{-1}(S^\prime)$.
By \cite[Thm.~2.1.2]{PflaumBook}, $\pi\co\stratan\sfG_1\to \sfG_1$ is a
topological projection,  hence $S_1 \subset\overline{S_1^\prime}$
and $T S_1 \subset \overline{T S_1^\prime}$ follow.

Choose a tangent vector $v \in \ker T_{|S_1}s$ and assume for simplicity that
$v$ is a unit vector. Let $\obj{x} \in S$ be the footpoint of $v$, i.e.~$\obj{x}$
is the unique object such that $\pi(v) = u(\obj{x})$.
Choosing a singular chart for $\sfG_1$ at $u(\obj{x})$,
we may reduce to the case where $S_1$ and $S_1^\prime$ are closed subsets of
$\R^n$. Let $\widetilde{s}$ denote a smooth function from $\R^n$ into a
singular chart of $\sfG_0$ at  $\obj{x}$ that extends $s$. Note that
$\widetilde{s}$ may not be a submersion. Let $p(t)\co[-1,1]\to S_1$ be a
smooth path in $S_1$ with $p(0) = u(\obj{x})$ and tangent
$p^\prime(0) = v$, and put $\alpha_i = p(1/i)$ for $i\in\N$.
As $\alpha_1 \in S_1 \subset \overline{S_1^\prime}$, we may choose a sequence
$(\beta_{1,j})_{j\in \N} \subset S_1^\prime$ such that
$\lim_{j\to\infty} \beta_{1,j} = \alpha_1$. We then set
$\obj{x}_j=s(\beta_{1,j}) \in S^\prime$ for each $j$. By  continuity of $u$ and $s$,
we have $\lim_{j\to\infty}\obj{x}_j = \obj{x}$ and $\lim_{j\to\infty}u(\obj{x}_j) = u(\obj{x})$.
As $s_{|S_1^\prime}$ is a submersion, each $s^{-1}(\obj{x}_j)$ is a closed submanifold of $\R^n$.
Then the intersection of each $s^{-1}(\obj{x}_j)$ with a closed ball in $\R^n$ is compact.
Hence we may define, for each $i > 1$ and each $j \geq 1$, the element $\beta_{i,j}$ of $S_1^\prime$
to be the point on the connected component of $s^{-1}(\obj{x}_j)$ containing $u(\obj{x}_j)$
intersected with a closed annulus of external radius $1/j$ about $u(\obj{x}_j)$
that minimizes the distance to $\alpha_i$. In particular, for each $j$,
$\lim_{i\to\infty}\beta_{i,j}  = u(\obj{x}_j)$, and for $i$ sufficiently large,
$\lim_{j\to\infty}\beta_{i,j}  = \alpha_i$; if necessary, we restrict to a
subsequence in the $j$ direction so that this intersection is not empty.
With this, we set
\[
    v_j := \lim_{i\to\infty} \frac{\beta_{i,j} - u(\obj{x}_j)}{\|\beta_{i,j} - u(\obj{x}_j)\|},
\]
and then $v_j \in T_{\obj{x}_j} S_1^\prime$ for each $j$. Moreover, as $\beta_{i,j}\in s^{-1}(\obj{x}_j)$,
it follows that $v_j \in \ker T_{|S_1^\prime}s$. However,
\[
    \lim_{j\to\infty} v_j = \lim_{i\to\infty}\lim_{j\to\infty}
        \frac{\beta_{i,j} - u(\obj{x}_j)}{\|\beta_{i,j} - u(\obj{x}_j)\|}
        =   \lim_{i\to\infty}
        \frac{\alpha_i - u(\obj{x})}{\|\alpha_i - u(\obj{x})\|} = v
\]
so that $v \in  \overline{\ker T_{|S_1^\prime}s}$. It follows that the decomposition
of $\sfA$ into the $\sfA_S$ satisfies the condition of frontier.

The anchor map is a differentiable stratified morphism by
Eq.~\eqref{Ite5} in Lemma \ref{lem:DSGTrivConsequences} and because
$t$ is a stratified map.

Now, let $X, Y \in \Gamma^\infty (\sfG_0, \sfA)$.
We define $[X, Y]$ locally as follows. For $\obj{x}\in \sfG_0$, choose
a neighborhood $U_{\obj{x}}$ with slice $Y_{\obj{x}}$ as in
(LT\ref{it:StratifiedSlice}) of Definition \ref{def:LocTransStrat}.
By shrinking $U_{\obj{x}}$ if necessary so that we may extend the $G_{\obj{x}}$-action
to the closure of $Y_{\obj{x}}$ in $\sfG_0$, we have by the Equivariant Embedding
Theorem of \cite[Section 11.2]{NGonzalezSanchoBook} a $G_{\obj{x}}$-equivariant
closed embedding $Y_{\obj{x}}\to E$ where $E$ is a finite-dimensional representation
of $G_{\obj{x}}$. Let $f_0\co U_{\obj{x}} = Y_{\obj{x}}\times O_{\obj{x}}\to E\times O_{\obj{x}}$
denote the corresponding embedding which is the identity on the $O_{\obj{x}}$-factor. By Corollary
\ref{cor:OrbitMnfld}, $O_{\obj{x}}$ is a smooth manifold so that that $G_{\obj{x}} \ltimes E$ and
the pair groupoid $O_{\obj{x}}\times O_{\obj{x}}$ are both Lie groupoids. Hence, $f_0$
is the map on objects of an embedding of $\sfG_{|U_\obj{x}}$ into the Lie groupoid
$\sfH:=(G_{\obj{x}} \ltimes E)\times(O_{\obj{x}}\times O_{\obj{x}})$. Roughly speaking,
we will define $[X,Y]$ following the definition of the bracket for the Lie algebroid
of a Lie groupoid taking advantage of the structure of $\sfH$; see
\cite[Section 3.2]{Marle}, and note that we use right-invariant vector fields as
in \cite[Remark 3.2.4]{Marle}. By shrinking $O_{\obj{x}}$
if necessary, we may assume that $f_0$ maps $U_{\obj{x}}$ into an open subset $V$ of $\sfH_0$
small enough so that the Lie algebroid $\sfB$ of $\sfH$, a vector bundle in the usual sense and
hence locally trivial, restricts to a trivial vector bundle over $V$. Fixing a trivialization of
$\sfB_{|V}$, the section $u_{\sfH}^\ast (f_1\circ u)_\ast X$, which is defined only on
$f_0(U_{\obj{x}})$, can be described by coordinate functions with respect to this trivialization,
which are smooth by construction.
Hence each coordinate function extends to a smooth function on a neighborhood $W$ of
$Y_{\obj{x}}\times O_{\obj{x}}$ in $\sfH_0$ yielding a smooth section of $\sfB$ on $W$.
We let $\hat{X}$ denote the unique right-invariant vector field on $\Sat(u_{\sfH}(W))$ in $\sfH_1$
generated by this extension and define $\hat{Y}$ similarly. Then we define $[X,Y]$ to be
$(f_1\circ u)^\ast[\hat{X},\hat{Y}]_{\sfH}\in \Gamma^\infty (U_{\obj{x}}, \sfA)$,
where $[ - , - ]_{\sfH}$ denotes the ordinary bracket of vector fields on $\sfH_1$.
Note that $[X,Y]$ is a morphism of differentiable spaces as it is defined as the pullback
of the bracket $[\hat{X},\hat{Y}]_{\sfH}$ on the manifold $\sfH_1$, which is smooth.
That $[X,Y]$ is a section of $\sfA$ over $U_{\obj{x}}$ follows from the fact that
$[\hat{X},\hat{Y}]_{\sfH}$ is a section of $\sfB$ and $s_{\sfH}\circ f_1 = f_0\circ s_{\sfG}$.

Now, if $S \in \calS^0$ is a stratum that intersects $U_{\obj{x}}$, then by Proposition
\ref{prop:RestrictLie}, $\sfG_{|S\cap U_{\obj{x}}}$ is a Lie groupoid. The morphism $f$
restricts to an isomorphism of $\sfG_{|S\cap U_{\obj{x}}}$ with a Lie subgroupoid of
$\sfH$ so that the unique right-invariant
vector field on the object space of $\sfG_{|S\cap U_{\obj{x}}}$ generated by
$X_{|S\cap U_{\obj{x}}}$ pushes forward via $f_1$ to the restriction of $\hat{X}$, and
similarly for $\hat{Y}$. Hence, on $S\cap U_{\obj{x}}$, $[X,Y]$ restricts to the usual
bracket of $X_{|S\cap U_{\obj{x}}}$ and $Y_{|S\cap U_{\obj{x}}}$ in the Lie algebroid
of $\sfG_{|S\cap U_{\obj{x}}}$. Therefore, $[ - , - ]$ coincides on each stratum with the
bracket of the corresponding Lie algebroid. In particular, this demonstrates that the value
of $[X,Y]$ on each stratum, and hence on $\sfG_0$, does not depend on the various choices made.
Moreover, as the anchor map $\varrho$ restricts to each stratum as the anchor map of the corresponding
Lie algebroid by definition, the relation $\varrho \circ [ - , - ] =  [ \varrho (-)  , \varrho (-) ]$
follows. This completes the proof.
\end{proof}


\section{A de Rham theorem}
\label{sec:Invariants}

In this section, we prove a de Rham theorem for sliceable proper reduced
differentiable stratified groupoids which satisfy
a certain local contractibility condition. As we will see below, this
condition is satisfied in a number of the examples we have considered.

\emph{Throughout this section let $\sfG$ be a proper reduced structurally (b)-regular differentiable
stratified groupoid}. Let $\Omega^\bullet$ be the sheaf complex
of abstract forms on $\sfG_0$ as constructed in Section \ref{SectionDiffForms}.
Recall from \cite[Def.~8.1]{PflPosTanGOSPLG}
that a differential form on the object space of a proper Lie groupoid is called
\emph{basic} if contraction with any smooth section of the Lie algebroid
vanishes and if the form is invariant under the conormal action of the Lie groupoid.

\begin{definition}
  Let $U \subset |\sfG|$ be an open subset of the orbit space of $\sfG$, and $U_0$ its
  preimage under the canonical projection $\pi \co \sfG_0 \to |\sfG|$. One calls an abstract $k$-form
  $\omega \in \Omega^k (U_0)$  $\sfG$-\emph{horizontal} or simply \emph{horizontal},
  if for each smooth section $\xi \co U_0 \to \sfA$ of the differentiable stratified algebroid $\sfA$ of $\sfG$
  given by Definition \ref{def:Algebroid},
  the stratawise contracted form $(\varrho \circ \xi) \lrcorner \omega$ vanishes.
\end{definition}

Next let $S\subset \sfG_0$ be a (relatively closed and open) component of a stratum of $\sfG_0$ such that the
projection $\pi (S) $ to the orbit space is connected.
For $\obj{x}\in\sfG_0$ consider the Zariski tangent space
$\zartan_x\sfG_0$; see Appendix \ref{SectionTangentSpace}.
Note that $\zartan_x\sfG_0$ does in general not coincide with $\stratan_x\sfG_0$, but that the latter is always
contained in the former.

\begin{lemma}
  Let $\sfG$ be sliceable. Then the space
  $\zartan_{|S} \sfG_0 := \bigcup_{\obj{x}\in S} \zartan_x\sfG_0$ naturally carries the structure of
  a smooth vector bundle over $S$.
\end{lemma}
\begin{proof}
  Given $\obj{x}\in S$ choose a trivializing neighborhood $U_\obj{x}$ and a groupoid
  slice $Y_\obj{x}$ with a $\sfG_\obj{x}$-action such that there exists an isomorphism
  of differentiable groupoids
  $\sfG_{|U_\obj{x}} \to ( O_\obj{x} \times  O_\obj{x}) \times (\sfG_\obj{x} \ltimes Y_\obj{x})$,
  where $O_\obj{x}$ is an open contractible neighborhood of $\obj{x}$ in its orbit.
  Such $U_\obj{x}$, $Y_\obj{x}$, and $O_\obj{x}$ exist according to (LT\ref{it:ExistenceSlice}).
  After possibly shrinking these data, we can assume by (LT\ref{it:StratifiedSlice}) that the slice
  $Y_\obj{x}$ is stratified and is isomorphic as a differentiable stratified space
  to the product space $Z_\obj{x} \times R_\obj{x}$, where $Z_\obj{x}  \subset Y_\obj{x}$
  is a $\sfG_\obj{x}$-invariant subspace, and $R_\obj{x} \subset Y_\obj{x}$
  is the stratum through $\obj{x}$.
  Hence $\bigcup_{\obj{y}\in O_\obj{x} \times R_\obj{x}} \zartan_{\obj{y}}\sfG_0$ is a vector
  bundle isomorphic to $T (O_\obj{x} \times R_\obj{x}) \times \zartan_\obj{x}Z_\obj{x}$.
  But the set $O_\obj{x} \times R_\obj{x}$ contains $\obj{x}$ and is relatively open
  in the stratum $S$. This means that, locally, $\zartan_{|S} \sfG_0$ is a vector bundle.
  By construction, the transition maps arise from the local isomorphisms $\sfG_{U_\obj{x}} \to ( O_\obj{x} \times  O_\obj{x})
   \times (\sfG_\obj{x} \ltimes Y_\obj{x})$, and are also vector bundle isomorphisms, hence the claim follows.
\end{proof}


We can now define basic forms on $\sfG$. Note that we will apply the
pull-back morphisms constructed in Appendix \ref{SectionDiffForms}.

\begin{definition}
 Let $U \subset |\sfG|$ be an open subset of the orbit space, and $U_0 := \pi^{-1} (U)$.
 One calls an abstract $k$-form $\omega \in \Omega^k (U_0)$  $\sfG$-\emph{basic} or simply \emph{basic}
 if it is horizontal and if for every $\obj{x} \in U_0$ and every smooth bisection
 $\sigma : U_\obj{x}  \to \sfG$ defined on an open neighborhood
 $U_\obj{x} \subset U_0$ of $\obj{x}$ the equality
 $(t\circ\sigma)^* \omega = \omega_{| U_\obj{x}}$ holds true.
 Let $\Omega^k_\textup{basic} (U)$ denote the set of $\sfG$-basic $k$-forms on $U$.
\end{definition}


Observe that by definition, $\Omega^0_\textup{basic} (U)$ coincides with the space of smooth functions
on $U_0$ which are invariant under the $\sfG$-action, hence
$\Omega^0_\textup{basic} (U)$ can be naturally identified with $\calC^\infty (U)$.
Moreover, we will now show that the spaces  $\Omega^k_\textup{basic} (U)$, where $U$ runs through the
open sets of $|\sfG|$, form the space of sections of a sheaf on $|\sfG|$ which will
be denoted by $\Omega^k_\textup{basic}$.

\begin{proposition}
  The sheaves of basic $k$-forms $\Omega^k_\textup{basic}$ are fine. Moreover, the
  exterior differential on  $\Omega^\bullet$ descends to a differential $d$ on
  $\Omega^\bullet_\textup{basic}$ turning $\big( \Omega^\bullet_\textup{basic},d\big)$
  into a sheaf of commutative differential graded algebras over the orbit space $|\sfG|$.
\end{proposition}
\begin{proof}
  By construction the sheaf $\Omega^k_\textup{basic}$ is a  $\calC^\infty_{|\sfG|}$-module.
  It remains to show that $d$ descends to $\Omega^\bullet_\textup{basic}$. But that follows
  from the fact that $d$ commutes with the pull-back morphism $(t\circ\sigma)^*$
  for every bisection $\sigma : U_\obj{x}  \to \sfG$ defined on an open neighborhood
  $U_\obj{x}$ of $\obj{x} \in \sfG_0$.
\end{proof}
The next result will be needed for the proof of a Poincar\'e Lemma for basic forms.

\begin{proposition}
\label{proposition-extension-basic-forms}
  Assume $\sfG$ is sliceable.
  Let $V$ be an open subset of $\sfG_0$ and let $\omega \in \Omega^k (V)$
  be an abstract $k$-form which is $\sfG_{|V}$-basic.
  Then there exists a unique form $\widehat{\omega}\in \Omega^k \big( \Sat (V) \big)$
  which is $\sfG$-basic and whose restriction to $V$ coincides with $\omega$. We call
  $\widehat{\omega}$ the \emph{basic extension} of $\omega$.
\end{proposition}

\begin{proof}
  Let $\obj{x} \in \Sat (V)$, and choose a bisection $\sigma : U_\obj{x} \to \sfG_1$ defined on an open neighborhood of $\obj{x}$ such that
  $(t \circ \sigma)  (U_\obj{x}) \subset V$. Then put $\omega_{U_\obj{x}} := (t\circ \sigma)^* (\omega )$.
  If $\eta : U_\obj{x} \to \sfG_1$ is another bisection with $U_{\obj{z}} := (t \circ \eta)  (U_\obj{x}) \subset V$,
  where $\obj{z} := (t \circ \eta) (\obj{x})$, then  the bisection $\mu\co U_{\obj{z}} \to \sfG_1$ defined by
  $$
   \mu (\obj{y}) = \sigma \big( (t \circ \eta)^{-1} (\obj{y}) \big) \cdot \eta^{-1} \big( (t \circ \eta) (\obj{y}) \big)
   \quad\text{for $\obj{y} \in U_{\obj{z}}$}
  $$
  starts and ends in $V$. Hence, by assumption, $(t\circ \mu)^* \omega = \omega_{|U_{\obj{z}}}$.
  But $t \circ \mu = (t \circ \sigma) \circ (t \circ \eta)^{-1}$, which implies
  \[
    (t \circ \eta)^* \omega = (t \circ \eta)^* (t\circ \mu)^* \omega =
    (t \circ \sigma)^* \omega.
  \]
  This shows that  $\omega_{U_\obj{x}}$ does not depend on the particular choice of the bisection
  $\sigma : U_\obj{x} \to t^{-1} (V) $. An analogous argument proves that
  for points $\obj{x}, \obj{y} \in \Sat (V)$ the forms  $\omega_{U_\obj{x}}$ and $\omega_{U_\obj{y}}$ coincide over
  $U_\obj{x} \cap U_\obj{y}$. Let $\widehat{\omega} \in \Omega^k \big(\Sat (V)\big)$ be the abstract $k$-form such that
  $\widehat{\omega}_{|U_\obj{x}} = \omega_{U_\obj{x}}$ for all $ \obj{x} \in \Sat (V)$. Obviously,
  $\widehat{\omega}$ is horizontal and basic by construction.
  Uniqueness of  $\widehat{\omega}$ is clear since the form has to be basic.
\end{proof}

Since the kernel of the sheaf morphism $d: \Omega^0_\textup{basic} \to \Omega^1_\textup{basic}$
can be naturally identified with $\R_{|\sfG|}$, the sheaf of locally constant real-valued functions on $|\sfG|$,
we obtain a  sheaf complex
\[
 0 \longrightarrow \R_{|\sfG|} \longrightarrow \Omega^\bullet_\textup{basic}
\]
which is exact at $\R_{|\sfG|}$ and at $\Omega^0_\textup{basic} $.
Now observe that the orbit space $|\sfG|$ is paracompact, locally contractible and locally path connected
by \cite{PflPosTanGOSPLG}. Since the sheaves $\Omega^k_\textup{basic} $ are fine, the following is
a consequence of \cite[Sec.~3.9]{GodTATF}.

\begin{theorem}
  Let $\sfG$ be a proper reduced structurally (b)-regular differentiable stratified groupoid.
  If the sheaf complex $\big( \Omega^\bullet_\textup{basic} , d \big)$ of basic forms on $\sfG$ is exact,
  then the \emph{basic cohomology} $H^\bullet_\textup{basic} (\sfG) :=
  H^\bullet \big( \Omega^\bullet_\textup{basic} (|\sfG|) \big)$ of $\sfG$
  naturally coincides with the real singular cohomology of $|\sfG|$.
\end{theorem}

In the remainder of this section we will show that the sheaf complex of basic
forms on a groupoid satisfies a Poincar\'e Lemma, i.e. is exact, if a certain local contractibility condition is satisfied.
Before stating the local contractibility condition, recall that by
Proposition \ref{prop:equivariant-embedding_zariski-space} every $\sfG$-slice can be, possibly after
shrinking, equivariantly embedded into the Zariski tangent space around the fixed point.

\begin{definition}\label{def:local-contractibility}
  Let $\sfG$ be a sliceable proper differentiable stratified groupoid. We then say that $\sfG$ fulfills the
  \emph{local contractibility condition} if the following holds true.
  \begin{itemize}
  \item[{\rm (LC)}] \label{it:LocContract}
    For each $\obj{x}\in\sfG_0$ there exists a groupoid slice $Y_\obj{x}$ as in
    (LT\ref{it:StratifiedSlice}), a linear $\sfG_\obj{x}$-action on some $\R^n$
    together with a singular $\sfG_\obj{x}$-equivariant chart
    $\iota: Z_\obj{x} \to\widetilde{V_\obj{x}}\subset \R^n$, $y \mapsto \widetilde{y}$,
    and a smooth homotopy
    $h \co \widetilde{V_\obj{x}}\times[0,1]\to \widetilde{V_\obj{x}}$ having the
    following properties:
    \begin{enumerate}
    \item
      The chart $\iota$ maps $\obj{x}$ to $0$ and the stratum $R_\obj{x}$ through $\obj{x}$ to the subspace of
      $\R^n$ fixed by $\sfG_\obj{x}$. Moreover,
      $\widetilde{V_\obj{x}}$ is an open neighborhood of $0$ in $\R^n$,
      and $\widetilde{Y}_\obj{x} := \iota( Y_\obj{x})$ is relatively closed in
      $\widetilde{V_x}$.
    \item
      One has $\im h_0 = \{ 0 \}$ and $h_1 = \id_{\widetilde{V}_\obj{x}}$.
    \item
      The homotopy $h$ is a homotopy along  $\widetilde{Y}_\obj{x}$ which means that
      $h (\widetilde{\obj{y}},t) \in \widetilde{Y}_\obj{x}$ for all $\obj{y}\in Z_\obj{x}$
      and $t\in [0,1]$.
    \item
      The homotopy $h$ is $G_\obj{x}$-equivariant.
    \item
      The homotopy $h$ preserves the stratification which means for each $\obj{y}\in Y_\obj{x}$
      and $ t \in (0,1]$ the points $\iota^{-1}  h(\widetilde{\obj{y}},t)$
      and $\obj{y}$ are in the same stratum.
    \end{enumerate}
\end{itemize}
\end{definition}

\begin{example}
Let $G$ be a compact Lie group and $M$ is a $G$-manifold. The transverse cotangent bundle described in
Example \ref{ex-transv-cot-space} satisfies the local contractibility condition of Definition \ref{def:local-contractibility}.
To verify this, observe first that $T_G^\ast M$ is equivariantly embedded in the $G$-manifold $T^\ast M$, and that
$\R_{\geq 0}$ acts on  $T^\ast M$ by fiberwise homotheties. Since the transverse
cotangent bundle is invariant under these homotheties, it  contracts
smoothly to the $0$-section which is diffeomorphic to $M$. But $M$ is a smooth manifold, so is clearly locally
smoothly contractible. Hence $T_G^\ast M$  is locally smoothly contractible as well.
\end{example}

\begin{example}
Let $(M, \omega)$ be a connected symplectic Hamiltonian $G$-manifold with $0$ a singular value of the
moment map $J\co M\to\mathfrak{g}^\ast$ so that the singular symplectic quotient is the orbit space
of the differentiable stratified groupoid $G\ltimes J^{-1}(0)$; see Example \ref{sympl-sing-red}.
By the Marle--Guillemin--Sternberg
normal form of the moment map \cite[Thm.~7.5.5]{OrtegaRatiu},
the contractibility condition (LC) in  Definition \ref{it:LocContract})
is satisfied for the zero level set $J^{-1}(0)$.
Theorem \ref{thrm:deRhamDiffStrat} below reduces to the de Rham
theorem of \cite{SjamaarDeRham} in this case.
\end{example}

As our primary example of a differentiable stratified groupoid satisfying the locally contractibility condition,
we will show in Section \ref{sec:InertiaGpoid} that the  inertia groupoid of a
proper Lie groupoid fulfills the local contractibility condition as well.

We now prove that the sheaf complex of basic forms on a
sliceable proper reduced differentiable stratified groupoid
fulfilling the contractibility condition is exact, or in other words
satisfies Poincar\'e's Lemma.

\begin{theorem}
\label{thrm:deRhamDiffStrat}
Let $\sfG$ be a sliceable proper reduced structurally (b)-regular differentiable stratified groupoid satisfying the local
contractibility condition.
The complex of sheaves $(\Omega_{\textup{basic}}^\bullet, d)$ on $|\sfG|$ then is a fine resolution of the sheaf of locally constant
real-valued functions on $|\sfG|$. In particular this implies that the cohomology of the complex $(\Omega_{\textup{basic}}^\bullet, d)$ of basic
differential forms on $\sfG$ is naturally isomorphic to the singular cohomology of $|\sfG|$ with coefficients in $\R$.
\end{theorem}
\begin{proof}
We first consider the case where the groupoid $\sfG$ is of the form
$(O\times O) \times (G \ltimes Y) \rightrightarrows O \times Y$, where $O$ is an open contractible set in some
$\R^m$, $O\times O$ denotes the corresponding pair groupoid, $G$ is a Lie group acting linearly on some $\R^n$,
and  $Y \subset \R^n$ is an affine $G$-invariant differentiable stratified space on which $G$ acts by
strata preserving maps. In addition we assume that $0\in Y$ and that there exists a smooth homotopy
$h : \widetilde{V} \times [0,1] \to \widetilde{V}$ defined on an open $G$-invariant subset $\widetilde{V} \subset \R^n$
such that the five conditions of the local contractibility condition
are satisfied with  $\iota : Y \hookrightarrow \R^n$ being the identical embedding.
Denote by $I \subset \calC^\infty (O \times \widetilde{V})$ the vanishing ideal
of $O \times Y$. Observe that by \cite[Sec.~2]{DomJanZhiRPLCQHVFTSV} and the smooth contractibility of $Y$, we have that the
subcomplex $I^\bullet \subset \Omega^\bullet (O\times \widetilde{V} )$ defined by
\[
  I^k :=
  \begin{cases}
    I, & \text{for $k=0$}, \\
    I \Omega^k (O\times \widetilde{V}) + dI \wedge \Omega^{k-1} (O\times \widetilde{V}), & \text{for $k=1,\ldots ,n+m$}
  \end{cases}
\]
is contractible. More precisely, a contraction is given by
\begin{equation}
  \label{eq:AlgContraction}
  K\omega  =
  \begin{cases}
  0, & \text{for $\omega \in \Omega^0 (O\times \widetilde{V}) = \calC^\infty (O\times \widetilde{V})$},\\
  \int_0^1 H_t^* (\xi_t \lrcorner \omega ) \, dt, &
  \text{for $\omega \in \Omega^k (O\times \widetilde{V})$, $k\in \N^*$},
  \end{cases}
\end{equation}
where the homotopy $h$ has been extended to a homotopy on $O \times \widetilde{V}$ by putting
$H_t (v, x) = (v,h_t(x))$ for  $v\in O$, $x\in \widetilde{V}$, $t \in [0,1]$,
and $\xi_t : O \times \widetilde{V} \to T \widetilde{V}$ is the vector field defined by
$\xi_t := \partial_t H_t$. Cartan's magic formula now implies that
\begin{equation}
  \label{eq:AlgHomotopyrelation}
  \omega - H^*_0 \omega = dK\omega + K d\omega, \quad  \text{ for $\omega \in \Omega^k (O\times \widetilde{V})$,
  $k\in \N$}.
\end{equation}
But this implies that the restriction of $K$ to $I^\bullet$ is an algebraic contraction indeed, since
every form $\omega \in  I^k$ satisfies the relation $H^*_0 \omega =0$.

Now consider the subcomplex $\Omega^\bullet_\textup{r-basic} (O\times \widetilde{V})\subset \Omega^\bullet (O\times \widetilde{V})$ of
\emph{relative basic forms}, or more precisely of \emph{relative  $O\times Y$ basic forms} on $O\times \widetilde{V}$, defined as follows.
The subcomplex $\Omega^\bullet_\textup{r-basic} (O\times \widetilde{V})$ is the set of all $\omega \in \Omega^k (O\times \widetilde{V})$ which
are invariant under the $G$-action and which have the property
that for each stratum $S$ of $O\times Y$ the form $\iota_S^* \omega $ is a basic form for the restricted Lie groupoid $\sfG_{|S}$.
In particular, this implies that $H_0^* \omega = \iota_{O\times \{ 0 \}}^* \omega $ vanishes for each restricted basic form $\omega$.
Since $H_t$ commutes with the $G$-action and maps fibers $O\times \{ y\}$ to
$O \times \{ h_t(y)\}$, the algebraic contraction $K$ maps basic forms to basic forms.
One concludes that the complex $\Omega^\bullet_\textup{r-basic} (O\times \widetilde{V})$ is exact, and that
the subcomplex $I^\bullet_\textup{r-basic} := I^\bullet \cap \Omega^\bullet_\textup{r-basic} (O\times \widetilde{V})$
is contractible. Hence the quotient complex $\Omega^\bullet_\textup{r-basic} (O\times \widetilde{V}) / I^\bullet_\textup{r-basic}$
is exact. Moreover, one has
\[
  \Omega^\bullet_\textup{basic} (Y/G) = \Omega^\bullet_\textup{r-basic} (O\times \widetilde{V}) / I^\bullet_\textup{r-basic},
\]
which can be seen by averaging a representative  of an element $[\omega] \in \Omega^k_\textup{basic} (Y/G)$ over the orbits of the $G$-action
using a bi-invariant Haar measure on $G$. The resulting new representative $\omega$ then is $G$-invariant.
Since $[\omega]$ is basic, the pull-back of such a representative $\omega$ to each stratum $S$ of $O \times Y$ has to be basic as well,
hence $\omega \in \Omega^\bullet_\textup{r-basic} (O\times \widetilde{V})$. So we have shown Poincare's Lemma in the special
case where the groupoid $\sfG$ is of the form $(O\times O) \times (G \ltimes Y) \rightrightarrows O \times Y$ with $Y$ and $O$ as stated above.

Next let us consider the general case of a proper sliceable structurally (b)-regular differentiable stratified groupoid
$\sfG$ which fulfills the local
contractibility condition. Let $\obj{x} \in \sfG_0$ be a point in the  object space. Since $\sfG$ is sliceable, we can choose
a trivializing neighborhood $U_\obj{x} \subset \sfG_0$ of $\obj{x}$, an open contractible neighborhood $O_\obj{x}$ in the orbit through $\obj{x}$ and
a groupoid slice $Y_\obj{x} \subset U_\obj{x}$  with a $\sfG_\obj{x}$-action such that $\sfG_{|U_\obj{x}}$ as a differentiable stratified
groupoid is isomorphic to the groupoid $(O_\obj{x} \times O_\obj{x}) \times (\sfG_\obj{x} \ltimes Y_\obj{x})$.
We will show that $\Omega^\bullet_\textup{basic} (\pi(U_\obj{x}))$ is exact. To this end let $\omega \in \Omega^k (\Sat (U_\obj{x}))$ be
a closed basic $k$-form. The restriction of  $\omega$ to $U_\obj{x}$ is a closed and $\sfG_{|U_\obj{x}}$-basic form. By the proceeding considerations
there exists a $\sfG_{|U_\obj{x}}$-basic $(k-1)$-form $\eta  \in \Omega^{k-1} (U_\obj{x})$ such that $d \eta = \omega_{|U_\obj{x}}$.
By Proposition \ref{proposition-extension-basic-forms}, $\eta$ has a unique extension to a  basic form
$\widehat{\eta} \in \Omega^{k-1}_\textup{basic} (\pi(U_\obj{x}))$. By construction of $\widehat{\eta}$ we have $d \eta = \omega$, since
$d$ commutes which each of the isomorphisms $t\circ \sigma$, where $\sigma : U \to \sfG$ is a bisection. This proves exactness of
$\Omega^\bullet_\textup{basic} (\pi(U_\obj{x}))$.
\end{proof}

\begin{remark}
\label{rem:DeRhamReduceLie}
If $\sfG$ is a proper Lie groupoid, then the complex of sheaves
$(\Omega_{\textup{basic}}^\bullet, d)$ coincides with the sheaf of basic differential
forms defined in \cite[Def.~8.1]{PflPosTanGOSPLG} so that Theorem \ref{thrm:deRhamDiffStrat}
reduces to \cite[Prop.~8.6 \& Cor.~8.7]{PflPosTanGOSPLG}.
\end{remark}

If $\sfG$ is a sliceable groupoid, then $(s,t)(\sfG_1)$ is necessarily
locally closed in $\sfG_0\times\sfG_0$. Specifically, using the local model of $\sfG$ given by
condition (LT\ref{it:ExistenceSlice}), the condition that $(s,t)(\sfG_1)$ is locally closed is
equivalent to the requirement that the $G_x$-action on $Y_x$ is a proper group action, which
is automatically satisfied as $G_x$ is compact by \cite[Prop.~2.10(ii)]{TuNonHausGpoid}.
Therefore, $|\sfG|$ is locally compact by Proposition
\ref{prop:BasicPropTopGroupoids}(\ref{it:LocCpctOrbitSp}). If we assume further that
$|\sfG|$ carries a stratification according to Proposition \ref{prop:LocTransStrat} and that this
stratification is Whitney (b)-regular, then $|\sfG|$ is a differentiable stratified space with
control data in the sense of Mather by \cite[Thm.~3.6.9]{PflaumBook}.
Hence  $|\sfG|$ admits in this case a triangulation subordinate to its stratification,
cf.~\cite[Thm.~7.1]{PflPosTanGOSPLG}. Therefore, given an open covering of $|\sfG|$, there exists a
subordinate good covering. See \cite[Sec.~7]{PflPosTanGOSPLG} for more details.
As in the case of Lie groupoids \cite[Cor.~8.8]{PflPosTanGOSPLG}
one concludes that the singular cohomology of $|\sfG|$ with real coefficients coincides with the \v{C}ech
cohomology under these hypotheses, and that it is finite-dimensional if $|\sfG|$ is compact.


\section{The inertia groupoid of a proper Lie groupoid as a differentiable stratified groupoid}
\label{sec:InertiaGpoid}

The goal of this section is to construct an explicit Whitney (b)-regular stratification of the loop
space $\inertianull{\sfG}$ of a proper Lie groupoid $\sfG$ so that the inertia groupoid $\inertia{\sfG}$ becomes
a reduced differentiable stratified groupoid which is sliceable and satisfies the
local contractibility condition. Our strategy employs the
slice theorem for proper Lie groupoids which enables us to describe the groupoid $\sfG$ locally
in terms of translation groupoids by compact Lie groups.
This will allow us to describe $\inertia{\sfG}$ locally in terms of the inertia groupoid associated
to such a translation groupoid. Using isotropy types and an equivalence relation on the group defined
in terms of this action, we will construct stratifications for inertia groupoids of such translation
groupoids that patch together to a well-defined stratification of $\inertianull{\sfG}$.


\subsection{The inertia groupoid of a proper Lie groupoid}
\label{subsec:InertiaGpoidIntro}

Let $\sfG$ be a proper Lie groupoid. Define the \emph{loop space} of $\sfG$ to be
\[
    \inertianull{\sfG} :=  \{ \arr{h} \in \sfG_1 \mid s(\arr{h}) = t(\arr{h}) \}.
\]
Since the loop space is closed in $\sfG_1$ it inherits the structure of a reduced differentiable space from
the ambient manifold $\sfG_1$. The map $s = t\co \inertianull{\sfG}\to \sfG_0$ serves as an anchor map
for the action of $\sfG$ on $\inertianull{\sfG}$ by conjugation.
More precisely, the action of $\arr{g}\in \sfG_1$ on $\arr{h}\in \inertianull{\sfG}$ with
$s(\arr{g}) = s(\arr{h})$ is given by
\begin{equation}
\label{eq:InertiaActionDef}
    \arr{g}\cdot\arr{h} := \arr{g}\arr{h}\arr{g}^{-1}.
\end{equation}

\begin{definition}
\label{def:InertiaGroupoid}
The \emph{inertia groupoid} $\inertia{\sfG}$ of a proper Lie groupoid $\sfG$ is the action
groupoid $\inertia{\sfG} := \sfG\ltimes\inertianull{\sfG}$.  The space of its objects
is the loop space $\inertianull{\sfG}$, its space of
arrows is $\sfG_1 \fgtimes{s}{s} \inertianull{\sfG}$.
The \emph{inertia space} of $\sfG$ then is the orbit space $|\inertia{\sfG}|$.
\end{definition}

\begin{remark}
  Note that, while $\inertianull{\sfG}$ is a differentiable subspace of the
  smooth manifold $\sfG_1$, the action of $\sfG$ on $\inertianull{\sfG}$ does not necessarily extend to an
  action on $\sfG_1$ or an open neighborhood of $\inertianull{\sfG}$ in $\sfG_1$.
\end{remark}

\subsection{The stratification of the loop space: Statement of the results}
\label{subsec:InertiaGpoidStrat}

In this section, we state our main results about the stratifications of the loop and inertia spaces of a proper Lie groupoid.
The proofs will be given in Sections \ref{subsec:GonMStratProof}, \ref{subsec:GpoidStratProof}, and \ref{subsec:Whitney}.
We begin with the case of a translation groupoid.

\subsubsection*{The compact Lie group action case}
Assume that the compact Lie group $G$ acts by diffeomorphisms on the
smooth manifold $M$ without boundary. The loop space $\inertianull{(G \ltimes M)}$ coincides in this case
with the union $\bigcup_{g\in G} \{ g \} \times M^g \subseteq G \times M$, where $M^g$ denotes the fixed point
space of $g\in G$. The $G$-action on $\inertianull{(G \ltimes M)}$ is the diagonal action, acting by conjugation on the first factor.
To describe our stratification of $\inertianull{(G \ltimes M)}$ recall first
\cite[IV Def.~4.1]{BroeckertomDieck} that a closed subgroup $\cartan$ of the Lie group
$G$ is called a \emph{Cartan subgroup} if it is closed, topologically cyclic, and of finite index in its normalizer;
\emph{topologically cyclic} means that it contains a dense cyclic subgroup, and implies by
\cite[I Cor.~4.14]{BroeckertomDieck} that $\cartan$ is isomorphic to the product of a torus and a finite cyclic group.
As in \cite{FarsiPflaumSeaton}, we say $\cartan$ is \emph{associated to an element $h \in G$} if $h \in \cartan$, and
$\cartan/\cartan^\circ$ is generated by $h\cartan^\circ$ where $\cartan^\circ$ denotes the connected component of the identity.
Now let $(h, x) \in\inertianull{(G \ltimes M)} \subset G \times M$, and choose a slice
$Y_x$ at $x$ for the $G$-action on $M$, see \cite[Def.~2.3.1]{duistermaatkolk}. By shrinking $Y_x$ if necessary
and choosing an appropriate $G$-invariant riemannian metric on $M$, we can assume that $Y_x$ is the
image under the exponential map of an open ball $B_x \subset N_x$ around the origin of the
normal space $N_x := T_xM/T_x(Gx)$ to the tangent space of the orbit through $x$.
Let $H = G_{(h,x)}$, and note that $H = G_x \cap \centralizer_G(h) = \centralizer_{G_x}\!(h)$
is the centralizer of $h$ in $G_x$. Let $\cartan_{(h,x)}$ be a Cartan subgroup of $H$ associated to $h$.
Note that if $G_x$ is connected, the relation $h \in (\centralizer_{G_x}\!(h))^\circ = H^\circ$
holds true by \cite[Thm.~3.3.1 (i)]{duistermaatkolk}, so that $\cartan_{(h, x)}$ is a maximal torus of $H^\circ$
containing $h$. Define an equivalence relation $\simeq$ on $\cartan_{(h,x)}$ by $s \simeq t$ if
$N_x^s  = N_x^t$, and let $\cartan_{(h,x)}^{\ast}$ denote the connected
component of the $\simeq$ class containing $h$. Note that by construction,
$s \simeq t$  if and only if the germs of the sets $Y_x^s$ and $Y_x^t$ at $x$ coincide.

By \cite[Prop.~(3.1.1)(i)]{duistermaatkolk}, $H$ is a slice at $h$ for the action of $G_x$
on itself by conjugation. Because $G_x$ acts on $Y_x$ and fixes $x$, it follows that $H\times Y_x$
is a slice at $(h,x)$ for the action of $G_x$ on $G_x\times Y_x$ which is given by
$g(k,y) = (gkg^{-1},gy)$, i.e.~by the diagonal action with conjugation on the $G_x$-factor. Hence,
because $\cartan_{(h,x)} \leq H$, we have that $\cartan_{(h,x)} \times Y_x$ is contained in the slice
$H\times Y_x$ at the point $(h,x)$, which has $G_x$-isotropy group $H$. It follows that
$\left(\cartan_{(h,x)}^{\ast} \times Y_x^{G_x}\right)^H = \left(\cartan_{(h,x)}^{\ast} \times Y_x^{G_x}\right)_H$
where $S_H$ denotes the set of points in a set $S$ with isotropy group equal to $H$,
i.e. that each point in $\cartan_{(h,x)}^{\ast} \times Y_x^{G_x}$ fixed by $H$ has $G_x$-isotropy $H$.

\begin{remark}
\label{rem:HIsotropyUnambiguous}
Note that if $(k,y) \in \left(\cartan_{(h,x)}^{\ast} \times Y_x^{G_x}\right)_H$, then $y\in Y_x^{G_x}$
so that, as $Y_x$ is a slice for the $G$-action on $M$, $G_y = G_x$. Hence, $G_{(k,y)} \leq G_x$
so that $G_{(k,y)} = H$. That is, a point in $\cartan_{(h,x)}^{\ast} \times Y_x^{G_x}$ has $G_x$-isotropy
$H$ if and only if it has $G$-isotropy $H$, and so the notation $\left(\cartan_{(h,x)}^{\ast} \times Y_x^{G_x}\right)_H$
is unambiguous.
\end{remark}

With these observations, assign to $(h,x)$ the germ
\begin{equation}
\label{eq:MLocalStratDef}
    \mathcal{S}_{(h,x)}
    =
    \left[G \left( \cartan_{(h,x)}^{\ast} \times Y_x^{G_x}\right)_H \right]_{(h,x)}
    =
    \left[G \left( \cartan_{(h,x)}^{\ast} \times Y_x^{G_x}\right)^H \right]_{(h,x)}.
\end{equation}
That is, near $(h,x)$, the stratum of $(h,x)$ consists of points $(k,y)$ such that
$G_{(k,y)}$ is conjugate to $H$ in $G$, and the conjugation isomorphism $G_{(k,y)}\to H$
maps $G_y$ onto $G_x$ and as well maps $k$ to a point in $\cartan_{(h,x)}$ that fixes
the same subspace of $Y_x$ as $h$.
It will be demonstrated below that Equation \eqref{eq:MLocalStratDef} yields a stratification of the loop space $\inertianull{(G \ltimes M)}$
that induces a stratification of the inertia space $|\inertianull{(G \ltimes M)}|$.

We refer to the stratification of Equation \eqref{eq:MLocalStratDef} as the \emph{orbit Cartan type stratification} of the loop space of the
Lie groupoid $G \ltimes M$. Note that this stratification differs from that given in
\cite{FarsiPflaumSeaton} in that we require the germ $\mathcal{S}_{(h,x)}$ to be a subset of
$Y_x^{G_x}$, and we use a different (weaker) definition of the equivalence relation $\simeq$ defining
$\cartan_{(h,x)}^\ast$.

The germ $\mathcal{S}_{(h,x)}$ is obviously $G_x$-invariant, and hence, if intersected with
$\inertianull{(G_x\ltimes Y_x)}$ (i.e.~take $G_x$-orbits rather than $G$-orbits)
defines a germ in the quotient $|\inertia{(G_x\ltimes Y_x)}|$.
This stratification depends only on the $G$-orbit of $(h,x)$,
and hence defines a germ in the quotient $|\inertia{(G\ltimes M)}|$ as well.
To see this, note that if $g \in G$, then $gY_x$ is a slice at $gx$ for the $G$-action
on $M$, and conjugation by $g$ maps $G_x$ onto $G_{gx}$. Choosing $\cartan_{g(h,x)}$ to be the
image of $\cartan_{(h,x)}$ under the induced isomorphism $G_x\ltimes Y_x \to G_{gx} \ltimes gY_x$,
the germ $\mathcal{S}_{g(h,x)}$ coincides with $\mathcal{S}_{(h,x)}$.
Hence, the stratification given in Equation \eqref{eq:MLocalStratDef} induces a stratification of
the inertia space $|\inertianull{(G \ltimes M)}|$, the \emph{orbit Cartan type stratification}, given by
\begin{equation}
\label{eq:MLocalStratQuotDef}
    \mathcal{R}_{(h,x)}
    =
    G\backslash\left[G \left( \cartan_{(h,x)}^{\ast} \times Y_x^{G_x}\right)^H \right]_{(h,x)}.
\end{equation}
We have the following result which we will  prove in Section \ref{subsec:GonMStratProof}.

\begin{theorem}
\label{thrm:InertiaStratGMnfld}
Let $G$ be a compact Lie group and $M$ a smooth $G$-manifold. Then Equation \eqref{eq:MLocalStratDef}
defines a Whitney (b)-regular $G$-invariant stratification of the loop space $\inertianull{(G\ltimes M)}$
with respect to which $\inertianull{(G\ltimes M)}$ is a differentiable stratified space such
that the $G$-orbits are subsets of strata. Using the induced stratifications on the space of arrows and corresponding
fibered product, the inertia groupoid $\inertia{(G\ltimes M)}$ is a proper reduced structurally (b)-regular
differentiable stratified groupoid satisfying the local contractibility condition. Moreover, Equation \eqref{eq:MLocalStratQuotDef}
defines a stratification of the inertia space $|\inertianull{(G \ltimes M)}|$ with respect to which the orbit map
is a differentiable stratified surjective submersion.
\end{theorem}

\subsubsection*{The proper Lie groupoid case}
We now turn to the case of a proper Lie groupoid $\sfG$, and endow $\sfG$ with a transversally invariant
riemannian metric.  Recall \cite[Cor.~3.11]{PflPosTanGOSPLG}
that for each point $\obj{x} \in \sfG_0$, there is an open neighborhood $U$ of $\obj{x}$
in $\sfG_0$ diffeomorphic to $O \times B_{\obj{x}}$  such that $\sfG_{|U}$ is isomorphic to the product of the pair groupoid
over $O$ and $\sfG_{\obj{x}} \ltimes B_{\obj{x}}$.
Here, $O$ is an open ball around $\obj{x}$ in the orbit
of $\obj{x}$ and $B_{\obj{x}}$ is a $\sfG_{\obj{x}}$-invariant open ball around the origin
in the normal space $N_{\obj{x}} = T_\obj{x} \sfG_0 / T_{\obj{x}}\calO_{\obj{x}}$ to the tangent
space of the orbit through  $\obj{x}$.
According to \cite[Thm.~4.1]{PflPosTanGOSPLG}, one can
assume that the corresponding diffeomorphism $O \times B_{\obj{x}} \to \sfG_0$ is given,
over the factor $B_{\obj{x}}$, by the exponential map
with respect to the chosen transversally invariant riemannian metric.
We let $Y_{\obj{x}}\subset \sfG_0$ denote the image of $\{ \obj{x} \}\times B_{\obj{x}}$ under this diffeomorphism
and call it a \emph{slice for $\sfG$ at $\obj{x}$}; note that
$Y_{\obj{x}}$ is a slice in the sense of (LT\ref{it:ExistenceSlice}) in Definition~\ref{def:LocTransDiffGroupoid}.
In particular, by \cite[Thm.~4.1]{PflPosTanGOSPLG}, $\sfG_{|Y_{\obj{x}}}$ is isomorphic to
$\sfG_{\obj{x}} \ltimes B_{\obj{x}}$.
Since the latter is isomorphic to $\sfG_{\obj{x}} \ltimes Y_{\obj{x}}$,
we obtain an isomorphism between $\sfG_{|Y_{\obj{x}}}$ and $\sfG_{\obj{x}} \ltimes Y_{\obj{x}}$
which is induced by the exponential map and the canonical action of $\sfG_{\obj{x}}$ on
$N_{\obj{x}}$. This isomorphism gives rise to an inclusion
$\sfG_{\obj{x}} \ltimes Y_{\obj{x}} \hookrightarrow \sfG$ of differentiable stratified groupoids.
Note that the metric is used only to define the exponential map. Moreover, we will see in Proposition
\ref{prop:BasePointInvar} below that the stratification does not depend on the choice of metric.

\begin{remark}
Note that if $\sfG = G\ltimes M$ is a translation groupoid, then a slice as defined here corresponds to a slice for the $G$-action on $M$, so using the same notation for both will cause no confusion.
\end{remark}

To define a stratification of $\inertianull{\sfG}$, we will employ the stratification  constructed above
of the loop space $\inertianull{(\sfG_{\obj{x}}\ltimes Y_{\obj{x}})}$.
To this end  choose $\arr{g} \in \inertianull{\sfG}$ with $s(\arr{g}) = t(\arr{g}) = \obj{x} \in \sfG_0$.
We then define the germ $\mathcal{S}_{\arr{g}}^{\sfG}$ of the stratification of $\inertianull{\sfG}$ as follows.
Let $h \in \sfG_{\obj{x}}$ denote the element such that $(h,\obj{x})$ corresponds to the arrow $\arr{g}$ under the
isomorphism between $\sfG_{|Y_{\obj{x}}}$ and $\sfG_{\obj{x}} \ltimes Y_{\obj{x}}$.
Let $\mathcal{S}_{(h,\obj{x})}$ be the germ of the orbit Cartan type stratification of
$\inertianull{(\sfG_{\obj{x}} \ltimes Y_{\obj{x}})}$, and
let $\Sat (\mathcal{S}_{(h,\obj{x})})$ denote its saturation within $\sfG$, i.e.~the germ of the saturation of
a defining set for $\mathcal{S}_{(h,\obj{x})}$.
Putting
\begin{equation}
\label{eq:MLocalStratGpoidDef}
    \mathcal{S}_{\arr{g}}^{\sfG}:= \Sat (\mathcal{S}_{(h,\obj{x})})
\end{equation}
then defines the \emph{orbit Cartan type stratification of the loop space $\inertianull{\sfG}$}.
Similarly, we define
\begin{equation}
\label{eq:MLocalStratGpoidDefQuot}
    \mathcal{R}_{\inertia{\pi}(\arr{g})}^{\sfG}:= \inertia{\pi}(\Sat (\mathcal{S}_{(h,\obj{x})})),
\end{equation}
where $\inertia{\pi}\co\inertianull{\sfG}\to |\inertia{\sfG}|$ is the orbit map of the inertia
groupoid. That means that the germ $\mathcal{R}_{\inertia{\pi}(\arr{g})}^{\sfG}$ in the orbit space
$|\inertia{\sfG}|$ is defined to be the projection of $\mathcal{S}_{\arr{g}}^{\sfG}$ to the orbit space.

With these definitions, we can state the next result. It will be proven in Sections \ref{subsec:GpoidStratProof}
and \ref{subsec:Whitney}.

\begin{theorem}
\label{thrm:InertiaStrat}
Let $\sfG$ be a proper Lie groupoid. Then Equation \eqref{eq:MLocalStratGpoidDef} defines
a Whitney (b)-regular stratification of the loop space $\inertianull{\sfG}$ with respect to which the
inertia groupoid $\inertia{\sfG}$ is a sliceable structurally (b)-regular differentiable stratified groupoid.
Moreover, the inertia space $|\inertia{\sfG}|$ inherits a differentiable structure, and
Equation \eqref{eq:MLocalStratGpoidDefQuot} defines a stratification with respect to which
$|\inertia{\sfG}|$ is a differentiable stratified space and the orbit map
$\inertianull{\sfG}\to|\inertia{\sfG}|$ a differentiable stratified surjective submersion.
\end{theorem}


\subsection{Proof of Theorem~\ref{thrm:InertiaStratGMnfld}}
\label{subsec:GonMStratProof}

The stratification given by Equation \eqref{eq:MLocalStratDef} is a variation on the
stratification given by \cite[Thm.~4.1]{FarsiPflaumSeaton}, so the proof of
Theorem~\ref{thrm:InertiaStratGMnfld} will follow a similar outline as a series of lemmata.
The primary virtue of the revised
definition given here is that it is invariant under Morita equivalences, and
hence can be glued together on intersections of charts of a proper Lie groupoid.
For the convenience of the reader, we present the complete argument.

We assume $G\times M$ is equipped with a riemannian metric given by the product of a $G$-invariant
metric on $M$ and a bi-invariant metric on $G$. For a point $y\in M$, we will denote by $Y_y$
a slice at $y$ for the $G$-action on $M$. We use $H$ to denote the isotropy group $G_{(h,x)} = \centralizer_{G_x}\! (h)$
of $(h,x)\in \inertianull{(G\ltimes M)}$ and the symbol $N_{(h,x)}$ to denote the normal space
$T_{(h,x)} (G_x\times Y_x) / T_{(h,x)} (G_x(h,x))\cong T_{(h,x)} (G \times M) / T_{(h,x)} (G (h,x))$.
Because $H\times Y_x$ is a slice at $(h,x)$ for the $G_x$-action on $G_x\times Y_x$, the exponential map
identifies a neighborhood of $N_{(h,x)}$ with a neighborhood of $(h,x)$ in $G_x\times Y_x$.

In order to prove that the stratification does not depend on the choice of Cartan subgroup $\cartan_{(h,x)}$,
it will be helpful to observe the following, which is a slight strengthening of a special
case of \cite[IV Prop.~4.6]{BroeckertomDieck}.

\begin{lemma}
\label{lem:CartanCongConComp}
Let $G$ be a compact Lie group, let $h$ and $k$ be elements of a single connected component
of $G$, and let $\cartan_h$ and $\cartan_k$ be Cartan subgroups of $G$ associated to $h$ and $k$,
respectively.  Then there is an element  $g \in G^\circ$ such that
\[
    g\cartan_k g^{-1} = \cartan_h \quad\text{and}\quad g(k\cartan_k^\circ) g^{-1} = h\cartan_h^\circ .
\]
\end{lemma}
\begin{proof}
By \cite[IV Prop.~4.6]{BroeckertomDieck}, we know a priori that $\cartan_h$
and $\cartan_k$ are conjugate and hence isomorphic.
Let $\cartan_k \cong \T^\ell \times \Z/r\Z$ for some $\ell$ and $r$.  Then the topological
generators of $\T^\ell \times \Z/r\Z$, i.e.~the elements that generate a dense subset of
$\T^\ell \times \Z/r\Z$, are given by $(s, \alpha)$ where $s$ is a topological generator of
$\T^\ell$ and $\alpha$ is a generator of $\Z/r\Z$; see
\cite[proof of IV Prop.~4.6]{BroeckertomDieck}.
Since $k\cartan_k^\circ$ generates $\cartan_k/\cartan_k^\circ \cong \Z/r\Z$,
we may therefore choose a topological generator $t$ of $\cartan_k$ such that
$t \in k\cartan_k^\circ$.
Note that $t$, $h$, and $k$ are all in the same connected component of $G$.
By the same argument a topological generator of $\cartan_h$ can be chosen to be an
element of $h\cartan_h^\circ$. Hence there is an element
$g \in G^\circ$ such that $g t g^{-1} \in h\cartan_h^\circ$. But then $gtg^{-1}$ generates a
subgroup of $\cartan_h$ that is dense in $g\cartan_k g^{-1}$.  Therefore $g\cartan_k g^{-1}$ is
contained in $\cartan_h$. But as $g\cartan_k g^{-1}$ is isomorphic to $\cartan_k$, hence
to $\cartan_h$, we have $g\cartan_k g^{-1} = \cartan_h$.
\end{proof}

We can now prove the following using an argument similar to \cite[Lemma~4.14]{FarsiPflaumSeaton}.

\begin{lemma}
Let $(h,x) \in G\times M$. The set germ $\mathcal{S}_{(h,x)}$ does not depend on the choice of the
Cartan subgroup $\cartan_{(h,x)}$.
\end{lemma}
\begin{proof}
Suppose $\cartan_{(h,x)}$ and $\cartan_{(h,x)}^\prime$ are two Cartan subgroups of $H$ associated
to $h$. Applying Lemma~\ref{lem:CartanCongConComp} with $k = h$, there is a $g\in H$ such that
$g\cartan_{(h,x)}g^{-1} = \cartan_{(h,x)}^\prime$ and $g(h\cartan_{(h,x)}^\circ)g^{-1} = h(\cartan_{(h,x)}^\prime)^\circ$;
as $g \in H = G_{(h,x)}$, $ghg^{-1} = h$.

Choose a slice $Y_x$ at $x$ for the $G$-action on $M$. Then as $g\in H\leq G_x$ so that $g^{-1}Y_x = Y_x$,
we have that for $t\in\cartan_{(h,x)}$, $(g^{-1}Y_x)^t = (g^{-1}Y_x)^h$ if and only if $Y_x^t = Y_x^h$,
i.e. $t\simeq h$ if and only if $gtg^{-1} \simeq ghg^{-1}$. Therefore,
$g\cartan_{(h,x)}^\ast g^{-1} = (\cartan_{(h,x)}^\prime)^\ast$. Finally, as $g \in H \leq G_x$, we have both
$gY_x = Y_x$ and $g(M_H) = M_H$ so that
\[
    g\left( \cartan_{(h,x)}^{\ast} \times Y_x^{G_x}\right)_H
    =
    \left((\cartan_{(h,x)}^\prime)^{\ast} \times Y_x^{G_x}\right)_H,
\]
implying
\[
    G \left( \cartan_{(h,x)}^{\ast} \times Y_x^{G_x}\right)_H
    =
    G \left( (\cartan_{(h,x)}^\prime)^{\ast} \times Y_x^{G_x}\right)_H
\]
and completing the proof.
\end{proof}

It will be helpful to observe the following simple fact.

\begin{lemma}
\label{lem:CartanGxZGx}
Let $(h,x) \in \inertianull{(G \ltimes M)}$ and $H = G_{(h,x)}$. A subgroup $\cartan_{(h,x)}$ of $G_x$
is a Cartan subgroup of $G_x$ associated to $h$ if and only if $\cartan_{(h,x)}$ is a Cartan subgroup of
$H$ associated to $h$.
\end{lemma}
\begin{proof}
Let $\cartan_{(h,x)}$ be a Cartan subgroup of $G_x$ associated to $h$. As $\cartan_{(h,x)}$ is abelian and contains
$h$, it is a subgroup of $\centralizer_{G_x}\!(h) = H$. As $\cartan_{(h,x)}$ has finite index in
$\normalizer_{G_x}(\cartan_{(h,x)})$, it obviously has finite index in
$\normalizer_H(\cartan_{(h,x)}) \leq \normalizer_{G_x}(\cartan_{(h,x)})$. The fact that $\cartan_{(h,x)}$ is closed
and topologically cyclic and that $\cartan_{(h,x)}/\cartan_{(h,x)}^\circ$ is generated by $h\cartan_{(h,x)}^\circ$
does not depend on the ambient group so that $\cartan_{(h,x)}$ is a Cartan subgroup of $H$ associated to $h$.

Conversely, if $\cartan_{(h,x)}$ is a Cartan subgroup of $H$ associated to $h$, then as noted in the previous paragraph,
we need only show that $\cartan_{(h,x)}$ has finite index in $\normalizer_{G_x}(\cartan_{(h,x)})$.
Following the proof of \cite[IV Prop.~4.2]{BroeckertomDieck}, note that as $\cartan_{(h,x)}$ is isomorphic
to the product of a torus and a finite cyclic group, its automorphism group is discrete. As
$\normalizer_{G_x}(\cartan_{(h,x)})/\centralizer_{G_x}(\cartan_{(h,x)})$ acts on $\cartan_{(h,x)}$ as automorphisms
and is compact, it must be finite. Moreover, $\centralizer_{G_x}(\cartan_{(h,x)})$ is clearly contained in
$\centralizer_{G_x}(h) = H$ and hence in
$\normalizer_{H}(\cartan_{(h,x)})$ so that as $\cartan_{(h,x)}$ has finite index in $\normalizer_{H}(\cartan_{(h,x)})$,
the index $[\centralizer_{G_x}(\cartan_{(h,x)}):\cartan_{(h,x)}]$ is finite. Therefore,
\[
    [\normalizer_{G_x}(\cartan_{(h,x)}):\cartan_{(h,x)}]
    =
    [\normalizer_{G_x}(\cartan_{(h,x)}):\centralizer_{G_x}(\cartan_{(h,x)})]
    [\centralizer_{G_x}(\cartan_{(h,x)}):\cartan_{(h,x)}]
\]
is finite, completing the proof.
\end{proof}

Recall that the equivalence relation $\simeq$ on a Cartan subgroup $\cartan$ of the isotropy group
$G_x$ is defined by setting $s\simeq t$ if and only if $N_x^s = N_x^t$, where $N_x$ denotes the
normal space $T_xM/T_x (Gx)$. We denote by $[s]$ the $\simeq$ class of $s \in \cartan$. Note that
by definition of a slice, $s\simeq t$ holds true if and only $Y_x^s = Y_x^t$ for one, hence all
(sufficiently small) slices $Y_x$ at $x$.
We recall the following properties of the relation $\simeq$, whose proofs are similar
to \cite[Lemmata 4.7, 4.8, \& 4.10]{FarsiPflaumSeaton}.

\begin{proposition}
\label{prop:simeqProp}
Let $Y$ be a slice for the $G$-action on $M$ through a point $x\in M$ and  let $\cartan \subset G_x $
be a Cartan subgroup. Then the following holds true.
\begin{enumerate}[\rm(1)]
\item
\label{SimeqIte1}
The group $\cartan$ is partitioned into a finite number of $\simeq$ classes, each with a finite
number of connected components. Each $\simeq$ class $[t]$ is an open subset of the closed subgroup
$t^\bullet$ of $\cartan$ defined by
\[
    t^\bullet
    :=
    \bigcap\limits_{t \in H_i} H_i
    =
    \bigcap\limits_{y \in Y^t} \cartan_y,
\]
where $\{H_0, \ldots, H_r\}$ is the finite set of isotropy groups for the $\cartan$-action on $Y$ and
$\cartan_y$ is the isotropy group of $y$ in $\cartan$. The closure
$\overline{[t]}$ of $[t]$ consists of a union of connected components of $t^\bullet$.

\item
\label{SimeqIte2}
If $s, t \in \cartan$ with $[s] \cap \overline{[t]} \neq \emptyset$, then for each connected component
$[s]^\circ$ of $[s]$ and $[t]^\circ$ of $[t]$ such that
$[s]^\circ \cap \overline{[t]^\circ} \neq \emptyset$, we have
$[s]^\circ \subset \overline{[t]^\circ}$.

\item
\label{SimeqIte3}
If $s, t \in \cartan$ such that $s \not\simeq t$ and $[s]$ is diffeomorphic to $[t]$, then
$[s] \cap \overline{[t]} = \emptyset$.
\end{enumerate}
\end{proposition}
\begin{proof}
The $\simeq$ class $[t]$ of $t\in\cartan$ is given by
\[
    [t] = \bigcap\limits_{t \in H_i} H_i \cap  \left(\bigcup\limits_{t \not\in H_j} H_j\right)^c
\]
and hence determined by a subset of $\{ 0, 1, \ldots, r \}$, yielding (\ref{SimeqIte1}).

If $u \in [s]_{}^\circ \cap \overline{[t]_{}^\circ}$, then $Y^s = Y^u$, and by continuity of the
action, $Y^t \subseteq Y^u$. Hence $Y^t \subseteq Y^s$, and
$s^\bullet = \bigcap_{y \in Y^s} \cartan_y \leq \bigcap_{y \in Y^t} \cartan_y = t^\bullet$.
As $[s]^\circ$ is contained in a connected component
$(s^\bullet)^\circ$ of $s^\bullet$, which is contained in a connected component
$(t^\bullet)^\circ$ of $t^\bullet$, and as $\overline{[t]}$ consists of
entire connected components of $t^\bullet$, we have
$[s]^\circ \subseteq (t^\bullet)^\circ = \overline{[t]^\circ}$, proving (\ref{SimeqIte2}).

For (\ref{SimeqIte3}), if $Y^s \subseteq Y^t$,
then $\bigcap_{y \in Y^t} \cartan_y \leq \bigcap_{y \in Y^s} \cartan_y$,
so that $t^\bullet \leq s^\bullet$. As $[s]$ and $[t]$ are diffeomorphic and open subsets of $s^\bullet$ and $t^\bullet$,
respectively, $s^\bullet$ and $t^\bullet$ have the same dimension.  Then as $[s]$
is open and dense in each connected component of $s^\bullet$ it intersects and similarly for $[t]$,
$[s]$ and $[t]$ do not intersect the same connected components of $t^\bullet$. This completes the proof in this case,
and the argument is identical if $Y^t \subseteq Y^s$.

If $Y^s \not\subseteq Y^t$ and $Y^t \not\subseteq Y^s$, then let
$l \in \overline{[s]} \cap \overline{[t]}$. By continuity of the action, $l$ fixes $Y^s \cup Y^t \neq Y_s$
so that $l \not\simeq s$.  So $l \in \overline{[s]} \smallsetminus [s]$,
and $[s] \cap \overline{[t]} = \emptyset$.
\end{proof}

%
%


We now demonstrate that $\mathcal{S}_{(h,x)}$ and $\mathcal{R}_{(h,x)}$ are the germs
of smooth manifolds, see Appendix~\ref{app:DiffStratSpaces}. The proof follows that of
\cite[Prop.~4.16]{FarsiPflaumSeaton}.

\begin{proposition}
\label{prop:LocalStratManifolds}
Each $\mathcal{S}_{(h,x)}$ is the germ of a smooth $G$-submanifold of $G \times M$ that
intersects $G_x \times Y_x$ as a smooth submanifold. Each $\mathcal{R}_{(h,x)}$ is the
germ of a smooth submanifold of the differentiable space $G\backslash(G \times M)$
that intersects $G_x\backslash(G_x \times Y_x)$ as a smooth submanifold.
\end{proposition}
\begin{proof}
Recall that $G\times M$ is equipped with an invariant riemannian metric given by the product of
metrics on $M$ and $G$ and that $Y_x$ is the image under the exponential map for $M$ of an open
ball $B_x$ about the origin in the normal space $N_x = T_xM/T_x(Gx)$.
Let $N_{(h,x)} = T_{(h,x)} (G_x\times Y_x)/T_{(h,x)}(G_x (h,x))$, the normal space
in $T_{(h,x)} (G_x\times Y_x)$ to the tangent space of the $G_x$-orbit of $(h,x)$,
let $B_{(h,x)} \subset N_{(h,x)}$ be an open $H$-invariant neighborhood of the origin,
and let $V_{(h,x)}$ be the image of $B_{(h,x)}$ under the exponential map. As $H\times Y_x$
is a slice at $(h,x)$ for the $G_x$-action on $G_x\times Y_x$, $V_{(h,x)}$ is an open
neighborhood of $(h,x)$ in $H\times Y_x$. By \cite[Sec.~6.1(4)]{michor}, the exponential map
$B_x\to Y_x$ maps $B_x^{G_x}$ onto $Y_x^{G_x}$, and the exponential map $B_{(h,x)}\to V_{(h,x)}$
maps $(B_{(h,x)})_H = B_{(h,x)}^H$ onto $(V_{(h,x)})_H = V_{(h,x)}^H$. As $\cartan_{(h,x)}^\ast$
is an open subset of the closed subgroup $h^\bullet$ by Proposition~\ref{prop:simeqProp}(\ref{SimeqIte1}),
we have $T_h\cartan_{(h,x)}^\ast = T_h h^\bullet$. Hence, the preimage of
$V_{(h,x)}\cap ( \cartan_{(h,x)}^{\ast} \times Y_x^{G_x})^H$ under the exponential map
is given by $(T_h h^\bullet \times B_x^{G_x})\cap B_{(h,x)}^H$. As $T_h h^\bullet$, $B_x^{G_x}$,
and $B_{(h,x)}^H$ are open subsets of the origin in linear spaces, it follows that $\mathcal{S}_{(h,x)}$
is the germ of a smooth manifold.

As $V_{(h,x)}$ is a slice at $(h,x)$ for the $G_x$-action on $G_x\times Y_x$, there is a $G_x$-equivariant
diffeomorphism $G_x\times_H V_{(h,x)} \to G_x V_{(h,x)} \subset G_x\times Y_x$ that restricts to a
$G_x$-equivariant diffeomorphism
\[
    G_x/H \times V_{(h,x)}\cap ( \cartan_{(h,x)}^{\ast} \times Y_x^{G_x})^H
    \to G_x\left(V_{(h,x)}\cap ( \cartan_{(h,x)}^{\ast} \times Y_x^{G_x})^H\right)
\]
and induces a homeomorphism on the quotients by $G_x$. The $G_x$-invariant functions
on $G_x/H \times V_{(h,x)}\cap ( \cartan_{(h,x)}^{\ast} \times Y_x^{G_x})^H$ coincide
with the smooth functions on $V_{(h,x)}\cap ( \cartan_{(h,x)}^{\ast} \times Y_x^{G_x})^H$
by \cite[Prop.~5.2]{tomDieck} and \cite[Thm.~1.2.2(5)]{NGonzalezSanchoBook} making this
homeomorphism into an isomorphism of differentiable spaces from the smooth manifold
$V_{(h,x)}\cap ( \cartan_{(h,x)}^{\ast} \times Y_x^{G_x})^H$ onto a subspace of the differentiable space
$G_x\backslash(G_x\times Y_x)$. Finally, it is easy to see that the inertia space
$|\inertia{(G_x\ltimes Y_x)}|\subset G_x\backslash(G_x\times Y_x)$ is isomorphic as a differentiable space
to an open neighborhood of $G(h,x)$ in the inertia space $|\inertia{(G\ltimes M)}|$, see
\cite[Prop.~3.6]{FarsiPflaumSeaton}, and that the germ of the image of
$V_{(h,x)}\cap ( \cartan_{(h,x)}^{\ast} \times Y_x^{G_x})^H$
in $|\inertia{(G\ltimes M)}|$ at $G(h,x)$ is $\mathcal{R}_{(h,x)}$, completing the proof.
\end{proof}

In order for the germs $\mathcal{S}_{(h,x)}$ to define a stratification,
one must verify that for each $(h,x) \in \inertianull{(G\ltimes M)}$ there is
a neighborhood $U$ in $\inertianull{(G\ltimes M)}$ and a decomposition $\mathcal{Z}$
of $U$ as defined in Appendix~\ref{ap:StratSpaces}
such that for all $(k, y) \in \inertianull{(G\ltimes M)}$, the germ
$\mathcal{S}_{(k,y)}$ coincides with the germ of the piece of $\mathcal{Z}$
containing $(k, y)$.  For the remainder of this section, we fix $(h,x)$,
a slice $Y_x$ at $x$ for the $G_x$-action on $M$, and a slice $V_{(h,x)}$
at $(h,x)$ for the $G_x$-action on $G_x \times Y_x$ given by an open neighborhood of
$(h,x)$ in $H\times Y_x$ that is the exponential image of a ball in the normal space $B_{(h,x)}$
as in the proof of Proposition~\ref{prop:LocalStratManifolds}.
Set $U := GV_{(h,x)} \cap \inertianull{(G\ltimes M)}$.  Note that $U$ is indeed an
open neighborhood of $(h,x)$ in $\inertianull{(G\ltimes M)}$, as $G_x V_{(h,x)}$
is an open $G_x$-invariant neighborhood of $(h,x)$ in $\inertianull{(G_x\ltimes Y_x)}$
and so the $G$-saturation is as well by \cite[Prop.~3.6]{FarsiPflaumSeaton}.
We now define the decomposition $\mathcal{Z}$ of $U$.

Given  $(\tilde{k}, \tilde{y}) \in U$ there is a $\tilde{g} \in G$ such that
$\tilde{g}(\tilde{k}, \tilde{y}) \in V_{(h,x)}$. Put $(k, y) = \tilde{g}(\tilde{k}, \tilde{y})$
and $K = G_{(k,y)} \leq H$, and let $\cartan_{(k,y)}$ be a Cartan subgroup in $K$
associated to $k$. Define $\mathcal{U}_{\tilde{g}}^{\cartan_{(k,y)}}(\tilde{k}, \tilde{y})$
to be the $G$-saturation of the set of points
$(l, z) \in (V_{(h,x)})_K \cap (\cartan_{(k,y)} \times (Y_x)_{G_y})$
such that $\cartan_{(k,y)}$ is also a Cartan subgroup of $K$ associated to $l$ and such that
the $\simeq$ class of $l$ at $z$ in $\cartan_{(k,y)}$ is diffeomorphic to $\cartan_{(k,y)}^{\ast}$.
Define the piece $\mathcal{Z}$ containing $(\tilde{k},\tilde{y})$ to be the connected
component of $\mathcal{U}_{\tilde{g}}^{\cartan_{(k,y)}}(\tilde{k}, \tilde{y})$
containing $(\tilde{k}, \tilde{y})$.

By \cite[Proposition 1.2(3)]{SchwarzLiftingHomotopies},
the slice representations of
points in the same orbit type are isomorphic.
Hence, if $z \in (Y_x)_{G_y}$, then the slice representations $N_y$ and $N_z$ for the action
of $G_x$ on $Y_x$ at $y$ and $z$, respectively, are isomorphic as $G_y$-representations.
This isomorphism induces an isomorphism of $\cartan_{(k,y)}$-representations, which
induces a diffeomorphism of $\simeq$
classes at $y$ onto $\simeq$ classes at $z$.  Then the $\simeq$ class of $l$ at $y$
is diffeomorphic to the $\simeq$ class of $l$ at $z$, and hence to the $\simeq$
class of $k$ at $y$.  Therefore, the set
$\mathcal{U}_{\tilde{g}}^{\cartan_{(k,y)}}(\tilde{k}, \tilde{y})$ can be written as
\[
  \mathcal{U}_{\tilde{g}}^{\cartan_{(k,y)}}(\tilde{k}, \tilde{y}) =
  G \Big( (V_{(h,x)})_K \cap \big(\cartan_{(k,y)}^{\ast\ast} \times (Y_x)_{G_y} \big) \Big),
\]
where $\cartan_{(k,y)}^{\ast\ast}$ denotes the union of $\simeq$ classes in $\cartan_{(k,y)}$
that are diffeomorphic to $\cartan_{(k,y)}^{\ast}$.  Choosing representatives
$k_0,\ldots, k_r$ from the collection of such $\simeq$ classes with $k_0 = k$ and noting
that this collection is finite by Proposition~\ref{prop:simeqProp}(\ref{SimeqIte1}), we can express
\begin{equation}
  \label{eq:decUrep}
  \mathcal{U}_{\tilde{g}}^{\cartan_{(k,y)}}(\tilde{k}, \tilde{y}) =
  G \Big(\bigcup\limits_{i=0}^r (V_{(h,x)})_K \cap \big(\cartan_{(k_i,y)}^{\ast}
  \times (Y_x)_{G_y} \big) \Big).
\end{equation}
In particular, it will be clear below that for every
$(k,y) \in V_{(h,x)}$ with isotropy group $K$ and every Cartan subgroup $\cartan_{(k,y)}$
of $K$ associated to $k$,
the connected component of $\cartan_{(k,y)}$ containing $k$ is the only connected
component that intersects the projection of $V_{(h,x)}$ onto $G_x$.
Note that $\mathcal{U}_{\tilde{g}}^{\cartan_{(k,y)}}(\tilde{k}, \tilde{y})$
is clearly a subset of $\inertianull{(G\ltimes M)}$ as $\cartan_{(k,y)} \leq G_y$.

We now demonstrate that $\mathcal{U}_{\tilde{g}}^{\cartan_{(k,y)}}(\tilde{k}, \tilde{y})$
depends only on the orbit $G(k,y)$. The proof follows those of
\cite[Lemmata 4.18 \& 4.19]{FarsiPflaumSeaton}, though is simpler in the situation at hand.

\begin{lemma}
\label{lem:DecompPieceWellDef}
The set $\mathcal{U}_{\tilde{g}}^{\cartan_{(k,y)}}(\tilde{k}, \tilde{y})$ does not depend on the choice
of $(\tilde{k}, \tilde{y})$, $\tilde{g}$, nor the Cartan subgroup $\cartan_{(k,y)}$.
\end{lemma}
\begin{proof}
Suppose $g \in G$ such that $g(\tilde{k}, \tilde{y}) =: (k^\prime, y^\prime) \in V_{(h,x)}$.
Then $g\tilde{g}^{-1}(k,y) = (k^\prime, y^\prime) \in V_{(h,x)} \subset G_x\times Y_x$.
This implies first that $\tilde{h} := g\tilde{g}^{-1} \in G_x$ as $\tilde{h}y = y^\prime$ so that
$\tilde{h}Y_x\cap Y_x\neq\emptyset$, and then that $\tilde{h}\in H$ as
$\tilde{h}V_{(h,x)}\cap V_{(h,x)}\neq\emptyset$. Then $G_{y^\prime} \leq G_x$ so that
$K^\prime := G_{(k^\prime, y^\prime)} = \tilde{h}K\tilde{h}^{-1} \leq H$.

Choose a Cartan subgroup $\cartan_{(k^\prime, y^\prime)}^\prime$ of $K^\prime$ associated to $k^\prime$.
Then $\tilde{h}\cartan_{(k,y)}\tilde{h}^{-1}$ is as well a Cartan subgroup of $K^\prime$
associated to $k^\prime$, so by \cite[IV.~Prop.~4.6]{BroeckertomDieck}, there is
a $\tilde{k} \in K^\prime$ such that
$\cartan_{(k^\prime, y^\prime)}^\prime = \tilde{k}\tilde{h}\cartan_{(k, y)} \tilde{h}^{-1}\tilde{k}^{-1}$.
As conjugation by elements of $G_x$ is easily seen to map $\simeq$ classes to $\simeq$ classes and
$\tilde{k}\tilde{h}(k,y) = \tilde{k}(k^\prime,y^\prime) = (k^\prime,y^\prime)$, we have
$\tilde{k}\tilde{h}\cartan_{(k, y)}^\ast \tilde{h}^{-1}\tilde{k}^{-1} = (\cartan_{(k^\prime, y^\prime)}^\prime)^\ast$
and hence
$\tilde{k}\tilde{h}\cartan_{(k, y)}^{\ast\ast} \tilde{h}^{-1}\tilde{k}^{-1}
= (\cartan_{(k^\prime, y^\prime)}^\prime)^{\ast\ast}$.
Then as $\tilde{k}\in K^\prime$ so that
$\tilde{k}\tilde{h}K \tilde{h}^{-1}\tilde{k}^{-1} = K^\prime$ and, as $K^\prime \leq G_{y^\prime}$,
$\tilde{k}\tilde{h}G_y \tilde{h}^{-1}\tilde{k}^{-1} = G_{y^\prime}$, it follows that
\[
    \tilde{k}\tilde{h} \Big( (V_{(h,x)})_K \cap \big(\cartan_{(k,y)}^{\ast\ast} \times (Y_x)_{G_y} \big) \Big)
    =
    (V_{(h,x)})_{K^\prime} \cap \big((\cartan_{(k^\prime,y^\prime)}^\prime)^{\ast\ast} \times
        (Y_x)_{G_{y^\prime}} \big).
\]
Hence, $\mathcal{U}_{\tilde{g}}^{\cartan_{(k,y)}}(\tilde{k}, \tilde{y})$ does not depend on the choice
of $(\tilde{k}, \tilde{y})$ nor $\tilde{g}$. Applying the same argument where $(k^\prime,y^\prime) = (k,y)$
so that $\cartan_{(k^\prime, y^\prime)}^\prime = \cartan_{(k, y)}^\prime$ is another choice of Cartan subgroup
demonstrates that $\mathcal{U}_{\tilde{g}}^{\cartan_{(k,y)}}(\tilde{k}, \tilde{y})$ does not depend on the choice
of Cartan subgroup.
\end{proof}

We may henceforth denote $\mathcal{U}_{\tilde{g}}^{\cartan_{(k,y)}}(\tilde{k}, \tilde{y})$ simply as
$\mathcal{U}(G(k, y))$ and let $\mathcal{U}_{(\tilde{k}, \tilde{y})}$ denote the connected component of
$\mathcal{U}(k, y)$ containing $(\tilde{k}, \tilde{y})$. The partition $\mathcal{Z}$ of $U$ then can
be written as
\begin{equation}
\label{eq:LocalStratPieces}
 \mathcal{Z} = \big\{ \mathcal{U}_{(\tilde{k} , \tilde{y})} \subset U \mid
             (\tilde{k} , \tilde{y}) \in U \big\} .
\end{equation}
In order to show that $\mathcal{Z}$ is a decomposition of $U$ as defined in Appendix~\ref{ap:StratSpaces},
we first demonstrate that it is finite.
If $(k^\prime,y^\prime)\in U = GV_{(h,x)}$ is chosen such that $y^\prime$ has $G$-isotropy type $(G_y)$
(i.e. $G_{y^\prime}$ is conjugate to $G_y$ in $G$), and $(k^\prime,y^\prime)$ has $G_y$-isotropy type $(K)$,
then there is a $g\in G$ such that $gy^\prime$ has $G$-isotropy group $G_y$ and $g(k^\prime,y^\prime)$
has $G_y$-isotropy group $K$. To see this, first choose $\tilde{g}\in G$ such that
$\tilde{g}G_{y^\prime} \tilde{g}^{-1} = G_{\tilde{g}y^\prime} = G_y$. Then choose $\tilde{k}\in G_y$
such that $\tilde{k} \tilde{g}G_{(k^\prime,y^\prime)}\tilde{g}^{-1}\tilde{k}^{-1} = K$.
As $\tilde{k}\in G_y$,  $\tilde{k} \tilde{g}(k^\prime,y^\prime) = \tilde{g}(k^\prime,y^\prime)$ is the desired
point. Then as the set $\mathcal{U}(k, y)$ is determined by choosing one of the finitely many isotropy types
$(G_y)$ for the $G$-action on $Y_x$, one of the finitely many isotropy types $(K)$ for the $G_y$-action
on $G_x\times Y_y$, and one of the finitely many diffeomorphism classes of $\simeq$ classes in a choice of
Cartan torus, there are finitely many set $\mathcal{U}(k, y)$ in $U = GV_{(h,x)}$.
It remains to show that each $\mathcal{U}(k, y)$ has finitely many connected components, which we will
accomplish following the proof of \cite[Lemmata~4.17 \& 4.21]{FarsiPflaumSeaton}, i.e. by giving a description
of $\mathcal{U}(k, y)$ as the $G$-saturation of a semialgebraic set.

Fix a $\mathcal{U}(k, y)$ with $(k,y)\in V_{(h,x)}$ and let $K = G_{(k,y)}$. As $(N_x)_{G_y}$ is formed by
removing from $(N_x)^{G_y}$ the finite collection of isotropy types it properly contains, and as each isotropy
type is semialgebraic by \cite[Theorem A]{BierstoneOrbitSpace},
it follows that $(N_x)_{G_y}$ and hence $(Y_x)_{G_y}$ are semialgebraic; $(V_{(h,x)})_K$ is semialgebraic
for the same reason. As the actions of $G_x$ on $N_x$ and $H$ on $N_{(h,x)}$ are linear, and as $N_{(h,x)}$
is a quotient of a product with $N_x$ so that the linear structures are compatible, both
$(N_x)_{G_y}$ and $(N_{(h,x)})_K$ are closed under multiplication by scalars $t\in (0,1]$.
As $K^\circ \leq H^\circ$, any maximal torus in $K^\circ$ is contained in a maximal torus in $H^\circ$.
Moreover, as $V_{(h,x)}^K$ is connected and closed under multiplication by scalars $t \in (0,1]$, taking
the limit of $t(k,y)$ as $t\to 0$, we see that $h \in K$, and in particular is in the connected component of
$K$ containing $k$. Following the proof of \cite[IV.~Prop.~4.2]{BroeckertomDieck}, we may then
choose a Cartan subgroup of $H$ associated to $h$ by taking the group generated by $h$ as well as a
maximal torus in $H^\circ$ that contains a maximal torus in $K^\circ$.  That is,
we may assume that $h \in \cartan_{(k,y)} \leq \cartan_{(h,x)}$ and $\cartan_{(k,y)} = \cartan_{(h,x)} \cap K$.

Now, a slice $Y_y$ at $y$ for the $G$-action on $M$ is isomorphic as a $G_y$-space to a $G_y$-invariant
subset of $Y_x$ by \cite[II. Corollary 4.6]{BredonBook}. Hence, the $\simeq$ classes in $\cartan_{(k,y)}$
(defined in terms of the action on $Y_y$) are unions of $\simeq$ classes in $\cartan_{(h,x)}$
(defined in terms of the action on $Y_x$) intersected with $\cartan_{(k,y)}$. By Proposition~\ref{prop:simeqProp},
$\cartan_{(h,x)}$ contains a finite number of $\simeq$ classes, each an open subset of a closed subgroup of
$\cartan_{(h,x)}$. Hence, we may assume that $V_{(h,x)}$ is small enough to only intersect $\simeq$ classes of
$\cartan_{(h,x)}$ whose closures contain $h$. Then the exponential map associated to the product metric on
$G_x \times Y_x$ maps the subset
\begin{equation}
  \label{eq:modspace}
  (N_{(h,x)})_K \cap \big( T_h (k^\bullet) \oplus (N_x)_{G_y} \big) \cap B_{(h,x)}
\end{equation}
onto $(V_{(h,x)})_K \cap \left(\overline{\cartan_{(k,y)}^{\ast}} \times (Y_x)_{G_y} \right)$.
Then \eqref{eq:modspace} is a semialgebraic subset of $N_{(h,x)}$ and is invariant
under the action of $t\in (0,1]$. Because there are only finitely many $\simeq$ classes in
$\cartan_{(k,y)}$, there are $l_1,\ldots ,l_N \in \cartan_{(k,y)}$
such that each group $l_j^\bullet$, $j=1,\ldots,N$,
has dimension less than $\dim k^\bullet$, and
\[
   \cartan_{(k,y)}^{\ast} = \overline{\cartan_{(k,y)}^{\ast}}
   \smallsetminus \bigcup_{j =1}^N l_j^\bullet \: .
\]
Then the exponential function maps the semialgebraic set
\begin{equation}
  \label{eq:modspace2}
  (N_{(h,x)})_K \cap \Big( \big( T_h (k^\bullet ) \smallsetminus \bigcup_{j =1}^N
  T_h (l_j^\bullet) \big) \oplus (T_x Y_x)_{G_y} \Big) \cap B_{(h,x)}
\end{equation}
onto $(V_{(h,x)})_K \cap \big(\cartan_{(k,y)}^{\ast} \times (Y_x)_{G_y} \big)$.
Finally, $(V_{(h,x)})_K \cap \big(\cartan_{(k,y)}^{\ast\ast} \times (Y_x)_{G_y} \big)$
is given by the finite union of such sets over $\simeq$ classes diffeomorphic to
$\cartan_{(k,y)}^\ast$. This set, and hence its $G$-saturation $\mathcal{U}_{(\tilde{k}, \tilde{y})}$,
has finitely many connected components by \cite[Thrm.~2.4.5]{BocCosRoyRAG}. We have hence
demonstrated the following.

\begin{lemma}
\label{lem:LocalStratLocFinite}
The partition $\mathcal{Z}$ of $U = GV_{(h,x)}$ given by Equation \eqref{eq:decUrep} is finite.
\end{lemma}

Restricting the inverse of the exponential map from $V_{(h,x)}$ to
$V_{(h,x)}\cap\inertianull{(G_x \ltimes Y_x)}$ yields an $H$-equivariant embedding
$\iota$ of a neighborhood of $(h,x)$ in $V_{(h,x)}\cap\inertianull{(G_x \ltimes Y_x)}$
into the normal space $N_{(h,x)}$, where the stratum of $(h,x)$ is mapped onto the subspace
\[
    N_{(h,x)}^H \cap \Big( T_h h^\bullet \times (T_x Y_x)^{G_x} \Big).
\]
Using the description of the image of each stratum
given in Equation \eqref{eq:modspace2}, one sees immediately that the homotopy
defined as multiplication in $N_{(h,x)}$ by scalars $t \in [0,1]$ contracts
the image of $\iota$ onto the origin preserving the image of $\iota$. The linearity
of the $H$-action on $N_{(h,x)}$ ensures that this homotopy is $H$-equivariant,
and scalars $t\in(0,1]$ preserve the images of strata as demonstrated above. We
therefore have proven the following.

\begin{proposition}
\label{prop:loccontractibilityinertiaspaces}
The inertia groupoid $\inertia{\sfG}$ of a proper Lie groupoid $\sfG$ satisfies
the local contractibility condition (LC) of Definition \ref{def:local-contractibility}.
\end{proposition}

The proof of the following is similar to that of \cite[Prop.~4.20]{FarsiPflaumSeaton}.

\begin{proposition}
\label{prop:LocalStratGermsCoincide}
The germs of the $\mathcal{U}\big( G(k,y) \big)$ coincide with the stratification.
That is, for $(\tilde{k}, \tilde{y}) \in U = GV_{(h,x)}$, the germs
$[\mathcal{U}\big( G(k,y) \big)]_{(\tilde{k},\tilde{y})}$,
$[\mathcal{U}_{(\tilde{k},\tilde{y})}]_{(\tilde{k},\tilde{y})}$
and $\mathcal{S}_{(\tilde{k},\tilde{y})}$ coincide.
\end{proposition}
\begin{proof}
As $\mathcal{S}_{(\tilde{k}, \tilde{y})}$ and $\mathcal{U}\big(G (k,y)\big)$
depend only on the orbit of $(\tilde{k}, \tilde{y})$ we assume
$(\tilde{k}, \tilde{y}) = (k, y) \in V_{(h,x)}$ and let $K = G_{(k,y)} \leq H$.
Fix a slice $Y_y$ at $y$ for the $G_x$-action on $Y_x$, and then
$Y_y$ is a slice at $y$ for the $G$-action on $M$ by \cite[II. Corollary 4.6]{BredonBook}.
Note that $Y_y$ may not be the image of $N_y$ under the exponential map, though it is
diffeomorphic to an open subset $B_y\subset N_y$ as a $G_y$-space. Then a slice $V_{(k,y)}$
for the $G_y$-action on $G_y\times Y_y$ is given by a neighborhood of $(k,y)$ in $K\times Y_y$ by
\cite[Proposition (3.1.1)(i)]{duistermaatkolk}. As $Y_y \subset Y_x$, we may assume
by shrinking $V_{(k,y)}$ if necessary that $V_{(k,y)}\subset V_{(h,x)}$, and as well that
$V_{(k,y)}$ does not intersect the product of $M$ with any connected component of $K$ except
$kK^\circ$. Choose a Cartan subgroup $\cartan_{(k,y)}$ of $K$ associated to $k$, and then we may
further shrink $V_{(k,y)}$ to assume by Proposition~\ref{prop:simeqProp}(\ref{SimeqIte3}) that
$V_{(k,y)}$ does not intersect (the product of $M$ with) the closure of any of the finitely many
$\simeq$ classes of $\cartan_{(k,y)}$ diffeomorphic but not equal to $\cartan_{(k,y)}^\ast$.
Let $Q := GV_{(k,y)}$; we claim that
\[
    \mathcal{U}\big( G(k,y) \big) \cap Q
    =
    G\left( V_{(k,y)}^K \cap (\cartan_{(k,y)}^\ast\times Y_y^{G_y})\right).
\]

Let $(\tilde{l}, \tilde{z}) \in \mathcal{U}\big( G(k,y) \big) \cap Q$. Then there is a
$\tilde{g}^\prime \in G$ such that $(l, z) := \tilde{g}^\prime(\tilde{l}, \tilde{z})
\in (V_{(h,x)})_K \cap (\cartan_{(k,y)}^{\ast\ast} \times (Y_x)_{G_y})$, and in particular
$\cartan_{(k,y)}$ is a Cartan subgroup of $K$ associated to $l$. As $(l, z) \in Q$, there is a
$g \in G$ such that $g(l, z) \in V_{(k,y)}$. Because $z \in Y_x$ and $gz \in Y_y \subset Y_x$,
we have $g\in G_x$. Then as $(l,z)\in V_{(h,x)}$ and $g(l,z)\in V_{(k,y)}\subset V_{(h,x)}$,
we have further that $g \in H$. Moreover, as $G_{(l,z)} = K$ and $g(l,z)\in V_{(k,y)}$ so that
$G_{g(l,z)} \leq K$, it follows that $G_{g(l,z)} = K$, i.e. $g(l,z)\in V_{(k,y)}^K$. In the same
way, as $G_z = G_y$, and as $V_{(k,y)}\subset G_y\times Y_y$ so that $gz\in Y_y$, we have
$G_{gz} = G_y$, i.e. $g(l,z)\in V_{(k,y)}^K \cap (K\times Y_y^{G_y})$. Finally, as
$l \in \cartan_{(k,y)}^{\ast\ast}$ and as $glg^{-1}$ is in $V_{(k,y)}$ and hence the same connected
component of $K$ as $k$, we may by Lemma~\ref{lem:CartanCongConComp} conjugate further by
a $\tilde{k}\in K\leq G_y$, which fixes $V_{(k,y)}^K$ and $Y_y^{G_y}$, so that
$\tilde{k}g(l,z) \in V_{(k,y)}^K \cap (\cartan_{(k,y)}^{\ast\ast}\times Y_y^{G_y})$.
But as $V_{(k,y)} \cap\cartan_{(k,y)}^{\ast\ast} \subset \cartan_{(k,y)}^\ast$,
$\tilde{k}g(l,z) \in V_{(k,y)}^K \cap (\cartan_{(k,y)}^\ast\times Y_y^{G_y})$.

Conversely, if
$(\tilde{j}, \tilde{w}) \in G\left( V_{(k,y)}^K \cap (\cartan_{(k,y)}^\ast\times Y_y^{G_y}) \right)$,
then there is a $\hat{g} \in G$ such that
$(j, w) := \hat{g}(\tilde{j}, \tilde{w}) \in V_{(k,y)}^K \cap (\cartan_{(k,y)}^\ast\times Y_y^{G_y})$.
Then as $V_{(k,y)} \subseteq V_{(h,x)}$ and $Y_y\subset Y_x$, we have
$(j, w) \in (V_{(h,x)})_K \cap \big(\cartan_{(k,y)}^\ast \times (Y_x)_{G_y} \big)$,
and so $(j, w) \in \mathcal{U}\big( G(k,y) \big)$.

Finally, restricting to a neighborhood of $(k,y)$ that only intersects one connected component
of $\mathcal{U}\big( G(k,y) \big)$, we see that the germ of $\mathcal{U}_{(k,y)}$ at $(k,y)$ coincides
with $\mathcal{S}_{(k,y)}$.
\end{proof}

Since the $\mathcal{S}_{(h,x)}$ are germs of smooth
$G$-submanifolds of $G \times M$, and the piece associated to a
point $(\tilde{k},\tilde{y}) \in U$ has the same set germ as
$\mathcal{S}_{(l, z)}$ at $(l, z) \in
\mathcal{U}_{(\tilde{k},\tilde{y})}$, it follows that the
pieces of $\mathcal{Z}$ are smooth submanifolds of $G \times M$
invariant under the $G$-action. We now verify that $\mathcal{Z}$ is a decomposition indeed.
The proof is similar to that of \cite[Prop.~4.22]{FarsiPflaumSeaton}.

\begin{proposition}
\label{prop:LocalStratFrontier} The decomposition $\mathcal{Z}$
satisfies the condition of frontier.
\end{proposition}
\begin{proof}
Suppose there are points $(h,x)$ and $(k,y)$ with
$\mathcal{U}\big( G(h,x) \big) \cap \overline{\mathcal{U}\big( G(k,y) \big)}\neq\emptyset$.
As the pieces of $\mathcal{Z}$ are defined to be connected components, it is sufficient to
show that $\mathcal{U}\big( G(h,x) \big) \cap \overline{\mathcal{U}\big( G(k,y) \big)}$,
which is obviously closed in $\mathcal{U}\big( G(h,x) \big)$, is also
open in $\mathcal{U}\big( G(h,x) \big)$. Note that we can assume with no loss of generality
that one of the points in question is $(h,x)$, the point used to define $U$, as we may
restrict consideration to a neighborhood of that point. Moreover, as the piece
$\mathcal{U}\big(G(h,x)\big)$ may be defined in terms of any point it contains,
we may assume that
$G(h,x) \subset \mathcal{U}\big( G(h,x) \big) \cap \overline{\mathcal{U}\big( G(k,y) \big)}$.
Similarly, we assume by choosing another representative of the orbit if necessary that
$(k,y) \subset V_{(h,x)}$. By Proposition \ref{prop:LocalStratGermsCoincide}, an open
neighborhood of $(h, x)$ in $\mathcal{U}\big( G(h,x) \big)$ is given by
$G\big(V_{(h, x)}^H \cap (\cartan_{(h, x)}^{\ast}\times Y_x^{G_x})\big)$
for a sufficiently small slice $V_{(h, x)}$ at $(h, x)$. We will show that
$G\big(V_{(h, x)}^H \cap (\cartan_{(h, x)}^{\ast}\times Y_x^{G_x})\big)$
is contained in $\overline{\mathcal{U}\big( G(k,y) \big)}$.

Let $K := G_{(k,y)} \leq H$. As explained in the proof of Lemma~\ref{lem:LocalStratLocFinite}
(before the statement of the Lemma), $h$ is in the connected component of $K$ containing $k$,
and we may choose Cartan subgroups $\cartan_{(h,x)}$ and $\cartan_{(k,y)}$
such that $h \in \cartan_{(k,y)} \leq \cartan_{(h,x)}$ and $\cartan_{(k,y)} = \cartan_{(h,x)} \cap K$.
Then we have $h \in \overline{\cartan_{(k',y)}^{\ast}}$ for some
$k'$ whose $\simeq$ class at $y$ is diffeomorphic to $\cartan_{(k,y)}^{\ast}$.
Similarly, as $G_y \leq G_x$, it follows that
$Y_x^{G_x} \subset \overline{(Y_x)_{G_y}}$.  From these observations, we have
\[
    (h,x) \in
    V_{(h,x)}^H \cap \Big(\overline{\cartan_{(k',y)}^{\ast}} \times Y_x^{G_x}\Big)
    \subset
    \overline{(V_{(h,x)})_K} \cap \big(\overline{\cartan_{(k',y)}^{\ast}}
        \times \overline{(Y_x)_{G_y}} \big).
\]

Now, let $l \in \cartan_{(h,x)}^{\ast}$ so that $Y_x^l = Y_x^h$.  In particular, as $h \in K \leq G_y$
and $y \in Y_x$, it follows that
$l \in G_y$.  Similarly, as $l \in \cartan_{(h,x)}$,  as $k \in \cartan_{(k,y)} \leq \cartan_{(h,x)}$,
and as $\cartan_{(h,x)}$ is abelian, we have $l(k,y) = (k,y)$ so that $l \in K$.  In particular,
$l \in \cartan_{(h,x)} \cap K = \cartan_{(k,y)}$.  This demonstrates
$\cartan_{(h,x)}^{\ast} \subset \cartan_{(k,y)}$.  As $Y_y \subset Y_x$,
we have that the relation $\simeq$ at $x$ implies $\simeq$ at $y$, so that the $\simeq$ classes at $y$
are the intersection with $\cartan_{(k,y)}$ of a (finite) union of $\simeq$ classes at $y$.
That is, using Proposition \ref{prop:simeqProp} (\ref{SimeqIte2}),
$\cartan_{(h,x)}^{\ast} \subset\overline{\cartan_{(k',y)}^{\ast}}$.  Then
\[
    V_{(h,x)}^H \cap \Big(\cartan_{(h,x)}^{\ast} \times Y_x^{G_x}\Big)
    \subset
    \overline{(V_{(h,x)})_K} \cap \big(\overline{\cartan_{(k',y)}^{\ast}}
        \times \overline{(Y_x)_{G_y}} \big)
    \subset \overline{\mathcal{U}\big( G(k,y) \big)}.
\]
Considering the $G$-saturations of both sides of this
inclusion, it follows that an open neighborhood of $(h, x)$ in
$\mathcal{U}\big( G(k,x) \big)$ is contained in
$\mathcal{U}\big( G(h,x) \big) \cap \overline{\mathcal{U}\big( G(k,y) \big)}$.
This completes the proof.
\end{proof}


\subsection{Proof of Theorem~\ref{thrm:InertiaStrat}}
\label{subsec:GpoidStratProof}

Let $\sfG$ be a proper Lie groupoid. We now demonstrate that the stratification given by Equation \ref{eq:MLocalStratGpoidDef}
is well-defined, i.e.~that it does not depend on the choice of a point in an orbit, a slice at that point, nor the riemannian metric.

\begin{proposition}
\label{prop:BasePointInvar}
Let $\sfG$ be a proper Lie groupoid,
and let $\obj{x}, \obj{y} \in \sfG_0$ be points in the same orbit.
Let $\arr{g}\in\sfG_1$ such that $s(\arr{g}) = t(\arr{g}) = \obj{x}$, and let $\arr{h}\in\sfG_1$ such that
$s(\arr{h}) = \obj{x}$ and $t(\arr{h}) = \obj{y}$. Put $\arr{g}^\prime:= \arr{h}\arr{g}\arr{h}^{-1}$.
Then $\mathcal{S}_{\arr{g}} = \mathcal{S}_{\arr{g}^\prime}$.  In particular,
$\mathcal{S}_{\arr{g}}$ does not depend on the choice of a slice $Y_{\obj{x}}$  and hence does not depend on
the choice of transversally invariant metric.
\end{proposition}
\begin{proof}
Choose slices $Y_{\obj{x}}$ and $Y_{\obj{y}}$ for $\sfG$ at $\obj{x}$ and $\obj{y}$, respectively.
Then there are identifications $\sfG_{|Y_{\obj{x}}} \cong \sfG_{\obj{x}}\ltimes Y_{\obj{x}}$ and
$\sfG_{|Y_{\obj{y}}} \cong \sfG_{\obj{y}}\ltimes Y_{\obj{y}}$ by \cite[Thm.~3.3]{PflPosTanGOSPLG}.
Under these identifications let $\arr{g} = (h,\obj{x})$ and $\arr{g}^\prime = (k,\obj{y})$
for some $h\in \sfG_{\obj{x}}$ and  $k\in \sfG_{\obj{y}}$.
Choose a local bisection $\sigma \co U \to \sfG_1$ defined on an open neighborhood of $\obj{x}$ such that
$\sigma (\obj{x}) = \arr{h}$ and such that $t\circ\sigma_{|Y_{\obj{x}}}$ induces a diffeomorphism from
$Y_{\obj{x}}$ to $Y_{\obj{y}}$. The
existence of such a bisection, after possibly shrinking  $Y_{\obj{x}}$ and $Y_{\obj{y}}$, is guaranteed by
\cite[Prop.~3.9 \& proofs of Lemmata 5.1 \& 5.2]{PflPosTanGOSPLG}.
Then we obtain an isomorphism $\Psi\co \sfG_{\obj{x}}\ltimes Y_{\obj{x}} \to \sfG_{y}\ltimes Y_{y}$
which is given by the composition of the diffeomeorphism
\[
    \sfG_{|Y_{\obj{x}}}  \to \sfG_{|Y_{\obj{y}}}, \quad \arr{k}
    \mapsto
    \big(\sigma(t(\arr{k}))\big)\arr{k} \big(\sigma(s(\arr{k}))\big)^{-1}
\]
with the identifications $\sfG_{|Y_{\obj{x}}} \cong \sfG_{\obj{x}}\ltimes Y_{\obj{x}}$ and
$\sfG_{|Y_{\obj{y}}} \cong \sfG_{\obj{y}}\ltimes Y_{\obj{y}}$. Note that $\Psi$
obviously restricts to a diffeomorphism from $\inertianull{(\sfG_{\obj{x}}\ltimes Y_{\obj{x}})}$ onto
$\inertianull{(\sfG_{\obj{y}}\ltimes Y_{\obj{y}})}$.  Moreover, by construction,
$\Psi (h,\obj{x})$ is the image of
\[
  \big(\sigma (\obj{x})\big)\arr{g} \big(\sigma (\obj{x})\big)^{-1} = \arr{h}\arr{g}\arr{h}^{-1} = \arr{g}^\prime
\]
under the identification $\sfG_{|Y_{\obj{y}}} \cong \sfG_{\obj{y}}\ltimes Y_{\obj{y}}$, hence
$\Psi (h,\obj{x}) = (k,\obj{y})$.
%
%
Now we choose a Cartan subgroup $\cartan_{(h,\obj{x})}$ of $\centralizer_{\sfG_{\obj{x}}}\!(h)$ associated to $h$.  Then
as $\Psi$ restricts to an isomorphism of the isotropy group in $\sfG$ of $\obj{x}$ onto the isotropy group
of $\obj{y}$, $\Psi(\cartan_{(h,\obj{x})} \times\{\obj{x}\})$  a Cartan subgroup of $\centralizer_{\sfG_{\obj{y}}}\!(k)$
associated to $k$.
Moreover, we claim that $\Psi(\cartan_{(h,\obj{x})}^{\ast} \times\{\obj{x}\}) = \cartan_{(k,\obj{y})}^{\ast} \times\{\obj{y}\}$.
To see this, note that $\Psi_0\co Y_{\obj{x}} \to Y_{\obj{y}}$ is a diffeomorphism that is
$\cartan_{(h,\obj{x})}$-equivariant with respect to the isomorphism $\tau\co \cartan_{(h,\obj{x})} \to \cartan_{(k,\obj{y})}$
given by $\tau(s) = \pi_1\circ\Psi(s,\obj{x})$, where $\pi_1$ denotes the projection
$\pi_1\co\sfG_{\obj{y}}\times Y_{\obj{y}}\to\sfG_{\obj{y}}$.
Then for $\obj{z} \in Y_{\obj{y}}$ and $s \in \cartan_{(h,\obj{x})}$,
$\tau(s)\obj{z} = \tau(s)(\Psi_0\circ \Psi_0^{-1}(\obj{z})) = \Psi_0(s (\Psi_0^{-1}(\obj{z})))$.
Hence $\tau(s)\obj{z} = \obj{z}$ if and only if $s (\Psi_0^{-1}(\obj{z}))= \Psi_0^{-1}(\obj{z})$, from which it follows
that $\Psi_0$ maps the set of points of $Y_{\obj{x}}$ fixed by $s\in\cartan_{(h,\obj{x})}$  onto the fixed set of
$\tau(s)$ in $Y_{\obj{y}}$.
The isomorphism $\Psi$ then maps $\mathcal{S}_{(h,\obj{x})}$
onto $\mathcal{S}_{(k,\obj{y})}$.  This proves the claim. If $\obj{x} = \obj{y}$, this argument shows that
the stratification does not depend on the choice of the slice $Y_{\obj{x}}$, hence is also independant on
the choice of a transversally invariant riemannian metric.
\end{proof}

Given Theorem \ref{thrm:InertiaStratGMnfld} and Proposition \ref{prop:BasePointInvar}, it remains only to
verify Whitney (b)-regularity; this will be done in the next section. Recall that Whitney (b)-regular stratified spaces are topologically locally trivial,
see \cite[Cor.~3.9.3]{PflaumBook}, and that strata of Equation \ref{eq:MLocalStratGpoidDef} contain orbits because they are defined as saturations.
Hence $\inertia{G} = \sfG\ltimes\inertianull{\sfG}$ is a differentiable stratified groupoid by
Proposition \ref{prop:StratTranslationGpoid}. That $\sfG$ is a sliceable differentiable groupoid
follows from \cite[Prop.~3.9 \& Cor.~3.11]{PflPosTanGOSPLG}. Similarly, by the definition of the
stratification of $\inertianull{\sfG}$ in terms of stratifications of the inertia spaces of slices
$\sfG_{\obj{x}}\ltimes Y_{\obj{x}}$, $\obj{x}\in \sfG_0$, the inertia groupoid $\inertia{\sfG}$ is even a sliceable
differentiable stratified groupoid. Moreover, since the elements of a stratum of a slice
$Y_{\obj{x}}$ all have the same $\sfG_{\obj{x}}$-orbit type, the inertia space is a differentiable
stratified space by Proposition \ref{prop:LocTransStrat}, and the orbit map is a stratified surjective
submersion.


\subsection{Whitney (b)-regularity}
\label{subsec:Whitney}

Here, we complete the proof of Theorem \ref{thrm:InertiaStrat} by demonstrating
that the stratifications of $\inertianull{\sfG}$ and $|\inertia{\sfG}|$ are
Whitney (b)-regular.
The proof follows \cite[Thm.~4.3.7]{PflaumBook} and \cite[Prop.~4.23]{FarsiPflaumSeaton}.
Roughly, the proof involves giving a parameterization of a neighborhood of a point in
$\inertia{\sfG}$ and its tangent space sufficient to describe the secants of points
in neighboring strata. Note that in our argument we use that the pieces satisfy
the condition of frontier, which was shown above.

\begin{proposition}
\label{prop:LocalStratWhitneyB}
Let $\sfG$ be a proper Lie groupoid. The orbit Cartan type stratifications of the loop space
$\inertia{\sfG}$ and the inertia space $|\inertia{\sfG}|$ are both Whitney (b)-regular.
\end{proposition}
\begin{proof}
Because the claim is local, we may assume that the groupoid $\sfG$ is given by
the product of $O\times O \rightrightarrows O$ and $G\ltimes Y$ where
$G$ is the isotropy group $\sfG_x$ of some point $x \in \sfG_0$, $O$ is an open
neighborhood of $x$ in its orbit, and $Y$ is a slice through $x$.
Let $(h,x) \in \inertia{\sfG}$, $H = \centralizer_{G}\!(h)$, and $V_{(h,x)}$
a slice at $(h,x)$ for the $G$-action on $G\times Y$ of the form $\exp (B_{(h,x)})$, where $B_{(h,x)}$ is
a ball around the origin in the normal space $N_{(h,x)}$; recall that $V_{(h,x)}$ is then an open neighborhood
of $(h,x)$ in $H\times V$. Now let us denote by $\mathcal{Z}$ the decomposition
of $\inertianull{\sfG}$ obtained by taking the saturations of the sets defined through Eq.~\eqref{eq:decUrep},
which amounts to taking their products with $O$. Let $R$ be the piece of $\mathcal{Z}$
containing $(h,x)$, i.e.~the set of points of the form $((o,o),(l,z))$ where $o \in O$ and
$(l,z) \in V_{(h,x)}^H \cap \big(\cartan_{(h,x)}^{\ast}\times Y^{G}\big)$.
We show that for any stratum $S \in \mathcal{Z}$ with $(h,x)\in \overline{S}$, the Whitney condition (b)
is satisfied at $(h,x)$ for the pair of strata $(R,S)$.  To describe the stratum $S$ in some more detail,
consider an orbit $G(k,y)$ for $(k, y) \in S$. As in the proof of Proposition \ref{prop:LocalStratFrontier},
we may choose the representative $(k,y)$ of the orbit $G(k,y)$ such that $(k,y) \in V_{(h,x)}$,
$h \in \cartan_{(k,y)} \leq \cartan_{(h,x)}$, and $h \in \overline{\cartan_{(k',y)}^{\ast}}$
for some $k'$. In particular, we then have
$K  \leq H$ for the isotropy group $K := \centralizer_{G_y}\!(k)$ of $(k,y)$ and $G_y \leq G$. As
shown above, $S$  coincides with the connected component of $\mathcal{U}\big( G(k,y) \big)$ containing
$(k,y)$.

Suppose now that $((u_i,u_i),(h_i, x_i))_{i \in \N}$ is a sequence in $R$ and $((o_i,o_i),(k_i, y_i))_{i \in \N}$
a sequence in $S$,
and that both sequences converge to $((x,x),(h,x))$. Assume in addition that in a smooth chart around
$((x,x),(h,x))$ the secant lines
\[ \ell_i = \overline{((u_i,u_i),(h_i,x_i)),((o_i,o_i),(k_i,y_i))} \]
converge to a straight line $\ell$, and the tangent spaces
$T_{((o_i,o_i),(k_i,y_i))}S$ converge to a subspace $\tau$. Then we must show that $\ell \subset \tau$.

Note that the hypotheses imply that
$((x,x),(h,x)) \in \mathcal{U}\big( G(h,x) \big)\cap \overline{\mathcal{U}\big( G(k,y) \big)}$. By
the proof of Proposition~\ref{prop:LocalStratFrontier} and
the choices of $(k,y)$ and $\cartan_{(k,y)} \subset K$ we
obtain the relation
\begin{equation}
\label{eq:closurerel1}
    V_{(h,x)}^H \cap \Big(\cartan_{(h,x)}^{\ast} \times Y^{G}\Big)
    \subset
    \overline{(V_{(h,x)})_K} \cap \big(\overline{\cartan_{(k',y)}^{\ast}}
        \times \overline{(Y)_{G_y}} \big).
\end{equation}

Denote by $\mathfrak{g}_x$ the Lie algebra of $G$, by $\mathfrak{h}$ the Lie algebra of $H$,
and let  $\mathfrak{m}$ denote the orthogonal complement of $\mathfrak{h}$ in $\mathfrak{g}_x$
with respect to the initially chosen bi-invariant metric on $G$. Then there is a neighborhood
$U \subset \sfG_0 \cong O \times_H V_{(h,x)}$ of $(h,x)$ such that
\[
    \Psi \co   U   \longrightarrow  O \times \mathfrak{m} \times N_{(h,x)} , \:
    [o, \exp_{|\mathfrak{m}} \xi, \exp_{(h,x)}(v)]\longmapsto (o, \xi, v)
\]
is a smooth chart at $((x,x),(h,x))$, where $\exp_{|\mathfrak{m}}$ denotes the restriction of the exponential
map of the Lie group $G$ to $\mathfrak{m}$, and $\exp_{(h,x)}$ the exponential function restricted to
the open ball $B_{(h,x)} \subset N_{(h,x)}$. After possibly shrinking $U$
there is an open neighborhood $Q$ of $H$ in $G$ such that
\[
    \Psi\left(O\times Q\left( V_{(h,x)}^H \cap
    (\cartan_{(h,x)}^{\ast}\times Y^{G})\right) \right)
    \subset
    O\times \mathfrak{m} \times \left( N_{(h,x)}^H \cap T_{(h,x)}
    (\cartan_{(h,x)}^{\ast}\times Y^{G})\right).
\]
We may assume that the sequences $((u_i,u_i),(h_i, x_i))_{i \in \N}$ and $((o_i,o_i),(k_i, y_i))_{i \in \N}$
are contained in $U$.
Since $((o_i,o_i),(k_i, y_i)) \in \mathcal{U}\big( G(k,y) \big)$, one knows that
\[
    \Psi((o_i,o_i),(k_i,y_i))
    \in O \times \mathfrak{m} \times H \left(
    (V_{(h,x)})_K \cap \left( \cartan_{(k,y)}^{\ast\ast} \times Y^{G}\right)\right).
\]
Recall that $\cartan_{(k,y)}^{\ast\ast}$ consists of a finite collection of diffeomorphic
$\simeq$ classes in $\cartan_{(k,y)}$.  Moreover, by Lemma~\ref{prop:simeqProp}(\ref{SimeqIte3}),
each such $\simeq$ class is disjoint from the closures
of the other classes.  By passing to a subsequence, we may assume without loss of generality
that each $k_i$ is in one fixed class, i.e.
\[
    ((o_i,o_i),(k_i, y_i)) \in O\times G \left((V_{(h,x)})_K \cap
    \left( \cartan_{(k^\prime,y)}^{\ast} \times Y^{G}\right)\right)
\]
for all $i$ and some fixed $k^\prime \in \cartan_{(k,y)}$.

Note that $\lim o_i = \lim u_i = x$.  Moreover, each piece
of $\mathcal{Z}$ is a product of a piece in $GV_{(h,x)} \subset G\times Y$
with the diagonal in $O\times O$, so we may project onto $G\times Y$ and ignore
the $O$-factor.
Choose $\tilde{l}_i \in G$ such that $(\tilde{k}_i,
\tilde{y}_i):= \tilde{l}_i (k_i, y_i) \in (V_{(h,x)})_K$ for
all $i\in \N$. Put $(\tilde{h}_i,\tilde{x}_i):= l_i (h_i,x_i)$.
After possibly passing to a subsequence, $(\tilde{l}_i)_{i\in
\N}$ converges to some $\tilde{l} \in H$,  the secant lines
$\tilde{\ell}_i =
\overline{(\tilde{h}_i,\tilde{x}_i),(\tilde{k}_i,\tilde{y}_i)}$
converge to a straight line $\tilde{\ell}$, and the tangent
spaces $T_{(\tilde{k}_i,\tilde{y}_i)}S$ converge to a subspace
$\tilde{\tau}$. By definition, and since $\tilde{l}_i
T_{(k_i,y_i)} S = T_{(\tilde{k}_i,\tilde{y}_i)}S$ for all $i$,
one obtains $\tilde{\ell} = \tilde{l}  \ell$, and $\tilde{\tau}
= \tilde{l} \tau$. Hence, the first claim is shown, if
$\tilde{\ell} \subset \tilde{\tau}$. Without loss of
generality we may therefore assume that for all $i \in \N$
\begin{equation}
\label{eq:reductionky}
  (k_i, y_i) \in  (V_{(h,x)})_K \cap \big( \cartan_{(k^\prime,y)}^{\ast}
        \times Y^{G} \big) ,
\end{equation}
and then show $\ell \subset \tau$ for the sequences $(k_i,
y_i)_{i\in \N}$ and $(h_i,x_i)_{i\in  \N}$.

Eq.~\eqref{eq:reductionky} now means in particular that
\[
    \Psi (k_i,y_i)
    \in \{ 0 \} \times \Big(
    (N_{(h,x)})_{K} \cap
    \exp^{-1}_{(h,x)}\big(\cartan_{(k^\prime,y)}^{\ast} \times Y^{G}\big)\Big).
\]
Since $\overline{\cartan_{(k^\prime,y)}^{\ast}}$ is an
open and closed subset of a closed subgroup of $G$ and also contains $h$, the set
\[
    V:= N_{(h,x)} \cap T_{(h,x)}
    \Big( \big(\overline{\cartan_{(k^\prime,y)}^{\ast}}\big) \times Y^{G}\Big)
\]
is a subspace of $N_{(h,x)}$. Let $W$ be the orthogonal
complement of the invariant space $V^H$ in $V$ with respect to
the $H$-invariant scalar product induced from $V_{(h,x)}$. Then
the image  under the chart $\Psi$ of every element of $G\big(
V_{(h,x)}^H \cap (\cartan_{(h,x)}^{\ast}\times Y^{G})\big)\cap U$ and every
$(k_i, y_i)$ is contained in
\[
    \mathfrak{m} \times (W_K \cup \{ 0 \}) \times V^H .
\]
With respect to this decomposition, $(h,x)$ has coordinates
$(0, 0, 0)$, each element of $G\left( V_{(h,x)}^H \cap(\cartan_{(h,x)}^{\ast}\times Y^{G})\right)$
has coordinates contained in $\mathfrak{m} \times 0 \times V^H$, and each
sequence element $(k_i, y_i)$ has coordinates contained in
$\{ 0 \}  \times W_{K} \times V^H$.  In particular, let
\[
    \Psi(k_i,y_i)
    =
    (0 , w_i, v_i)
\]
for every $i$.  Since $W_{K}$ is invariant under multiplication by non-vanishing scalars, we have
\[
\begin{split}
    (\xi, w, v)
     :=\, &
    \lim\limits_{i\to\infty} \frac{ \Psi(k_i,y_i) - \Psi(h_i,x_i) }
        {\| \Psi(k_i,y_i) - \Psi(h_i,x_i) \|}
    \, \in \mathfrak{m} \times \overline{W_{K}} \times V^H \ .
\end{split}
\]
By compactness of the unit sphere in $W$, the sequence
$\frac{w_i}{\| w_i \|}$ converges to some $\hat{w} \in SW$
after possibly passing to a subsequence. Then $w =  \| w
\|\hat{w}$. Since $W_{K}$ is invariant by non-vanishing
scalars, we have
\[
    \mathfrak{m} \times \operatorname{span}\: \hat{w} \times V^H \subset \tau,
\]
and
\[
    \ell = \operatorname{span} \: (\xi, \hat{w}, v) \subset \tau,
\]
proving the first claim.

Now let us show that the orbit Cartan type stratification of
$|\inertia{\sfG}|$ is Whitney (b)-regular as well. To this
end let us first choose a Hilbert basis of $H$-invariant
polynomials $p_1,\ldots , p_\kappa \co
\big(N_{(h,x)}^H\big)^\perp \rightarrow \R$ of the orthogonal
complement of the invariant space $N_{(h,x)}^H$ in $N_{(h,x)}$.
Next let $p_{\kappa + 1},\ldots , p_N \co N_{(h,x)}^H \rightarrow
\R$ with $N=\kappa +\dim N_{(h,x)}^H$ be a linear coordinate
system of the invariant space. We can even choose these $p_i$
in such a way that $p_{\kappa + 1},\ldots , p_{\kappa + \dim
V^H}$ is a linear coordinate system of $V^H$. By construction,
$p_1,\ldots , p_N$ then is a Hilbert basis of the normal space
$N_{(h,x)}$. Denote by $p\co N_{(h,x)} \rightarrow \R^N$ the
corresponding Hilbert map. Recall that $p$ induces a chart of
$|\inertia{\sfG}|$ over $G \backslash U$ by
\[
  \widehat{\Psi}\co G \backslash U \rightarrow \R^N, \: G \exp_{(h,x)} (v) \mapsto p(v) .
\]
Note that by $H$-invariance of $p$ and since for every orbit in
$U$ there is a representative in $V_{(h,x)}$, the chart
$\widehat{\Psi}$ is well-defined indeed. A decomposition of
$\widehat{U}:= \widehat{\Psi} ( G \backslash U )$ inducing the
orbit Cartan type stratification on $G \backslash U$  is given
by
\[
  \widehat{\mathcal{Z}}:=
  \big\{ \widehat{\Psi} (G\backslash(S\cap (G\times Y)) \mid S \in \mathcal{Z} \big\} .
\]
Let $\widehat{S}\in \widehat{\mathcal{Z}}$ denote the stratum
containing the orbit $G(h,x)$, and $\widehat{S}\in \mathcal{Z}$
a stratum $\neq \widehat{R}$ such that $G(h,x)$ lies in the
closure of $\widehat{S}$. Now consider sequences of orbits
$\big( G(h_i,x_i) \big)_{i\in \N}$ in $\widehat{R}$ and
$\big( G(k_i,y_i) \big)_{i\in \N}$ in  $\widehat{S}$ such that both
sequences converge to $G(h,x)$. Moreover, assume  that the
sequence of secants $ \overline{\widehat{\Psi} (G(h_i,x_i)) , \widehat{\Psi} (G(k_i,y_i))}$
converges to a line $\widehat{\ell}$, and that the sequence of tangent spaces
$T_{\widehat{\Psi} (G(k_i,y_i))} \widehat{S}$ converges to some
subspace $\widehat{\tau} \subset \R^N$. Using notation from
before, we can choose representatives $(h_i,x_i)$ and
$(k_i,y_i)$  having coordinates in
$ \mathfrak{m} \times (W_K \cup \{ 0 \}) \times V^H \subset N_{(h,x)}$ such that
\begin{equation}
\label{eq:coordinates}
\begin{split}
  \Psi (h_i,x_i) & \, = (0,0, v_i^\prime ) \in \{ 0 \}  \times  \{ 0 \} \times V^H
  \: \text{ and } \\
  \Psi (k_i,y_i) & \, = (0,w_i,v_i) \in \{ 0 \} \times W_K  \times V^H .
\end{split}
\end{equation}

Next observe that by the Tarski--Seidenberg Theorem, the stratum
$\widehat{S}$ is semialgebraic as the image of the
semialgebraic set $(W_K \times V^H) \cap B_{(h,x)}$ under the
Hilbert map $p$. By the same argument, $p(W_K)$ is
semialgebraic, too, and an analytic manifold, since $p(W_K)
\cong \normalizer_H(K)\backslash W_K \cong H \backslash W_{(K)}$.
Moreover, the equality
\[
  \widehat{S} = ( p(W_K) \times V^H) \cap p(B_{(h,x)} )
\]
holds true, where we have canonically identified $V^H$ with its
image under the Hilbert map $p$. By Eq.~\eqref{eq:coordinates},
this implies that
\begin{equation}
\label{eq:tangentlimits}
  \widehat{\tau} =
  \lim_{i\rightarrow \infty} T_{\widehat{\Psi} (G(k_i,y_i))} \widehat{S} =
  \lim_{i\rightarrow \infty} T_{p(w_i)}p(W_K) \times V^H .
\end{equation}
Since $p(W_K)$ is semialgebraic and an analytic manifold,
\cite[Prop.~3, p.~103]{LojESA} by {\L}ojasiewicz entails that
$p(W_K)$ satisfies the Whitney condition (b) over the origin. This
means after possibly passing to subsequences, that $\ell_{W_K}
\subset \tau_{W_K}$, where $\ell_{W_K}$ is the limit line of
the secants $\overline{p(w_i),0}$, and  $\tau_{W_K}$ the limit
of the tangent spaces $T_{p(w_i)}p(W_K)$ for $i\rightarrow
\infty$. By Eqs.~\eqref{eq:coordinates} and
\eqref{eq:tangentlimits} this implies that
\[
  \widehat{\ell} \subset \ell_{W_K}  \times V^H  \subset \tau_{W_K} \times V^H = \widehat{\tau} .
\]
This finishes the proof.
\end{proof}

For a translation groupoid $\sfG = G\ltimes M$, if $S$ is a stratum of $\inertianull{(G\ltimes M)}$,
then by Lemma \ref{lem:DSGTrivConsequences} \eqref{Ite1} and \eqref{Ite3}, the induced stratum of the arrow space is
$G\times S$, and the corresponding stratum of $(\inertia{G})_1 \sttimes(\inertia{G})_1$ is $G\times G\times S$.
Hence, Whitney (b)-regularity of the stratification of $\inertianull{(G\ltimes M)}$ implies Whitney (b)-regularity of the stratifications
of the arrow space and $(\inertia{G})_1 \sttimes(\inertia{G})_1$.

Note that $\inertianull{\sfG}$ is clearly topologically locally trivial because
of its description in slices. Therefore,
by Proposition \ref{prop:StratTranslationGpoid}, the inertia groupoid
$\inertia{\sfG}$ is a differentiable stratified groupoid.
Now, recall that the loop space $\inertianull{\sfG}$ is a
differentiable subspace of the smooth manifold $\sfG_1$ and that the space of arrows
$\inertia{\sfG}_1 = \sfG_1 \fgtimes{s}{t} \inertianull{\sfG}$ is a differentiable
subspace of the smooth manifold $\sfG_1 \fgtimes{s}{t} \sfG_1$. Similarly,
if $\obj{x}, \obj{y} \in \sfG_0$ are in the same orbit, then the slices
$Y_\obj{x}$ and $Y_\obj{y}$ for $\sfG$ can be chosen such that $\sfG_{|Y_\obj{x}}$
and $\sfG_{|Y_\obj{y}}$ are isomorphic by \cite[Lemma 5.1]{PflPosTanGOSPLG}.
This defines a diffeomorphism between the arrow spaces of $\sfG_{|Y_\obj{x}}$ and
$\sfG_{|Y_\obj{y}}$ (which are both smooth manifolds) whose restriction defines
an isomorphism between $\inertia{\sfG_{|Y_\obj{x}}}$ and $\inertia{\sfG_{|Y_\obj{y}}}$.
It follows that the inertia groupoid $\inertia{G}$ satisfies conditions
(LT\ref{it:ExistenceBisections}) to (LT\ref{it:StratifiedSlice})
in Definitions \ref{def:LocTransDiffGroupoid} and \ref{def:LocTransStrat},
hence is sliceable.
Moreover, by Proposition \ref{prop:loccontractibilityinertiaspaces}, the inertia
groupoid $\inertia{G}$ satisfies the local contractibility condition of
Definition \ref{def:local-contractibility}.

It is straightforward to verify that a weak equivalence $f = (f_0, f_1) \co\sfG\to\sfH$ of proper Lie groupoids
induces a weak equivalence $\inertia{\sfG}\to\inertia{\sfH}$ given by the restriction
of $f_1$ to the loop spaces. In particular, because the stratification of $\inertianull{\sfG}$
is defined in terms of slices for $\sfG$, and the representation of the isotropy group on a
slice is Morita invariant, the stratification
of $\inertianull{\sfG}$ is obviously the pullback of the stratification of
$\inertianull{\sfH}$ via $f_1$. Moreover, as the stratification of the inertia space
$|\inertia{\sfG}|$ can be defined locally in terms of the actions of isotropy
groups on slices, it is as well Morita invariant. This means that the isomorphism between
$|\inertia{\sfG}|$ and $|\inertia{\sfH}|$ from Proposition \ref{prop:MoritaHomeo}
is an isomorphism of differentiable stratified spaces.
We summarize these observations in the following.

\begin{theorem}
\label{thrm:InertiaDeRham}
Let $\sfG$ be a proper Lie groupoid. Then the inertia groupoid $\inertia{\sfG}$
is a proper reduced structurally (b)-regular
differentiable stratified groupoid.
Moreover the inertia groupoid $\inertia{\sfG}$ is sliceable
and satisfies the local contractibility condition.
Finally, the inertia space $|\inertia{\sfG}|$ inherits from $\inertianull{\sfG}$
via the canonical projection $\pi:\inertianull{\sfG} \to
|\inertia{\sfG}|$
a Whitney (b)-regular stratification.
\end{theorem}

\appendix

\section{Differentiable stratified spaces}
\label{app:DiffStratSpaces}

In this appendix we describe the category
of differentiable stratified spaces used throughout this paper. Our notion
of differentiable spaces is that of \cite{NGonzalezSanchoBook} to which we refer the reader
for more details. For the definition of stratified
spaces we follow Mather \cite{MatherStratMap} and \cite[Chap.~1]{PflaumBook}, except
that we relax the assumption that the spaces under consideration are Hausdorff and
only require that they are locally Hausdorff. Hence, a stratified space with smooth
structure as defined in \cite[Chap.~1]{PflaumBook} or a differentiable stratified
space as defined in \cite{FarsiPflaumSeaton} corresponds to a Hausdorff
differentiable stratified space as defined here.
Note that in addition to \cite{PflaumBook,FarsiPflaumSeaton} various other concepts of
structure sheaves respectively structure algebras of smooth functions over
stratified spaces have been introduced in the literature. See for example
the work by Kreck \cite{KreDAT} on stratifolds,
by Lusala--\'{S}niatycki \cite[Sec.~4]{LusalaSnia} on
stratified subcartesian spaces, by Watts
\cite{WattsThesis,WattsOrbifold} on differential spaces, and finally by
Somberg--V\^an L\^e--Vanzura \cite{VanLeSombergVanzura,VanLeSombergVanzuraConical}
on smooth structures on locally conic stratified spaces.

\subsection{Differentiable spaces}
\label{ap:DiffSpaces}
\begin{definition}[{\cite[Chap.~3]{NGonzalezSanchoBook}}]
\label{def:DiffSpace}
Let $(X, \mathcal{O})$ be a locally $\R$-ringed space which we always assume to
be commutative.
One says that $(X, \mathcal{O})$ is an \emph{affine differentiable space}, if
there is a closed ideal $\mathfrak{a} \subset\mathcal{C}^\infty(\R^n)$ such that
$(X, \mathcal{O})$ is isomorphic as a ringed space to the real spectrum of
$\mathcal{C}^\infty(\R^n)/\mathfrak{a}$ equipped with its structure sheaf, which
associates to each open set its localization over that set. Here, we consider the
unique topology with respect to which $\mathcal{C}^\infty(\R^n)$ is a Frech\'{e}t algebra.
A locally $\R$-ringed space $(X, \mathcal{O})$ is a \emph{differentiable space} if,
for each $x \in X$, there is an open neighborhood $U$ of $X$ such that the restriction
$(U, \mathcal{O}|_U)$ is an affine differentiable space. A differentiable space
is \emph{reduced} if for each open subset $U$ of $X$ the map
$\mathcal{O}(U) \to \mathcal{C}(U)$ defined by the evaluation map is injective.

A \emph{morphism of differentiable spaces}
$(f,\varphi) \co (X, \mathcal{O}_X) \to (Y, \mathcal{O}_Y)$ consists of a continuous map
$f\co X \to Y$ and a morphism
 $\varphi \co\mathcal{O}_Y\to f_\ast \mathcal{O}_X$ of sheaves
of $\R$-algebras such that for each $x\in X$ the induced morphism
on the stalks $\varphi_x \co\mathcal{O}_{Y,f(x)}\to\mathcal{O}_{X,x}$ is local,
i.e.~maps the maximal ideal $\mathfrak{m}_{f(x)} \subset \mathcal{O}_{Y,f(x)}$
to the maximal ideal $\mathfrak{m}_x \subset \mathcal{O}_{X,x}$.
\end{definition}

Note that if $ (X, \calC^\infty_X) $ and $(Y, \mathcal{C}^\infty_Y)$ are reduced,
a morphism of differentiable spaces
$(f,\varphi)\co (X, \calC^\infty_X) \to (Y, \calC^\infty_Y) $
is fully determined by the map $f \co X \to Y$. The sheaf morphism
$\varphi$ is given in this case over each open $V \subset Y$ by the pullback
map $f^*: \calC^\infty_Y (V) \to  \calC^\infty_X(f^{-1} (V))$, $g \mapsto g\circ f_{|V}$.
We therefore sometimes call a morphism between reduced differentiable spaces
a \emph{smooth map}, and just denote it by the underlying map $f$.

By \cite[Thm.~3.23]{NGonzalezSanchoBook}, a differentiable space $(X,\calO)$ is reduced if
and only if each $x \in X$ is contained in an open neighborhood $V$ isomorphic as a
differentiable space to a locally closed subset of the affine space $\R^n$
with structure sheaf given by restrictions of smooth functions from $\R^n$.
We refer to such a $V$ as an \emph{affine neighborhood of} $x$,
and call an embedding $\iota : V \hookrightarrow \R^n$ such that
$(\iota,\iota^*) : (V,\calO_{|V}) \rightarrow (\iota (V), \calC^\infty_{|\iota (V)})$ is an
isomorphism of locally ringed spaces a \emph{singular chart} (of \emph{rank} $n$) for $X$.

We often denote the structure sheaf of a reduced differentiable space $X$ by $\calC^\infty_X$
or shortly by $\calC^\infty$, if no confusion can arise.
By a smooth submanifold of a differentiable space $X$, we mean a differentiable subspace
whose differentiable structure is that of a smooth manifold in the usual sense.

\subsection{Stratified spaces}
\label{ap:StratSpaces}
Let $X$ be a paracompact separable locally Hausdorff topological space.
A \emph{decomposition} $\mathcal{Z}$ of $X$ is a locally finite partition of
$X$ into locally closed subspaces such that each $S\in\mathcal{Z}$ is a
countable union of smooth (not necessarily Hausdorff) manifolds
such that the following condition of frontier is satisfied:
\begin{itemize}
\item[(CF)] If $R\cap\overline{S}\neq\emptyset$ for $R, S\in\mathcal{Z}$,
          then $R\subset\overline{S}$.
\end{itemize}
If $R\subset\overline{S}$, one writes $R \leq S$ and says that $R$ is incident to $S$.
The incidence relation is an order relation on $\mathcal{Z}$. The elements of a decomposition
$\mathcal{Z}$ are called its \emph{pieces}.

In the following we provide a generalization of the definition of a stratification by
Mather \cite{MatherStratMap}  to the case of a locally
Hausdorff space; cf.~also \cite[Sec.~1.2]{PflaumBook}.

\begin{definition}
\label{def:Stratification}
Let $X$ be a paracompact separable and locally Hausdorff topological space. A \emph{stratification} of $X$ is an
assignment to each $x \in X$ of a germ $\mathcal{S}_x$ of subsets of $X$ at $x$ such that
for each $x \in X$ there is a Hausdorff neighborhood $U$ of $x$ in $X$ and a decomposition
$\mathcal{Z}$ of $U$ with the property that for each $y \in U$ the germ $\mathcal{S}_y$
is equal to the germ at $y$ of the piece of $\mathcal{Z}$ containing $y$. The set $X$
along with the stratification $\mathcal{S}$ is called a \emph{stratified space}.
If, moreover, $X$ is a differentiable space, and for each $x \in X$, the germ
$\mathcal{S}_x$ is that of a smooth submanifold of $X$ as defined in the previous section, then we say $X$ is a
\emph{differentiable stratified space}.

A continuous function $f\co (X, \mathcal{S})\to (Y,\mathcal{R})$ is a
\emph{morphism of stratified spaces} if, for each $x \in X$ with $f(x) = y$,
there are open Hausdorff neighborhoods $U$ of $x$ and $V$ of $y$ with $U \subset f^{-1}(V)$
and decompositions of $U$ and $V$ inducing their respective stratifications
such that for every $z \in U$ contained in the piece $S$ of $U$, there is an open neighborhood
$O$ of $z$ in $U$ such that $f_{|S\cap O}$ maps into the piece of $V$ containing $f(z)$.

If $(X, \mathcal{O}_X, \mathcal{S})$ and $(Y, \mathcal{O}_Y, \mathcal{R})$ are
differentiable stratified spaces, a function $f\co X \to Y$ is a \emph{differentiable
stratified morphism} if it is simultaneously a morphism of differentiable spaces
and a morphism of stratified spaces.
\end{definition}

Note that the definition of a differentiable
stratified space coincides, for a Hausdorff space $X$, with that of a stratified space with
$\mathcal{C}^\infty$ structure defined in \cite[Section 1.3]{PflaumBook};
see also \cite[Section 2]{FarsiPflaumSeaton}.

If $(X, \mathcal{Z})$ is a decomposed space, then the decomposition induces a stratification
by assigning to $x \in X$ the germ at $x$ of the piece containing $x$; two decompositions of
$X$ are \emph{equivalent} if they induce the same stratification.
The \emph{depth} of $x \in X$ with respect to the decomposition $\mathcal{Z}$ is the maximum
$k$ such that $x \in S_0 < S_1 < \cdots < S_k$ for $S_i \in \mathcal{Z}$.

Now, let $(X, \mathcal{S})$ be a stratified space and let $x \in X$. The proof of
\cite[Lemma 2.1]{MatherStratMap} (see also \cite[Lem.~1.5.2]{PflaumBook}) is local,
hence can be executed on a Hausdorff neighborhood of $x$. It therefore extends to our
case and demonstrates that the depth of $x$ coincides for any decomposition of a Hausdorff
neighborhood of $x$ inducing $\mathcal{S}$. Hence, we may define the \emph{depth of $x$
with respect to $\mathcal{S}$} to be the depth with respect to any such decomposition.
In the same way, the proofs of \cite[Lem.~2.2]{MatherStratMap} and
\cite[Prop.~1.2.7]{PflaumBook}) extend to the situation of locally Hausdorff stratified space.
So $X$ admits a decomposition $\mathcal{Z}$ that induces $\mathcal{S}$ and is maximal in the
sense that for every Hausdorff open subset $U$ of $X$, the restriction of $\mathcal{Z}$ to $U$
is coarser than any decomposition of $U$ that induces $\mathcal{S}$. We will often refer to
$\mathcal{Z}$ simply as the \emph{maximal decomposition of $X$}.
Its pieces are called the \emph{strata} of $X$.
Note that if $X$ is a (Hausdorff) differentiable stratified space, the strata of $\mathcal{Z}$
are obviously (Hausdorff) smooth manifolds.

We recall the following from \cite[Section 1.4.1]{PflaumBook}.

\begin{definition}
\label{def:TopLocTriv}
A stratified space $(X, \mathcal{S})$ is \emph{topologically locally trivial} if for every
$x \in X$ in the stratum $S$ of $X$, there is a neighborhood $U$, a stratified space
$(F, \mathcal{S}^F)$, a point $o \in F$, and an isomorphism of stratified spaces
$h \co U \to (S \cap U) \times F$ such that $h^{-1}(y,o) = y$ for all $y \in S\cap U$,
and such that $\mathcal{S}_o^F$ is the germ of the set $\{ o \}$.
\end{definition}

See \cite[Section 1.4.3]{PflaumBook} for the definitions of Whitney (a)- and Whitney (b)-regularity,
which impose restrictions on how the tangent spaces of strata limit to one another. Note in
particular that Whitney (b)-regularity implies topological local triviality; see
\cite[Corollary 3.9.3]{PflaumBook}.

\subsection{Fibered products}
By \cite[Theorem 7.6]{NGonzalezSanchoBook}, the fibered product of differentiable spaces
has a unique differentiable space structure with respect to which the projection maps
are morphism of differentiable spaces. We now demonstrate that the same holds true for
differentiable stratified spaces. Let $X$, $Y$, and $Z$ be differentiable stratified
spaces with respective stratifications
$\mathcal{S}^X$, $\mathcal{S}^Y$, and $\mathcal{S}^Z$. Suppose
$f\co X\to Z$ and $g\co Y \to Z$ are differentiable stratified mappings.
If $f$ is in addition a stratified submersion, then we define a stratification
of the fibered product $X \fgtimes{f}{g} Y$ as follows. Let
$(x,y) \in X \fgtimes{f}{g} Y$, let $P \subset X$ be a subset whose germ
$[P]_x = \mathcal{S}_x^X$, and let $R \subset Y$ such that
$[R]_y = \mathcal{S}_y^Y$. Then we assign to $(x,y) \in X \fgtimes{f}{g} Y$
the germ $\mathcal{S}_{(x,y)} := [P \fgtimes{f}{g} R]_{(x,y)}$. We refer to
$\mathcal{S}$ as the \emph{induced stratification} of
$X \fgtimes{f}{g} Y$ by the stratifications $\mathcal{S}^X$ and $\mathcal{S}^Y$.

\begin{lemma}
\label{lem:InducedStrat}
Suppose $X$, $Y$, and $Z$ are differentiable stratified spaces
and $f\co X\to Z$ and $g\co Y \to Z$ are differentiable stratified mappings.
If $f$ is in addition a stratified submersion, then the induced stratification
$\mathcal{S}$ is a stratification of $X \fgtimes{f}{g} Y$.
\end{lemma}
\begin{proof}
For simplicity, we work with the maximal decompositions $\mathcal{Z}^X$, $\mathcal{Z}^Y$,
and $\mathcal{Z}^Z$ of $X$, $Y$, and $Z$, respectively,
which is clearly sufficient as the definition of the fibered product is local. Then the fact
that $\mathcal{Z}^X$ and $\mathcal{Z}^Y$ are partitions of $X$ and $Y$ immediately implies
that $\mathcal{Z}:=\{ P \fgtimes{f}{g} R \mid P \in\mathcal{Z}^X, R \in\mathcal{Z}^Y \}$ is
a partition of $X \fgtimes{f}{g} Y$. Moreover, the fact that each $P \in\mathcal{Z}^X$ and $R\in\mathcal{Z}^Y$
is locally closed implies that $P \times R$ is locally closed in $X\times Y$ and hence
$P \fgtimes{f}{g} R = (P\times R)\cap (X \fgtimes{f}{g} Y)$ is locally closed
in $X \fgtimes{f}{g} Y$. Given $(x,y)\in X \fgtimes{f}{g} Y$, let $U_x$ and $U_y$
be open neighborhoods of $x$ in $X$ and $y$ in $Y$, respectively, such that each intersects
finitely many elements of $\mathcal{Z}^X$ and $\mathcal{Z}^Y$,
and then $(U_x\times U_y)\cap (X \fgtimes{f}{g} Y)$ is an open neighborhood of $(x,y)$
in $X \fgtimes{f}{g} Y$ that meets finitely many elements of $\mathcal{Z}$.
Therefore, $\mathcal{Z}$ is a locally finite partition of $X \fgtimes{f}{g} Y$ into
locally closed sets.

Now, let $P\in\mathcal{Z}^X$ and $R\in\mathcal{Z}^Y$, and choose connected components
$P_0$ of $P$ and $R_0$ of $R$. Then as $f$ and $g$ are stratified mappings, there is a
piece $S\in\mathcal{Z}^Z$ with $f(P_0), g(R_0) \subset S$; see \cite[1.2.10]{PflaumBook}.
Moreover, as $f$ is a stratified submersion, $f_{|P_0}$ is by definition a submersion.
Then \cite[Prop.~2.5 \& 2.6]{LangDiffManifolds} and the fact that $f_{|P_0}$
is a submersion imply that $f_{|P_0}$ is transversal to $g_{|R_0}$. Hence,
$P_0 \fgtimes{f}{g} R_0$ is a smooth submanifold of $P_0 \times R_0$
and hence of the differentiable space $X \fgtimes{f}{g} Y$. Therefore each connected component
of $P \fgtimes{f}{g} R$ is a smooth manifold.

Finally, suppose
$(P \fgtimes{f}{g} R)\cap\overline{(P^\prime \fgtimes{f}{g} R^\prime)}\neq\emptyset$
for $P, P^\prime\in\mathcal{Z}^X$ and $R, R^\prime\in\mathcal{Z}^Y$.
Choose $(x,y) \in (P \fgtimes{f}{g} R)\cap \overline{(P^\prime \fgtimes{f}{g} R^\prime)}$.
Then for any open neighborhoods $U_x$ and $U_y$ of $x$ and $y$ in $X$ and $Y$, respectively,
$U_x\times Y_y$ intersects $P^\prime \fgtimes{f}{g} R^\prime$. It follows that
$P\cap\overline{P^\prime}, R\cap\overline{R^\prime}\neq \emptyset$ so that
$P\subset\overline{P^\prime}$ and $R\subset\overline{R^\prime}$, which imply
$P \fgtimes{f}{g} R \subset\overline{(P^\prime \fgtimes{f}{g} R^\prime)}$.
That is, $\mathcal{Z}$ satisfies the condition of frontier and hence is a
decomposition of $X \fgtimes{f}{g} Y$.
\end{proof}

\subsection{Tangent space}
\label{SectionTangentSpace}
  Assume that $(X,\calC^\infty)$ is a differentiable space. Then, given a point $x\in X$,
  the maximal ideal $\mfrak_x \subset \calC^\infty_x$ in the stalk at $x$ is finitely generated,
  which implies that the quotient space
  $\mfrak_x/\mfrak_x^2$ is a finite dimensional real vector space. We will call this space
  the \emph{Zariski cotangent space} $\zartan_x^* X$ of $(X,\calC^\infty)$ at $x$,
  and  its dual $(\mfrak_x/\mfrak_x^2)^*$ the \emph{Zariski tangent space} $\zartan_x X$.


\begin{remark}
  There is another notion of a tangent bundle for a differentiable stratified space
  $(X,\calC^\infty)$, namely the \emph{stratified tangent space} $(T^\textup{st}X,\calC^\infty)$.
  If  $(X,\calC^\infty)$ is Whitney (a)-regular, then  $(T^\textup{st}X,\calC^\infty)$ is
  also a differentiable stratified space. Indeed, $T^\textup{st}X \subset TX$
  with equality if and only if $(X,\calC^\infty)$ is a smooth manifold (without boundary).
  See \cite[Sec.~2.1]{PflaumBook} for more details on stratified tangent bundles
  and  \cite[Rem.~2.2.4]{PflaumBook} for  details on the relation between
  $T^\textup{st}X$ and $TX$.
\end{remark}

\subsection{Differential forms}
\label{SectionDiffForms}
Let $(X,\calC^\infty)$ be a reduced differentiable stratified space. Let  $U$ be an affine open subset
of $X$ and $\iota: U \hookrightarrow \R^n$ be a singular chart of $X$. Denote by $\calI_\iota$
the sheaf of smooth functions vanishing on $\overline{\iota(U)}$. Then define
the sheaf $\Omega^k_\iota$ for $k=0$ as $ \iota^{-1} (\calC^\infty / \calI_\iota) \cong \calC^\infty_{|U}$
and for $k \in \N^*$ as the following inverse image sheaf
\[
  \Omega^k_\iota  := \iota^{-1}\big( \Omega^k_{\R^n} / (\calI_\iota \Omega^k_{\R^n}  +
  d \calI_\iota \wedge \Omega^{k-1}_{\R^n}) \big).
\]
Observe that by construction the exterior differential factors
through the $ \Omega^k_\iota$, hence we obtain a differential graded algebra $\big( \Omega^\bullet_\iota , d \big)$.
If $\kappa : V \hookrightarrow \R^m$ is another singular chart of  $X$,
there exists a unique sheaf isomorphism
$\eta_{\iota,\kappa}: \Omega^\bullet_{|\kappa(U\cap V)} \to \Omega^\bullet_{|\iota(U\cap V)}$ extending the
isomorphism of sheaves
$\eta_{\iota,\kappa} : (\calC^\infty / \calI_\kappa)_{|\kappa(U\cap V)} \to
(\calC^\infty / \calI_\iota)_{|\iota(U\cap V)}$.
We conclude that the cocycle condition
\begin{equation}
  \label{eq:CocycleCondition}
  \eta_{\kappa , \iota} =
  \eta_{\kappa , \lambda} \circ   \eta_{\lambda,\iota}
\end{equation}
holds, where we denote by $\lambda : V \hookrightarrow \R^l$ a third singular chart of  $X$.
Hence the sheaves $ \Omega^k_\iota$ glue to a globally defined sheaf $\Omega_X^k$ of
so-called \emph{abstract $k$-forms} on $X$ in such a way that the
gluing maps preserve $d$. So we obtain a sheaf complex
$\big( \Omega^\bullet_X,d\big)$ of differential   graded algebras.
The complex of global sections
$\big( \Omega^\bullet (X),d\big)$ will be called the \emph{Grauert--Grothendieck complex} of $X$.

\begin{remark}
  For $X \subset \C^n$ a complex space, the construction of the complex $\Omega^\bullet (X)$
  within  the analytic category  goes back to Grauert \cite{GraKerDSKR} and Grothendieck \cite{GroRCAV}.
\end{remark}

Let us now describe how one can represent elements of $\Omega^k (X)$.
To this end assume to be given an open covering $\mathcal{U}$ of $X$ by coordinate domains and a family
$(\kappa_U)_{U\in\mathcal{U}}$ of singular charts
$\kappa_U : U \hookrightarrow \widetilde{U} \subset \R^{n_U}$ such that $\widetilde{U}$ is open and
contains  $\kappa_U (U) $ as a relatively closed subset.
An element of $\Omega^k (X)$ can then be represented as a family
$([\omega_U])_{U \in \mathcal{U}}$, where $\omega_U \in \Omega^k(\widetilde{U})$
and where, for any two $U,V \in \mathcal{U}$, one has
\[
  \eta_{\kappa_V,\kappa_U} \big( [\omega_U] \big) = [\omega_V]
\]
on $U\cap V$.
In the above, $ [\omega_U]$ denotes the equivalence class of $\omega_U$ in $\Omega^k_{\kappa_U}$.

The above formula allows us to extend the forms $[\omega_U]$ globally as follows.
Assume that $S$ is a stratum of $X$,
and let $\iota_S : S \hookrightarrow X$ denote the canonical embedding. Given an element
$\omega = ([\omega_U])_{U \in \mathcal{U}} \in \Omega^k (X) $ one observes that for any two
$U, V \in \mathcal{U}$ the forms  $\iota^*_{S\cap U} \kappa_U^* (\omega_U)$
and  $\iota^*_{S\cap V} \kappa_V^* (\omega_V)$ coincide on the overlap $U\cap V$, hence glue together to an abstract
global form on $S$ which we denote by $\iota^*_S \omega \in \Omega^k (S)$.

By construction, each of the sheaves $\Omega^k_X$ carries the structure of a $\calC^\infty$-module in a natural way.
Therefore, we have the following.
\begin{proposition}
\label{prop:GrauertGrothendieck}
  The Grauert--Grothendieck complex $\big( \Omega^\bullet (X),d\big)$ of a differentiable stratified space $(X,\calC^\infty)$
  is a complex of fine sheaves.
\end{proposition}

Finally, we will define the pull-back morphism $f^* : \Omega^k_Y \to \Omega^k_X$ associated to a smooth
map $f:X\to Y$ between reduced differentiable stratified spaces $(X,\calC^\infty_X)$ and $(Y,\calC^\infty_Y)$.
By Proposition \ref{prop:GrauertGrothendieck} and the construction of the Grauert--Grothendieck complex it suffices to consider the
case where $X\subset \R^n$ and $Y\subset \R^m$ are affine.
Choose open neighborhoods $U \subset \R^n$ of $X$ and  $V \subset \R^m$ of $Y$ such that $X$ is closed in $U$ and
$Y$ in $V$. Choose a smooth function $F:U \to V$ such that $F_{|X}=f$. For $\omega \in \Omega^k (V)$ representing an
abstract $k$-form on $Y$ we put
\[
  f^* ([\omega] ) := [F^*\omega] \in \Omega^k(X).
\]
Since $F^*$ maps the vanishing ideal $I_Y \subset \calC^\infty (V)$ to the vanishing ideal $I_X \subset \calC^\infty (U)$
and since $F^*$ commutes with $d$, $F^*$ maps the $I_Y \Omega^k(V) + dI_Y \wedge \Omega^k(V)$ to
$I_X \Omega^k(U) + dI_X \wedge \Omega^k(U)$. Moreover, if $\widetilde{F}:U \to V$ is another smooth function such that
$\widetilde{F}_{|X}=f$, then $F^* g- \widetilde{F}^*g \in I_X $ and
$F^* dg- \widetilde{F}^*dg \in dI_X $ for all $g \in \calC^\infty (V)$, which entails that
$F^* \omega- \widetilde{F}^*\omega \in I_X \Omega^k(U) + dI_X \wedge \Omega^k(U)$. This proves that
$[F^*\omega]$ neither depends on the particular choice of the representative of $[\omega]$ nor on the particular smooth $F$
extending $f$ to an open neighborhood of $X$. Hence  $f^* : \Omega^k(Y) \to \Omega^k (X)$ is well-defined.
Obviously, $d$ commutes with $f^*$, since it commutes with $F^*$.


\bibliographystyle{amsalpha}
\bibliography{FarPflSea}

\providecommand{\bysame}{\leavevmode\hbox to3em{\hrulefill}\thinspace}
\providecommand{\MR}{\relax\ifhmode\unskip\space\fi MR }
\providecommand{\MRhref}[2]{%
  \href{http://www.ams.org/mathscinet-getitem?mr=#1}{#2}
}
\providecommand{\href}[2]{#2}
\begin{thebibliography}{RPSAW05}

\bibitem[AC07]{AdemCohenCommutElts}
A.~Adem and F.R. Cohen, \emph{Commuting elements and spaces of homomorphisms},
  Math. Ann. \textbf{338} (2007), no.~3, 587--626.

\bibitem[AG12]{AdemGomezEqKThry}
A.~Adem and J.~M. G{\'o}mez, \emph{Equivariant {$K$}-theory of compact {L}ie
  group actions with maximal rank isotropy}, J. Topol. \textbf{5} (2012),
  no.~2, 431--457.

\bibitem[ALR07]{AdemLeidaRuan}
A.~Adem, J.~Leida, and Y.~Ruan, \emph{Orbifolds and stringy topology},
  Cambridge Tracts in Mathematics, vol. 171, Cambridge University Press,
  Cambridge, 2007.

\bibitem[AS09]{AndroulidakisSkandalisFoliat}
I.~Androulidakis and G.~Skandalis, \emph{The holonomy groupoid of a singular
  foliation}, J. Reine Angew. Math. \textbf{626} (2009), 1--37.

\bibitem[BCR98]{BocCosRoyRAG}
Jacek Bochnak, Michel Coste, and Marie-Fran\c{c}oise Roy, \emph{Real algebraic
  geometry}, Ergebnisse der Mathematik und ihrer Grenzgebiete (3) [Results in
  Mathematics and Related Areas (3)], vol.~36, Springer-Verlag, Berlin, 1998,
  Translated from the 1987 French original, Revised by the authors.
  \MR{1659509}

\bibitem[Bie80]{BierstoneOrbitSpace}
E.~Bierstone, \emph{The structure of orbit spaces and the singularities of
  equivariant mappings}, Monograf\'\i as de Matem\'atica [Mathematical
  Monographs], vol.~35, Instituto de Matem\'atica Pura e Aplicada, Rio de
  Janeiro, 1980.

\bibitem[BM93]{BoualemMolino}
H.~Boualem and P.~Molino, \emph{Mod\`eles locaux satur\'es de feuilletages
  riemanniens singuliers}, C. R. Acad. Sci. Paris S\'er. I Math. \textbf{316}
  (1993), no.~9, 913--916.

\bibitem[Bou98]{BouGTC1-4}
N.~Bourbaki, \emph{General topology. {C}hapters 1--4}, Elements of Mathematics
  (Berlin), Springer-Verlag, Berlin, 1998, Translated from the French, Reprint
  of the 1989 English translation.

\bibitem[Bre72]{BredonBook}
G.E. Bredon, \emph{Introduction to compact transformation groups}, Academic
  Press, New York-London, 1972, Pure and Applied Mathematics, Vol. 46.
  \MR{0413144}

\bibitem[Bry87]{BrylinskiCycHomEquivar}
J.-L. Brylinski, \emph{Cyclic homology and equivariant theories}, Ann. Inst.
  Fourier (Grenoble) \textbf{37} (1987), no.~4, 15--28.

\bibitem[BtD95]{BroeckertomDieck}
T.~Br{\"o}cker and T.~tom Dieck, \emph{Representations of compact {L}ie
  groups}, Graduate Texts in Mathematics, vol.~98, Springer-Verlag, New York,
  1995, Translated from the German manuscript, Corrected reprint of the 1985
  translation.

\bibitem[CM18]{CrainicMestreOrbispace}
M.~Crainic and J.N. Mestre, \emph{Orbispaces as differentiable stratified
  spaces}, Lett. Math. Phys. \textbf{108} (2018), no.~3, 805--859.

\bibitem[CPS04]{ChoiParkSuh}
M.-J. Choi, D.H. Park, and D.Y. Suh, \emph{The existence of semialgebraic
  slices and its applications}, J. Korean Math. Soc. \textbf{41} (2004), no.~4,
  629--646.

\bibitem[DCPV11]{DeConcProcesiVergneCohom}
C.~De~Concini, C.~Procesi, and M.~Vergne, \emph{Infinitesimal index: cohomology
  computations}, Transform. Groups \textbf{16} (2011), no.~3, 717--735.

\bibitem[DCPV13]{DeConcProcesiVergne}
\bysame, \emph{The infinitesimal index}, J. Inst. Math. Jussieu \textbf{12}
  (2013), no.~2, 297--334.

\bibitem[Deb01]{DebordHolonomSingFoliat}
C.~Debord, \emph{Holonomy groupoids of singular foliations}, J. Differential
  Geom. \textbf{58} (2001), no.~3, 467--500.

\bibitem[DJZ04]{DomJanZhiRPLCQHVFTSV}
W.~Domitrz, S.~Janeczko, and M.~Zhitomirskii, \emph{Relative {P}oincar\'e
  lemma, contractibility, quasi-homogeneity and vector fields tangent to a
  singular variety}, Illinois J. Math. \textbf{48} (2004), no.~3, 803--835.

\bibitem[DK00]{duistermaatkolk}
J.J. Duistermaat and J.A.C. Kolk, \emph{Lie groups}, Universitext,
  Springer-Verlag, Berlin, 2000.

\bibitem[FOR09]{FernandesOrtegaRatiu}
R.L. Fernandes, J.-P. Ortega, and T.S. Ratiu, \emph{The momentum map in
  {P}oisson geometry}, Amer. J. Math. \textbf{131} (2009), no.~5, 1261--1310.

\bibitem[FPS15]{FarsiPflaumSeaton}
C.~Farsi, M.J. Pflaum, and C.~Seaton, \emph{Stratifications of inertia spaces
  of compact {Lie} group actions}, J. Singul. \textbf{13} (2015), 107--140.

\bibitem[GK64]{GraKerDSKR}
H.~Grauert and H.~Kerner, \emph{Deformationen von {S}ingularit\"aten komplexer
  {R}\"aume}, Math. Ann. \textbf{153} (1964), 236--260.

\bibitem[God58]{GodTATF}
R.~Godement, \emph{Topologie {Alg\'ebrique} et {Theorie} des {Faisceaux}},
  Hermann, Paris, 1958.

\bibitem[Gol84]{Goldman1984}
William~M. Goldman, \emph{The symplectic nature of fundamental groups of
  surfaces}, Adv. in Math. \textbf{54} (1984), no.~2, 200--225.

\bibitem[GPS12]{GomezPettet}
J.M. G{\'o}mez, A.~Pettet, and J.~Souto, \emph{On the fundamental group of
  {${\rm Hom}({\mathbb Z}^k,G)$}}, Math. Z. \textbf{271} (2012), no.~1-2,
  33--44.

\bibitem[Gro66]{GroRCAV}
A.~Grothendieck, \emph{On the {De Rham} cohomology of algebraic varieties},
  Publ. Math. IHES \textbf{29} (1966), 351--359.

\bibitem[Kre10]{KreDAT}
M.~Kreck, \emph{Differential algebraic topology from stratifolds to exotic
  spheres}, Graduate Studies in Mathematics, vol. 110, American Math.~Society,
  2010.

\bibitem[Lan02]{LangDiffManifolds}
S.~Lang, \emph{Introduction to differentiable manifolds}, second ed.,
  Universitext, Springer-Verlag, New York, 2002.

\bibitem[{\L}oj65]{LojESA}
S.~{\L}ojasiewicz, \emph{Ensembles semi-analytiques}, Mimeographi\'{e},
  Institute des Hautes \'{E}tudes Scientifique, Bures-sur-Yvette, France, 1965.

\bibitem[L{\'S}11]{LusalaSnia}
T.~Lusala and J.~{\'S}niatycki, \emph{Stratified subcartesian spaces}, Canad.
  Math. Bull. \textbf{54} (2011), no.~4, 693--705.

\bibitem[LSV13]{VanLeSombergVanzuraConical}
H.V. L{\^e}, P.~Somberg, and J.~Van{\v{z}}ura, \emph{Smooth structures on
  pseudomanifolds with isolated conical singularities}, Acta Math. Vietnam.
  \textbf{38} (2013), no.~1, 33--54.

\bibitem[LT97]{LermanTolman}
E.~Lerman and S.~Tolman, \emph{Hamiltonian torus actions on symplectic
  orbifolds and toric varieties}, Trans. Amer. Math. Soc. \textbf{349} (1997),
  no.~10, 4201--4230.

\bibitem[Mac05]{MackenzieGenThry}
K.C.H. Mackenzie, \emph{General theory of {L}ie groupoids and {L}ie
  algebroids}, London Mathematical Society Lecture Note Series, vol. 213,
  Cambridge University Press, Cambridge, 2005.

\bibitem[Mar08]{Marle}
Charles-Michel Marle, \emph{Calculus on {L}ie algebroids, {L}ie groupoids and
  {P}oisson manifolds}, Dissertationes Math. (Rozprawy Mat.) \textbf{457}
  (2008), 57. \MR{2455155}

\bibitem[Mat73]{MatherStratMap}
J.N. Mather, \emph{Stratifications and mappings}, Dynamical systems ({P}roc.
  {S}ympos., {U}niv. {B}ahia, {S}alvador, 1971), Academic Press, New York,
  1973, pp.~195--232.

\bibitem[Mel91]{MelrosePDO}
R.B. Melrose, \emph{Pseudodifferential operators, corners and singular limits},
  Proceedings of the {I}nternational {C}ongress of {M}athematicians, {V}ol.\
  {I}, {II} ({K}yoto, 1990), Math. Soc. Japan, Tokyo, 1991, pp.~217--234.

\bibitem[Mic96]{michor}
P.W. Michor, \emph{Isometric actions of {L}ie groups and invariants}, lecture
  notes, \href{https://www.mat.univie.ac.at/~michor/tgbook.pdf}, 1996.

\bibitem[Mol88]{MolinoBook}
P.~Molino, \emph{Riemannian foliations}, Progress in Mathematics, vol.~73,
  Birkh\"auser Boston, Inc., Boston, MA, 1988, Translated from the French by
  Grant Cairns, With appendices by Cairns, Y. Carri{\`e}re, {\'E}. Ghys, E.
  Salem and V. Sergiescu.

\bibitem[MROD93]{MargalefOuterelo}
J.~Margalef~Roig and E.~Outerelo~Dom{\'{\i}}nguez, \emph{Lie group actions over
  manifolds with corners}, Math. Japon. \textbf{38} (1993), no.~3, 577--582.

\bibitem[NGSdS03]{NGonzalezSanchoBook}
J.A. Navarro~Gonz{\'a}lez and J.B. Sancho~de Salas,
  \emph{{$C\sp\infty$}-differentiable spaces}, Lecture Notes in Mathematics,
  vol. 1824, Springer-Verlag, Berlin, 2003.

\bibitem[NW14]{NogWinNTSCVDA}
J.~Noguchi and J.~Winkelmann, \emph{Nevanlinna theory in several complex
  variables and {D}iophantine approximation}, Grundlehren der Mathematischen
  Wissenschaften [Fundamental Principles of Mathematical Sciences], vol. 350,
  Springer, Tokyo, 2014.

\bibitem[OR04]{OrtegaRatiu}
J.-P. Ortega and T.S. Ratiu, \emph{Momentum maps and {H}amiltonian reduction},
  Progress in Mathematics, vol. 222, Birkh{\"a}user Boston, Inc., Boston, MA,
  2004.

\bibitem[Par12]{ParadanKGTGM}
P.-{\'E}. Paradan, \emph{On the structure of ${K}_{G}({T}_{G} {M})$},
  \href{http://arxiv.org/abs/1209.3852}{arXiv:1209.3852}, 2012.

\bibitem[Par13]{ParkOrbitTypes}
D.H. Park, \emph{On orbit types of semialgebraically proper actions}, Arch.
  Math. (Basel) \textbf{101} (2013), no.~1, 33--41.

\bibitem[Pfl01]{PflaumBook}
M.J. Pflaum, \emph{Analytic and geometric study of stratified spaces}, Lecture
  Notes in Mathematics, vol. 1768, Springer-Verlag, Berlin, 2001.

\bibitem[PPT07]{PflPosTanAITO}
M.J. Pflaum, H.B. Posthuma, and X.~Tang, \emph{An algebraic index theorem for
  orbifolds}, Adv. Math. \textbf{210} (2007), no.~1, 83--121.

\bibitem[PPT14]{PflPosTanGOSPLG}
\bysame, \emph{Geometry of orbit spaces of proper {Lie} groupoids}, J.~f\"ur
  die Reine und Angewandte Mathematik \textbf{25} (2014), 1135–1153.

\bibitem[PPT17]{PflPosTanGrauGroth}
\bysame, \emph{The {G}rauert-{G}rothendieck complex on differentiable spaces
  and a sheaf complex of {B}rylinski}, Methods Appl. Anal. \textbf{24} (2017),
  no.~2, 321--332.

\bibitem[PPT20]{PflPosTanHochschild}
\bysame, \emph{On the {H}ochschild homology of convolution algebras of proper
  {L}ie groupoids}, to appear in J. of Noncommutative Geometry,
  \href{https://arxiv.org/abs/2009.03216}{arXiv:2009.03216 [math.KT]}.

\bibitem[Pra85]{PradinesSingFoliat}
Jean Pradines, \emph{How to define the differentiable graph of a singular
  foliation}, Cahiers Topologie G\'eom. Diff\'erentielle Cat\'eg. \textbf{26}
  (1985), no.~4, 339--380.

\bibitem[PS02]{ParkSuhLinearEmbed}
D.H. Park and D.Y. Suh, \emph{Linear embeddings of semialgebraic {$G$}-spaces},
  Math. Z. \textbf{242} (2002), no.~4, 725--742.

\bibitem[PV09]{ParadanVergneIndex}
P.-{\'E}. Paradan and M.~Vergne, \emph{Index of transversally elliptic
  operators}, Ast\'erisque (2009), no.~328, 297--338 (2010).

\bibitem[PY07]{PanyushevYakimova}
D.~Panyushev and O.~Yakimova, \emph{Symmetric pairs and associated commuting
  varieties}, Math. Proc. Cambridge Philos. Soc. \textbf{143} (2007), no.~2,
  307--321.

\bibitem[Ren80]{RenGACA}
J.~Renault, \emph{A groupoid approach to {$C^{\ast} $}-algebras}, Lecture Notes
  in Mathematics, vol. 793, Springer, Berlin, 1980.

\bibitem[Ric79]{RichardsonCommutVar}
R.W. Richardson, \emph{Commuting varieties of semisimple {L}ie algebras and
  algebraic groups}, Compositio Math. \textbf{38} (1979), no.~3, 311--327.

\bibitem[RPSAW05]{RoyoPrietoSAWSingFoil}
J.~I. Royo~Prieto, M.~Saralegi-Aranguren, and R.~Wolak, \emph{Top-dimensional
  group of the basic intersection cohomology for singular {R}iemannian
  foliations}, Bull. Pol. Acad. Sci. Math. \textbf{53} (2005), no.~4, 429--440.

\bibitem[Sch80]{SchwarzLiftingHomotopies}
G.W. Schwarz, \emph{Lifting smooth homotopies of orbit spaces}, Inst. Hautes
  \'Etudes Sci. Publ. Math. (1980), no.~51, 37--135.

\bibitem[Sch01]{Schmah}
T.~Schmah, \emph{Torus actions on symplectic orbi-spaces}, Proc. Amer. Math.
  Soc. \textbf{129} (2001), no.~4, 1169--1177.

\bibitem[Shi97]{ShiotaBook}
M.~Shiota, \emph{Geometry of {Subanalytic} and {Semialgebraic} {Sets}},
  Progress in Mathematics, vol. 150, Birkh\"auser Boston, Inc., Boston, MA,
  1997.

\bibitem[Sja05]{SjamaarDeRham}
R.~Sjamaar, \emph{A de {R}ham theorem for symplectic quotients}, Pacific J.
  Math. \textbf{220} (2005), no.~1, 153--166.

\bibitem[SL91]{SjamaarLerman}
R.~Sjamaar and E.~Lerman, \emph{Stratified symplectic spaces and reduction},
  Ann. of Math. (2) \textbf{134} (1991), no.~2, 375--422.

\bibitem[SLV15]{VanLeSombergVanzura}
P.~Somberg, H.V. L{\^e}, and J.~Van{\v{z}}ura, \emph{Poisson smooth structures
  on stratified symplectic spaces}, Mathematics in the 21st Century (Pierre
  Cartier, A.D.R. Choudary, and Michel Waldschmidt, eds.), Springer Proceedings
  in Mathematics \& Statistics, vol.~98, Springer Basel, 2015, pp.~181--204.

\bibitem[Spa91]{SpallekCTG}
K.~Spallek, \emph{Continuous transformation groups on spaces}, Proceedings of
  the {T}enth {C}onference on {A}nalytic {F}unctions ({S}zczyrk, 1990),
  vol.~55, 1991, pp.~301--320. \MR{1141445}

\bibitem[Ste74]{StefanFoliatSing}
P.~Stefan, \emph{Accessible sets, orbits, and foliations with singularities},
  Proc. London Math. Soc. (3) \textbf{29} (1974), 699--713.

\bibitem[Sus73]{SussmannOrbitsVFldsDistrib}
H.~J. Sussmann, \emph{Orbits of families of vector fields and integrability of
  distributions}, Trans. Amer. Math. Soc. \textbf{180} (1973), 171--188.

\bibitem[tD87]{tomDieck}
Tammo tom Dieck, \emph{Transformation groups}, De Gruyter Studies in
  Mathematics, vol.~8, Walter de Gruyter \& Co., Berlin, 1987. \MR{889050}

\bibitem[TGS08]{TorresSjerveCommutNTuple}
E.~Torres~Giese and D.~Sjerve, \emph{Fundamental groups of commuting elements
  in {L}ie groups}, Bull. Lond. Math. Soc. \textbf{40} (2008), no.~1, 65--76.

\bibitem[Tu04]{TuNonHausGpoid}
J.-L. Tu, \emph{Non-{H}ausdorff groupoids, proper actions and {$K$}-theory},
  Doc. Math. \textbf{9} (2004), 565--597 (electronic).

\bibitem[Ver96]{Vergne}
M.~Vergne, \emph{Equivariant index formulas for orbifolds}, Duke Math. J.
  \textbf{82} (1996), no.~3, 637--652.

\bibitem[Wat12]{WattsThesis}
J.~Watts, \emph{Diffeologies, {D}ifferential {S}paces, and {S}ymplectic
  {G}eometry}, ProQuest LLC, Ann Arbor, MI, 2012, Thesis (Ph.D.)--University of
  Toronto (Canada).

\bibitem[Wat15]{WattsOrbifold}
\bysame, \emph{The differential structure of an orbifold},
  \href{http://arxiv.org/abs/1503.01740}{arXiv:1503.01740}, 2015.

\bibitem[Zun06]{Zung06}
N.~Zung, \emph{Proper groupoids and momentum maps: linearization, affinity, and
  convexity}, Ann.~Sci.~\'ecole Norm.~Sup.~(4) \textbf{39} (2006), no.~5,
  841--869.

\end{thebibliography}

\end{document}